\documentclass[12pt]{amsart}
\usepackage{amsmath}
\usepackage{amscd}
\usepackage{amssymb}
\usepackage{amsfonts}
\usepackage{amsthm}
\usepackage{bbm}
\usepackage{cancel}
\usepackage{color}
\usepackage{eucal}
\usepackage{enumerate,yfonts}
\usepackage{enumitem}
\usepackage{pdfsync}
\usepackage[all,cmtip]{xy}
\usepackage{graphicx}
\usepackage{graphics}
\usepackage{hyperref}
\usepackage{latexsym}
\usepackage{mathrsfs}
\usepackage{placeins}
\usepackage{pstricks}
\usepackage{bookmark}
\usepackage{mathtools}
\usepackage{stmaryrd}
\usepackage{url}
\usepackage{mathabx}
\usepackage{float}

\newtheorem{thm}{Theorem}[section]
\newtheorem{theorem}[thm]{Theorem}
\newtheorem{corollary}[thm]{Corollary}
\newtheorem{lemma}[thm]{Lemma}

\newtheorem{prop}[thm]{Proposition}
\newtheorem{conjecture}[thm]{Conjecture}
\newtheorem{thm-dfn}[thm]{Theorem-Definition}

\newtheorem{remark}[thm]{Remark}
\newtheorem{definition}[thm]{Definition}

\numberwithin{equation}{section}

\newenvironment{rouge}
{\relax\color{red}}
{\hspace*{.3ex}\relax}

\newenvironment{bluet}
{\relax\color{blue}}
{\hspace*{.3ex}\relax}

\newcommand{\br}{\begin{rouge}}
\newcommand{\er}{\end{rouge}}

\newcommand{\bb}{\begin{bluet}}
\newcommand{\eb}{\end{bluet}}

\newcommand{\nc}{\newcommand}

\newcommand{\cD}{{\mathcal D}}

\newcommand{\cA}{{\mathcal A}}
\newcommand{\cF}{{\mathcal F}}
\newcommand{\cN}{{\mathcal N}}
\newcommand{\cH}{{\mathcal H}}

\newcommand{\cL}{{\mathcal L}}

\newcommand{\cM}{{\mathcal M}}

\newcommand{\bC}{{\mathbb C}}
\newcommand{\bZ}{{\mathbb Z}}

\newcommand{\bQ}{{\mathbb Q}}

\newcommand{\bN}{{\mathbb N}}

\newcommand{\bG}{{\mathbb G}}

\newcommand{\La}{{\mathfrak{a}}}

\newcommand{\Lt}{{\mathfrak{t}}}
\newcommand{\Lg}{{\mathfrak{g}}}

\newcommand{\Ll}{{\mathfrak{l}}}
\newcommand{\Lp}{{\mathfrak{p}}}

\newcommand{\Ls}{{\mathfrak{s}}}

\newcommand{\ft}{{\mathfrak{t}}}
\newcommand{\fg}{{\mathfrak g}}

\newcommand{\fF}{{\mathfrak{F}}}

\newcommand{\fa}{{\mathfrak{a}}}

\nc{\ot}{\otimes}
\nc{\on}{\operatorname}

\nc{\Lie}{{\operatorname{Lie}}}
\nc{\oh}{{\operatorname{H}}}
\nc{\gr}{{\operatorname{gr}}}
\nc{\rk}{{\operatorname{rank}}}
\nc{\codim}{{\operatorname{codim}}}
\nc{\img}{{\operatorname{Im}}}
\nc{\IC}{{\operatorname{IC}}}

\nc{\bI}{{\mathbf 1}}

\nc{\lp}{{\left(}}
\nc{\rp}{{\right)}}

\newcommand{\beqn}{\begin{equation*}}
\newcommand{\eeqn}{\end{equation*}}

\newcommand{\beq}{\begin{equation}}
\newcommand{\eeq}{\end{equation}}

\newcommand{\bega}{\begin{gathered}}
\newcommand{\eega}{\end{gathered}}

\newcommand{\bern}{\begin{eqnarray*}}
\newcommand{\eern}{\end{eqnarray*}}
\newcommand{\ber}{\begin{eqnarray}}
\newcommand{\eer}{\end{eqnarray}}

\newcommand{\inv}{{\mathbin{/\mkern-4mu/}}}

\newcommand{\baseRing}{\bC}

\newcommand{\braid}{b}

\newcommand{\basepta}{a_0}
\newcommand{\barbasepta}{\bar{a}_0}

\newcommand{\Cartan}{\fa}

\nc{\diag}{{\on{diag}}}  

\nc{\Ts}{{T_{\mathbf{s}}}}
\nc{\Tds}{{\check T_{\mathbf{s}}}}

\begin{document}
 
\title[Character sheaves for graded Lie algebras: Stable Gradings]{Character Sheaves for Graded Lie Algebras: Stable Gradings}

\author{Kari Vilonen}\address{School of Mathematics and Statistics, University of Melbourne, VIC 3010, Australia, also Department of Mathematics and Statistics, University of Helsinki, Helsinki, Finland}
\email{kari.vilonen@unimelb.edu.au, kari.vilonen@helsinki.fi}
\thanks{KV was supported in part by the ARC grants DP150103525 and DP180101445  and the Academy of Finland}

\author{Ting Xue}
\address{School of Mathematics and Statistics, University of Melbourne, VIC 3010, Australia, also Department of Mathematics and Statistics, University of Helsinki, Helsinki, Finland} 
\email{ting.xue@unimelb.edu.au}
\thanks{TX was supported in part by the ARC grant DP150103525.}

\begin{abstract}In this paper we construct full support character sheaves for stably graded Lie algebras. Conjecturally these are precisely the cuspidal character sheaves. Irreducible representations of Hecke algebras associated to complex reflection groups at roots of unity enter the description. We do so by   analysing the Fourier transform of the nearby cycle sheaves constructed in \cite{GVX2}.
\end{abstract}

\date{Dec 15, 2020}

\maketitle

\tableofcontents

\section{Introduction}

In this paper we initiate the study of character sheaves for general $\bZ/m\bZ$-graded Lie algebras. The goal is to determine the cuspidal character sheaves, i.e., the character sheaves which cannot be obtained (as a direct summand) by parabolic induction from smaller graded Lie algebras.

Let us consider a semisimple algebraic group $G$ over complex numbers. Character sheaves were introduced by Lusztig in~\cite{Lu2}. In this paper we consider a generalization of this situation, namely, the case of graded Lie algebras. 
Let $\theta$ be an automorphism of $G$ of finite order $m$. It induces a grading on the Lie algebra
\beqn
\Lg = \bigoplus_{i\in\bZ/m\bZ} \fg_i
\eeqn
and $K=G^\theta$ acts on the constituents of this decomposition.   Graded Lie algebras were initially studied by Vinberg~\cite{Vin}. They arise naturally in various context, for example, from the Moy-Prasad filtrations in the theory of $p$-adic groups. 

For simplicity we will assume that $G$ is simply connected so that $K$ is connected. We can identify $\fg_1^*$ with $\fg_{m-1}=\fg_{-1}$.  In~\cite{MV} character sheaves on a group are characterized by the property that their singular support is nilpotent.  We define them similarly here as irreducible $K$-equivariant  perverse sheaf on $\fg_1$ with nilpotent singular support.  Thus, a character sheaf is the Fourier transform of an irreducible $K$-equivariant  perverse sheaf on $\cN_{-1}=\cN \cap \fg_{-1}$, where we have written $\cN$ for the nilpotent cone in $\fg$. 

Recall that $\Lg_1$ possesses a Cartan subspace $\fa\subset \fg_1$ which is a maximal abelian subspace consisting of semisimple elements. We have the adjoint quotient map
\beqn
f: \Lg_1 \to \Lg_1 \inv K \cong \La \slash W_\La \,,
\eeqn
where $W_\fa=N_K(\fa)/Z_K(\fa)$ is the little Weyl group which is a complex reflection group. In what follows the ``nearby cycle construction" refers to nearby cycles with respect to the entirely analogous adjoint quotient map for $\Lg_{-1}$. By the rank of the graded Lie algebra we mean the dimension of the Cartan subspace $\fa$, and by the ``semisimple" rank we mean $\dim\fa-\dim\fa^K$.

The case of involutions was considered in~\cite{VX1} where we gave a complete classification of character sheaves for classical groups. The case of higher order gradings is significantly more complicated as the Weyl groups arising in this situation are no longer Coxeter groups but rather complex reflection groups. These complexities are already present in the geometric prequel~\cite{GVX2} to this paper. Hence, we will not attempt to give as complete an answer as in~\cite{VX1} in general but will rather describe the cuspidal character sheaves. 

We recall that Lusztig has described the character sheaves in the case when there is no grading, i.e., in the case that $G$ acts on the Lie algebra $\fg$ by the adjoint action in~\cite{Lu3,Lu1}. We call this setting the ungraded case. This case can also be viewed as arising from the involution on $G\times G$ which switches the two factors. The case of rank zero graded Lie algebras was also done by Lusztig in~\cite{Lu4}. In this case the character sheaves themselves also have nilpotent support. Lusztig describes the cuspidal ones in this setting in terms of the cuspidals in the ungraded case. 

For automorphisms of finite order the situation is different as was observed in~\cite{VX1} in the case of involutions. Let us consider $Z(G)^\theta$, where $Z(G)$ is the center of $G$. The $Z(G)^\theta$ acts on a character sheaf via a character. We expect the cuspidal character sheaves to have as large a support as allowed by the central character. Note that the character sheaves can be broken into ``blocks" with respect to this action. 

We further expect that the cuspidal character sheaves in the case of finite order $\theta$ come from nearby cycle constructions as explained in~\cite{GVX1,GVX2} and a generalization of these constructions. In this manner we obtain character sheaves associated to irreducible representations of Hecke algebras of (complex) reflection groups with specific parameters. Note that only certain gradings afford cuspidal character sheaves. In this paper we consider stable gradings and, still conjecturally, produce all cuspidal character sheaves in this context. In the sequel~\cite{VX2} we consider certain very special unstable gradings. 
We will make use of a generalization of the nearby cycle construction of~\cite{GVX2} which starts with, instead of a local system, an intersection homology sheaf which is constructed from a  character sheaf arising from rank zero gradings. In this context we also obtain cuspidal character sheaves associated to irreducible representations of Hecke algebras of (complex) reflection groups with specific parameters as above. In addition, the central character can impose a very strong condition on the support of the character sheaf. For example, it can force the support to be nilpotent. Such character sheaves are isolated and can be studied via their central character. 

So far we have explained constructions that produce character sheaves. To show that we have obtained all cuspidal character sheaves we will eventually relate our study to double affine Hecke algebras (DAHA's) in~\cite{VX3}. The idea is the following. In a series of papers~\cite{LY1,LY2,LY3} Lusztig and Yun relate the irreducible $K$-equivariant  perverse sheaves on $\cN_{-1}$ to representations of certain trigonometric DAHA's. On the other hand the character sheaves can be related, via the KZ functor, to representations of certain rational DAHA's. In~\cite{VX3} we will explain, at the moment up to some conjectures, how we have, in the context of cuspidal character sheaves, a bijection between the (full support) irreducible representations of the rational DAHA's arising from our context and the finite dimensional irreducible representations of the trigonometric DAHA's of Lusztig-Yun.

As we already pointed out, in this paper we consider the case of stable gradings. A grading on $\Lg$ is stable if $\fg_1$ is $K$-stable in the sense of invariant theory so that, in particular, $I=Z_K(\fa)$ is finite. The pair $(\fg_1,K$) constitutes a {\it polar} representation and we apply the main result of~\cite{GVX2} to this situation. To a character $\chi$ of $I$ we associate a local system $\cL_\chi$ on a regular semisimple $K$-orbit in $\fg_{-1}$. Via a nearby cycle construction we take the limit of $\cL_\chi$ to the nilpotent cone and thus obtain a perverse sheaf $P_\chi\in \on{Perv}_{K}(\cN_{-1})$. Let us write $\fF$  for the Fourier transform. Following~\cite{GVX2} we have: 
\beqn
\fF P_\chi \ = \ \on{IC}(\fg_1,\cM_\chi)\,,
\eeqn
an intersection homology sheaf on $\fg_1$ associated to a local system $\cM_\chi$. Note that we can also think of $\cM_\chi$ as the local system obtained by microlocalizing $P_\chi$ to the conormal bundle at the origin. 

In~\cite{GVX2} we prove a general theorem which allows one to determine $\cM_\chi$ by reducing it to (semisimple) rank one calculations. Let us note that a priori $\cM_\chi$ is a representation of a semidirect product $I \rtimes B_{W_\fa}$ where $B_{W_\fa}$ is the braid group of the little Weyl group $W_\fa$. Obtaining such a semidirect product decomposition depends on a choice of a Kostant slice which, in the terminology of~\cite{GVX2} provides us with a regular splitting. The main point of~\cite{GVX2} is to show that the action of $B_{W_\fa}$ is obtained  by induction from a Hecke algebra  $\cH_{W_{\fa,\chi}^0}$ of a particular subgroup $W_{\fa,\chi}^0$ of the stabilizer group $W_{\fa,\chi}=\on{Stab}_{W_\fa} (\chi)$. The reason we have to pass to $W_{\fa,\chi}^0$ is that the group $W_{\fa,\chi}$ is not necessarily a complex reflection group. The subgroup ${W_{\fa,\chi}^0}$ is the maximal subgroup which is a complex reflection subgroup of $W_{\fa,\chi}$. 

The construction of the Hecke algebra proceeds as follows. Each (distinguished) primitive reflection $s\in W_\fa$ gives rise to a (stable) grading of semisimple rank one and a nearby cycle sheaf $P_{\chi_s}$. In this case the braid group is just $\bZ$ and the microlocal monodromy has a minimal polynomial $R_{\chi,s}$ whose degree we denote by $n_s$. The braid generator $\sigma_s$ (which gives rise to $1\in \bZ$) of the braid group does not necessarily lie in  $B_{W_{\fa,\chi}^0}$. We denote by $e_s$ the smallest integer such that $s^{e_s}\in W_{\fa,\chi}$. We show that there is a polynomial $\bar R_{\chi,s}(x)$ so that
$R_{\chi,s}(x)=\bar R_{\chi,s}(x^{e_s})$. We then set 
\beqn
\cH_{W_{\fa,\chi}^0} \ = \ \bC[B_{W_{\fa,\chi}^0} ] / \langle  \bar R_{\chi,s}(\sigma_s^{e_s})\mid \text{$s$ a (distinguished) primitive reflection}   \rangle.
\eeqn
In this paper we determine the polynomials $R_{\chi,s}(x)$ explicitly in the case of stable gradings. We do so with a few exceptions which arise from exceptional groups. At the end of the paper we write down these Hecke algebras completely explicitly in classical types. These explicit calculations have served as an invaluable guide during this project. For arbitrary gradings of rank one the invariant polynomial (there is only one invariant because we are in rank one!) appears to be very complicated. Thus, at least at first sight, it does not seem feasible to try to compute the polynomial $R_{\chi,s}(x)$ explicitly in general. However, in the stable grading case the situation is quite manageable as the nilpotent cone is given by a normal crossings divisor (except in some exceptional cases). The case of the special unstable gradings that we plan to treat in~\cite{VX2} is manageable for similar reasons. It is striking that in order to describe the cuspidal character sheaves (and hence all the character sheaves) we are precisely in a situation allowing an explicit calculation. 

As in~\cite{VX1} the full support constituents of the sheaf $\on{IC}(\Lg_1,\cM_\chi)$ are character sheaves. We conjecture that these are precisely the cuspidal character sheaves in our setting.

We have also been guided by the endoscopic point of view. The character $\chi$ of $I$ gives rise to an element in the dual torus $\check T$ of a split torus $T$ of $G$. Thus we can consider the stabilizer group $\check G(\chi)$ of the dual group $\check G$. The grading $\theta$ induces a (stable) grading $\check \theta$ of $\check G$. It is not too difficult to see that $W_{\fa,\chi} = W(\check G(\chi)^{\check\theta},\check T)$. We now set
$W_{\fa,\chi}^{en}= W((\check G(\chi)^{\check\theta})^0,\check T)$. As a little Weyl group  $W_{\fa,\chi}^{en}$ is a complex reflection group and we show that $W_{\fa,\chi}^{en}\subset W_{\fa,\chi}^0$. In the description of $\cM_\chi$ we can replace $W_{\fa,\chi}^0$ by $W_{\fa,\chi}^{en}$. We claim that $W_{\fa,\chi}^{en}$ is a minimal complex reflection subgroup of $W_{\fa,\chi}$ with this property. For a distinguished reflection $s\in W_\fa$, let us write $d_s$ for the smallest integer such that $s^{d_s}\in W_{\fa,\chi}^{en}$. Then $e_s | d_s$. We further claim that there are polynomials $\bar{\bar{R}}_{\chi,s}$ such that $R_{\chi,s}(x)=\bar{\bar{R}}_{\chi,s}(x^{d_s})$ and that $d_s$ is the largest power that can be extracted in this manner from $R_{\chi,s}(x)$. These facts follow for classical (and some other) groups by explicit case-by-case calculation. Furthermore, we expect that we are in an endoscopic situation, i.e., that the Hecke algebra $\cH_{W_{\fa,\chi}^{en}}$ associated to the triple $(G,\theta, \chi)$ is the Hecke algebra associated to $(\check G(\chi)^0,\check\theta,\on{triv})$.

The paper is organized as follows. In section 2 we recall some results about stable gradings and recall the main theorem of~\cite{GVX2}. In section 3 we discuss how the case of general reductive $G$ can be reduced to the case of an almost simple simply connected $G$. This is important as even if we start with  an almost simple simply connected $G$ in the rank one reduction we first obtain a reductive group. Section 4 contains a general treatment of the stable rank one situation. In section 5 we present an endoscopic point of view which has guided us throughout our project. In section 6 we determine the polynomials $R_{\chi,s}$ explicitly and finally, in section 7, we give a rather explicit description of cuspidal character sheaves for stable gradings in classical types. That these are all cuspidals is still conjectural.

{\bf Acknowledgements.} We thank the Research Institute for Mathematical Sciences, Kyoto University, Japan, for hospitality, support, and a nice research environment. We thank Misha Grinberg, Cheng-Chiang Tsai and Zhiwei Yun for many inspiring and helpful discussions. T.X. thanks George Lusztig for his continuous encouragement and his interest in this work.

\section{Preliminaries}
\label{sec-preliminaries}

We work throughout over the complex numbers. However, we expect that our statements can be extended to other situations with some restrictions on the characteristic of the field. For perverse sheaves we use the conventions of~\cite{BBD}.

\subsection{Stable Gradings}\label{ssec-stable gradings}
Let $G$ be a connected reductive algebraic group over the complex numbers and let $\theta: G \to G$ be a finite order automorphism of $G$ of order $m$. We write: 
\beqn
G^\theta = K.
\eeqn
Let $\Lg = \on{Lie} G$. Then we have the following $\bZ/m\bZ$-grading on $\Lg$ induced by $\theta$: 
\beqn
\Lg = \bigoplus_{i\in\bZ/m\bZ} \fg_i
\eeqn
where $\Lg_i$ is the $\zeta_m^i$-eigenspace of $d\theta$ for a fixed primitive $m$-th root of unity $\zeta_m=e^{2\pi\mathbf{i}/m}$, $\mathbf{i}=\sqrt{-1}$. 
We say that $\theta$ or the corresponding $\bZ/m\bZ$-grading is {\it stable} if the $K$-action on $\fg_1$ is stable in the sense of invariant theory, i.e., there exists an element in $\Lg_1$ whose $K$-orbit is closed and whose stabilizer in $K$ is finite. When the Lie algebra $\Lg$ is simple the stable gradings have been classified in \cite{RLYG}. For an explicit description of stable gradings for classical types, see \cite[\S6-8]{Y}.

We fix a Cartan subspace $\La$ of $\Lg_1$; i.e., $\La \subset \Lg_1$ is a maximal abelian subspace consisting of semisimple elements. 
An element $x \in \Lg_1$ is called regular if $\dim Z_K (x) \leq \dim Z_K (x')$ for all $x' \in \Lg_1$. 
We write $\Lg_1^{rs}$ for the set of regular semisimple elements of $\Lg_1$, and we let: 
\beqn
\La^{rs} = \La \cap \Lg_1^{rs}.
\eeqn
When the grading is stable, then there are elements in  $\fa$ which are regular semisimple elements of $\Lg$, i.e., $\La^{rs}=\La\cap\Lg^{rs}$. 

From now on we assume that $\theta$ induces a stable grading on $\Lg$. In section~\ref{Isogenies} we will explain how the case of a reductive $G$ can be reduced to the case of a simply connected $G$. Thus, in the main body of the text we will work under this hypothesis. This ensures, in particular, that $K$ is connected (\cite{St2}). It is also explained in section~\ref{Isogenies} that we can limit ourselves to the case of an almost simple $G$. We will impose this condition when appropriate. 

We will also make use of the adjoint form $G_{ad}$ of $G$. Let us recall that we have
\beqn
\on{Aut}(G) \ = \ \on{Aut}(G_{ad}) \ = \  \on{Aut}(\fg)
\eeqn
so that we can think of $\theta$ as an element in any of these groups. In~\cite{RLYG} the authors work with the adjoint group but all their results can in this manner be directly translated to the simply connected group $G$. Recall that we have 
\beq
\label{auto}
1 \to G_{ad} \to \on{Aut}(G) \to \on{Out}(\fg)\to 1\,,
\eeq
where $\on{Out}(\fg)$ stands for the outer automorphisms of $\fg$ which are precisely the automorphisms of the Dynkin diagram of $\fg$. The exact sequence above can be split by choosing a pinning. We recall this construction. 

Consider a maximal torus $T\subset G$ and a Borel subgroup $B\supset T$. Let $\Lt=\on{Lie }T$. We write $X^* (T) = \on{Hom} (T, \bG_m)$ and $X_* (T) = \on{Hom} (\bG_m, T)$. Let $ \Phi=\Phi(T,G)$ denote the set of roots. 
Let $ W(G, T) = N_G(T) / T$ denote the  Weyl group. Then $W(G,T)$ is the Coxeter group associated to the root system $\Phi \subset  X^* (T)_\bC \cong \Lt^*$.
Let $\Delta\subset \Phi$ be the set of simple roots determined by $B$. For each $\alpha\in\Delta$, we fix an isomorphism,
$
x_\alpha:\bG_a\to U_\alpha, 
$
such that $tx_\alpha(c)t^{-1}=x_\alpha(\alpha(t)c)$ for all $c\in\bG_a$ and $t\in T$, where $U_\alpha$ is the unique closed subgroup of $G$ such that such $x_\alpha$ exists. Let 
\beqn
\text{$X_{\alpha}=dx_\alpha(1)$, $\alpha\in\Delta$.}
\eeqn
Once we have fixed the pinning $(B,T,\{X_{\alpha}\}_{\alpha\in\Delta})$, we can consider pinned automorphisms of $G$, i.e., the automorphisms which preserve the datum $(B,T,\{X_{\alpha}\}_{\alpha\in\Delta})$. It is easy to see that pinned automorphisms of $G$ can naturally be identified with $\on{Out}(\fg)$ and in this manner we obtain a splitting of~\eqref{auto} and we can write $\on{Aut}(G) = G_{ad} \rtimes   \on{Out}(\fg)$. Note that any two pinnings of $G$ are conjugate to each other by a unique element in $G_{ad}$. 

Similarly, $\on{Aut}(\Phi) = W(G,T) \rtimes   \on{Out}(\fg)$. If the torus $T$ is stable under an automorphism  $\theta\in\on{Aut}(G)$, then $\theta$ gives rise to an element $w \vartheta\in\on{Aut}(\Phi)$, where $\vartheta\in  \on{Out}(\fg)$, and then 
\beq\label{theta-new}
\theta = \on{Int}(n_w)\circ \vartheta
\eeq
 for some lift $n_w\in N_G(T)$ of $w$, where $\on{Int}(n_w)$ denotes conjugation by $n_w$. 

Let us consider the split torus
\beq
\label{split torus}
\Ts=Z_G (\La).
\eeq
It  is, by construction, a maximal torus of $G$ which is $\theta$-stable. Thus, choosing a pinning with $\Ts$ as the maximal torus the element $\theta$ gives rise to an element $w\vartheta\in W(G,\Ts) \rtimes  \on{Out}(\fg)$. It is shown in~\cite[section 4]{RLYG} that $\theta$ is stable if and only if the element $w\vartheta$ is elliptic regular. Furthermore, the $W$-conjugacy class of $w\vartheta$ determines $\theta$ uniquely  (up to conjugacy) and $m$, the order of $\theta$, is the order of the element $w\vartheta$ in $W(G,\Ts) \rtimes  \on{Out}(\fg)$.

We will sometimes make use of another pinning with respect to a fundamental torus $T_{\mathbf{f}}$. Following the convention in real groups we call a maximal torus $T$ {\it fundamental} if it is stable under $\theta$  and  $(T^\theta)^0$ is a maximal torus in $K$. It follows from the results of~\cite[section 4]{RLYG} that the automorphism  $\theta$  inducing a stable grading takes the following form with respect to a pinning based on the fundamental torus: \beq\label{stable gradings}
\theta =\on{Int}(\check\rho_{ad}(\zeta_m))\circ\vartheta_{\mathbf{f}},
\eeq
where $\check\rho_{ad}$ is the half sum of positive coroots for $G_{ad}$,  $\vartheta_{\mathbf{f}}$ is a pinned automorphism of $G$ with respect to $T_{\mathbf{f}}$. Conversely, if $\theta$ is as in~\eqref{stable gradings} with respect to some pinning then $T$ is a fundamental torus and $(T^\theta)^0 = (T^\vartheta)^0$ is a maximal torus in $K$.

Let us write $W_\La = W(K,\La)$ for the ``little" Weyl group:
\beqn
W_\La = W(K,\La) \coloneqq N_K(\La) / Z_K(\La).
\eeqn
Because the grading is stable we have
\beqn
W_\fa=W^\theta=\{w\in W\mid w \circ \theta|_{\Ts}= \theta|_{\Ts}\circ w\} ,
\eeqn
where $W=N_G(\Ts)/\Ts$. 
The pair $(K,\fg_1)$ constitutes a polar representation in the sense of~\cite{DK}. 

\subsection{Hecke algebras associated to complex reflection groups}\label{sec-hec}

Let $V$ be a finite dimensional complex vector space and $W\subset GL(V)$ a finite complex reflection group, i.e., a finite subgroup generated by complex reflections $s$. Such an $s$ fixes a hyperplane in $V$, called the reflection hyperplane of $s$. Let $\cA$ be the set of reflection hyperplanes of $(W,V)$. The braid group  associated to $W$ is defined as (see~\cite{BMR})
\beqn
B_W=\pi_1(V^{reg},x),\ \ \text{where}\ \  V^{reg}=V-\cup_{H\in\cA}H,\ \ \text{and}\ \ x\in V^{reg}.
\eeqn
For each $H\in\cA$, the point-wise stabiliser $W_H$ of $H$ in $W$ is a cyclic group generated by a reflection $s_H$ with determinant $e^{2\pi\mathbf{i}/n_H}$, where $n_H=|W_H|$. We will refer to such $s_H$'s {\em distinguished} reflections in $W$. 

 Let $\cA/W$ denote the set of $W$-orbits in $\cA$. For $H\in\cA$ we write $\bar H\in\cA/W$ for the $W$-orbit of $H$ and $n_{\bar H}=n_H$.  Let $R=\bZ[u_{\bar H,j_{\bar H}}^{\pm 1},\bar H\in\cA/W, j_{\bar H}\in[0,n_{\bar H}-1]]$, where $u_{\bar H,j_{\bar H}}$ are  indeterminates.
In~\cite{BMR} the Hecke algebra $H(W)$ is defined to be the quotient of the group algebra $R[B_W]$ by the ideal generated by elements of the form
\beqn
\prod_{i=0}^{n_H-1}(\sigma_H-u_{\bar H,i})
\eeqn
where $\sigma_H\in B_W$ is a braid generator corresponding to $s_H$.

We will consider the specialisation $H_q(W)=\bC\otimes_R H(W)$ via a homomorphism $q:R\to\bC$. Then $H_q(W)$ is a $\bC$-algebra of dimension $|W|$ (see~\cite{E} and references therein). In other words, for each reflection hyperplane $H$  (or equivalently, each distinguished reflection $s_H$ in $W$) we specify a polynomial $R_H$ of degree $n_H$ and then the Hecke algebra is given by relations $R_H(\sigma_H)$ for the braid generators $\sigma_H$ corresponding to $s_H$.

\subsection{Nearby cycle sheaves with twisted coefficients}
\label{twisted nearby}

Our main tool in this paper is the geometric construction of nearby cycle sheaves in the context of stable polar representations  from~\cite{GVX2}. We have assumed that the grading on $\fg$ is stable and this ensures that
 $(K,\Lg_1)$ is a stable polar representation. Note, however, that the notion of stability for graded Lie algebras is much stronger than the notion of stability for polar representations. In~\cite{GVX2} we do not work with arbitrary stable polar representations but have imposed certain conditions on the polar representation. The conditions are: visibility (or rank 1), locality, and the existence of regular nilpotents. We will explain these conditions below and verify  that they  are satisfied for the pair $(K,\fg_1)$. Visibility is discussed in this subsection, the existence of regular nilpotents in subsection~\ref{subsec-fund-group} and locality in subsection~\ref{subsec-regular-splitting}. 

To explain the construction in our context let us consider the adjoint quotient map
\beqn
f: \Lg_1 \to \Lg_1 \inv K \cong \La \slash W_\La \, .
\eeqn
It restricts to a fibration $f: \Lg_1^{rs} \to \La^{rs} \slash W_\La$. Let $a_0\in\fa^{rs}$. We set $\barbasepta = f (\basepta)$. Recall that
\beqn
B_{W_\La} \coloneqq \pi_1 (\La^{rs} \slash W_\La, \barbasepta)
\eeqn
is the braid group associated to the complex reflection group $W_\La$.
Let:
\beqn
I = Z_K (\La) = (\Ts)^\theta 
\eeqn
where $\Ts=Z_G(\fa)$ is the split torus as in~\eqref{split torus}. Note that, by our stable grading assumption, $I$ is a finite (abelian) group. The little Weyl group $W_\La$ acts on $I$ via conjugation. 
 Let us write: 
\beqn
X_{\barbasepta} = f^{-1} (\barbasepta) = \text{the semisimple $K$-orbit through $\basepta \,$.}
\eeqn
We have $X_{\barbasepta} \cong K / Z_K(\La)$, and therefore:
\beq
\label{fiber fund group}
\pi^K_1 (X_{\barbasepta}, \basepta) \cong I.
\eeq
The fiber $\cN_1=f^{-1}(0)$ consists of nilpotent elements in $\fg$, i.e., $\cN_1 = \fg_1 \cap \cN$ where $\cN$ is the nilpotent cone in $\fg$. According to~\cite[Theorem 4]{Vin} the group $K$ has finitely many orbits on $\cN_1$. In the language of polar representations we say that the representation $(K,\fg_1)$ is {\it visible}. 

We identify $\Lg_1^*$ with $\Lg_{m-1}=\Lg_{-1}$ using a $G$-invariant and $\theta$-invariant non-degenerate symmetric bilinear form on $\Lg$. We write $\fF$ for the topological Fourier transform functor, which induces an equivalence:
\beqn
\fF : \on{Perv}_{K} (\Lg_{-1})_{\bC^*\text{-conic}} \to \on{Perv}_{K} (\Lg_1)_{\bC^*\text{-conic}} \, ,
\eeqn
on the category of $K$-equivariant $\bC^*$-conic perverse sheaves. See \cite[Definition 3.7.8]{KS} for a definition of this functor up to a shift.
\label{subsec-twisted}
 Consider a character:
\beqn
\chi \in \hat I \coloneqq \on{Hom} (I, \bG_m).
\eeqn
By~\eqref{fiber fund group} the character $\chi$ gives rise to a rank one $K$-equivariant local system $\cL_\chi$ on $X_{\barbasepta}$, with $(\cL_\chi)_{\basepta} = \baseRing$. As explained in~\cite{GVX2} we can take the limit of $\cL_\chi$ to the nilpotent cone via the nearby cycle functor and so we obtain a $K$-equivariant perverse sheaf $P_\chi$ on $\cN_{-1} = \fg_{-1} \cap \cN$. Note that here we are working on the adjoint quotient of  $\fg_{-1}$. We now consider the Fourier transform $\fF P_\chi$. It is given as 
\beq
\fF P_\chi \cong \on{IC} (\fg_1^{rs},  \cM_\chi).
\eeq
where  $ \cM_\chi$ is a certain $K$-equivariant local system on $\fg_1^{rs}$ which we now proceed to describe in broad terms.

\subsection{The equivariant fundamental group}
\label{subsec-fund-group}

In order to describe the local system $\cM_\chi$ we analyze the equivariant fundamental group $\pi_1^K (\Lg_1^{rs}, \basepta)$; here we have picked  a basepoint $\basepta \in \La^{rs} \subset \Lg_1^{rs}$.  Let us write:
\beqn
\widetilde B_{W_\La} = \pi_1^K (\Lg_1^{rs}, \basepta), \qquad\widetilde W_\fa \coloneqq N_K (\Cartan).
\eeqn
Note that by our stable grading assumption $\widetilde W_\fa$ is a finite group.

As in \cite[Section 2.2]{GVX2}, we have a commutative diagram:
\beq
\label{mainDiagram}
\begin{CD}
1 @>>> I @>>> \widetilde B_{W_\La} @>{\tilde q}>> B_{W_\La} @>>> 1 \;\,
\\
@. @| @VV{\tilde p}V @VV{p}V @.
\\
1 @>>> I @>>> \widetilde W_\La @>{q}>> W_\La @>>> 1 \, .
\end{CD}
\eeq
We can identify the group $\widetilde B_{W_\fa}$ as a subgroup of $\widetilde W_\La \times B_{W_\La}$ as follows:
\beqn
\widetilde B_{W_\La} \cong \{ (\tilde w, \braid) \in \widetilde W_\La \times B_{W_\La} \; | \; q(\tilde w) = p(\braid) \}.
\eeqn
The top row of diagram~\eqref{mainDiagram} splits, i.e., the map $\tilde q$ admits a right inverse $\tilde r : B_{W_\La} \to \widetilde B_{W_\La}$ by utilizing the Kostant slice. We construct the Kostant slice explicitly as follows following ~\cite{L1,L2,RLYG}. Consider the pinning using a fundamental torus so that $\theta$ is given by~\eqref{stable gradings}. Using the pinning we have a regular nilpotent element 
\beq
\label{E}
E = \sum_{\alpha \in \Delta} X_\alpha\in\fg\,.
\eeq
By construction $E\in \fg_1$. Let us consider a slice $\Ls \subset \Lg_1$ containing $E$ such that the tangent space to the nilpotent cone $[\fg_0,E]$ is transverse to $\Ls$ in $\fg_1$. In the language of \cite{GVX2} this amounts to $E$ being a regular nilpotent element. As explained in~\cite[Section 2.3]{GVX2} the existence of the Kostant slice (which also guarantees the existence of a regular nilpotent in the sense of~\cite{GVX2}) implies the existence of regular splittings. 
We will discuss the notion of regular splittings in the next subsection. 

 The $K$-orbit $X_{\barbasepta}$ intersects the section $\Ls$ at exactly one point $a' \in \Ls^{rs}$. The restriction $f |_{\Ls^{rs}} : \Ls^{rs} \to \La^{rs} / W_{\La}$ induces an isomorphism of fundamental groups:
\beqn
B_{W_\La} = \pi_1 (\La^{rs} / W_{\La}, \barbasepta) \cong \pi_1 (\Ls^{rs}, a').
\eeqn

We will describe the local system $\cM_\chi$ as a representation of the group $\widetilde B_{W_\La}$  viewed as a semidirect product of $B_{W_\La}$ and $I$ via the splitting $\tilde{r}$ given by the Kostant slice. 

\subsection{Reduction to (semisimple) rank one}
\label{subsec-regular-splitting}

In this subsection we explain how to pass from our global situation to (semisimple) rank one. 

Let $s \in W_\La$ be a distinguished reflection (recall that such reflections are in bijection with reflection hyperplanes of $W_\fa$ in $\fa$) and let $\La_s \subset \La$ denote the hyperplane fixed by $s$. We will construct a new graded Lie algebra as follows. We write $G_s = Z_G(\fa_s)$. The automorphism $\theta$ restricts to an automorphism on  $G_s$ which we continue to denote by $\theta$. Thus, we obtain a graded Lie algebra $\fg_s = \Lie(G_s)$ with $K_s = Z_K (\La_s)$. Note that $G_s$ is, of course, not semisimple  contrary to our running hypothesis. The results of section~\ref{Isogenies} will allow us to later reduce to that case. 
Define:
\beqn
W_{\fa,s} = N_{K^0_s} (\La) / Z_{K^0_s} (\La) \;\; \text{and} \;\; \widetilde W_{\fa,s} = N_{K^0_s} (\La).
\eeqn
We have a natural projection $q_s : \widetilde W_{\fa,s} \to W_{\fa,s}$. Therefore, we have inclusions $W_{\fa,s} \to W_\La$ and $\widetilde W_{\fa,s} \to \widetilde W_\La$. We will use these inclusions to view $W_{\fa,s}$ (resp. $\widetilde W_{\fa,s}$) as a subgroup of $W_\La$ (resp. $\widetilde W_\La$). We have the following short exact sequence:
\beqn
1 \longrightarrow I_s \coloneqq Z_{K_s^0} (\La)  \longrightarrow \widetilde W_{\fa,s} \xrightarrow{\; q_s \;} W_{\fa,s} \longrightarrow 1 \, .
\eeqn
Note that, by construction, we have $I_s  \subset I$.

\begin{remark}
Note that here we have passed from $K_s$ to $K_s^0$ following~\cite{GVX2}. We do so as we are using the reduction to (semisimple) rank one as a means to describe a Hecke algebra which we will use to describe $\cM_\chi$. In  section~\ref{endoscopy} on endoscopy we will not make this reduction. 
\end{remark}

To apply~\cite{GVX2}, we verify that the {\it locality} property holds in our setting.
\begin{lemma}
We have $W_{\fa,s} = Z_{W_\fa}(\fa_s)=\langle s\rangle$, the subgroup of $W_\fa$ generated by the distinguished reflection $s$.
\end{lemma}
\begin{proof}
The fact that $Z_{W_\fa}(\fa_s)=\langle s\rangle$ follows from a theorem of Steinberg's~\cite[Theorem 1.5]{St} (see also~\cite[Theorem 9.44]{LT}).

Note that the torus $\Ts$ is a maximal torus satisfying the hypothesis of~\cite[Proposition 13]{Vin}. Thus by [{\em loc.cit.}] the pair $(Z_\Lg(\fa_s),\theta)$ is saturated and therefore $W_{\fa,s}= N_{K^0_s} (\La) / Z_{K^0_s} (\La) =N_{K_s}(\La)/Z_{K_s}(\La)$. The latter group coincides with $\langle s\rangle$. 
\end{proof}

Note that specifying a  splitting $\tilde r : B_{W_\La} \to \widetilde B_{W_\La}$ is equivalent to specifying a map $r: B_{W_\La}\to \widetilde W_\La$ such that $q\circ  r = p$. A splitting $\tilde r$ is called {\it regular} if $r(\sigma_s)\in \widetilde W_{\fa,s}$, where $\sigma_s\in B_{W_\fa}$  is a chosen braid generator corresponding to $s$. As was observed earlier the fact that our splitting comes from a Kostant slice guarantees that it is regular. 

In~\cite{GVX2} we associate a polynomial $R_{\chi,s}$ to the polar representation $((G_s^\theta)^0=K_s^0,Z_{\Lg_1}(\fa_s))$ and the character $\chi_s=\chi|_{I_s}$ of $I_s$. These polynomials are used to construct a Hecke algebra and then we describe $\cM_\chi$ in terms of that Hecke algebra. The polynomials $R_{\chi,s}$ are determined explicitly in section~\ref{sec-stable rank 1} in our stable grading setting.

\subsection{Determination of the local system \texorpdfstring{$\cM_\chi$}{Lg}}\label{sec-local system Mchi}

In this subsection we explain how to express the local system $\cM_\chi$ in terms of the local data discussed in the previous subsection. 

Let us begin with a character $\chi$ of $I$. We set
\beqn
W_{\fa,\chi} \coloneqq \on{Stab}_{W_\fa} (\chi) \;\; \text{and} \;\;
B_{W_\fa}^\chi \coloneqq \on{Stab}_{B_{W_\fa}} (\chi) = p^{-1} (W_{\fa,\chi})\,.
\eeqn
For each distinguished reflection $s$, set 
\beqn
\text{$W_{\fa,s,\chi} = W_{\fa,s} \cap W_{\fa,\chi}$
and $e_s = |W_{\fa,s}|\ / |W_{\fa,s,\chi} | \in \bN$,}
\eeqn
 so that 
$W_{\fa,s,\chi} = \langle s^{e_s} \rangle$.  

\begin{definition}
\label{def 0}
Let $W_{\fa,\chi}^0 \subset W_{\fa,\chi}$ be the subgroup generated by all
the $W_{\fa,s,\chi}$. 
\end{definition} 
Note that $W_{\fa,\chi}^0$
is a complex reflection group acting on $\Cartan$.

The character $\chi$ induces a character $\chi_s :I_s \to \bG_m$ for every distinguished reflection $s$, that is, $\chi_s=\chi|_{I_s}$. As explained in the previous subsection, in~\cite{GVX2} to each such character  we have associated a polynomial 
$$
R_{\chi,s}\in\bC[x],
$$
which is the minimal polynomial of the monodromy around the hyperplane $\fa_s$.
 Moreover, we  have shown that there exists $\bar R_{\chi,s}\in\bC[x]$ such that 
\beq\label{mono-1}
R_{\chi,s}(x)=\bar R_{\chi,s}(x^{e_s}),
\eeq
i.e. that $R_{\chi,s}$ is in fact a polynomial in $x^{e_s}$. 
 Note that the distinguished reflections in $W_{\fa,\chi}^0$ are  precisely the $s^{e_s}$. Thus, we can define a Hecke algebra $\cH_{W_{\fa,\chi}^0}$ associated to the complex reflection group $W_{\fa,\chi}^0$ with relations given by the polynomials $\bar R_{\chi,s}$ as in subsection~\ref{sec-hec}.

Let
\beqn
B_{W_\fa}^{\chi, 0}=p^{-1}(W_{\fa,\chi}^{0}),\ \widetilde B_{W_\fa}^\chi = \tilde q^{-1} (B_{W_\fa}^\chi)
\;\; \text{and} \;\;
\widetilde B_{W_\fa}^{\chi, 0} = \tilde q^{-1} (B_{W_\fa}^{\chi, 0}).
\eeqn
By~\cite{GVX2}, $\cM_\chi$ is the $K$-equivariant local system on $\Lg_1$ corresponding to the following representation of $\pi_1^{K}(\Lg_1^{rs})=\widetilde{B}_{W_\fa}$
\beq
\cM_\chi \ = \ \left( \bC [\widetilde B_{W_\fa}]
\otimes_{\bC [\widetilde B_{W_\fa}^{\chi, 0}]}
(\bC_\chi \otimes  \cH_{W_{\fa,\chi}^0}) \right)\otimes \bC_\tau \, ,
\eeq
where the Hecke algebra $\cH_{W_{\fa,\chi}^0}$ is viewed as a
$\bC [\widetilde B_{W_\fa}^{\chi, 0}]$-module via the composition of maps
\beqn
\widetilde B_{W_\fa}^{\chi, 0}\to B_{W_\fa}^{\chi, 0} \to
B_{W_{\fa,\chi}^0} \, ,
\eeqn
$\bC_\chi$ is viewed as a
$\bC [\widetilde B_{W_\fa}^{\chi, 0}]$-module via the unique extension of the character $\chi$ to $\widetilde B_{W_\fa}^{\chi, 0}$, and $\tau$ is the character on $I$ given by $\tau(x)= \det(x|_{\fg_1})$. As observed in~\cite{GVX2}, the character $\tau$ takes values in $\{\pm 1\}$.

\begin{remark}\label{remark tau}
We will check that $\tau$ is trivial in our setting for classical types via a case by case analysis. We do not know of a uniform proof. Note that in~\cite{VX1} we allow disconnected $K$ in which case non-trivial $\tau$ occurs.
\end{remark}
\begin{remark}\label{remark rho}
In~\cite{GVX2} another character $\rho$ appears. However, our results in section~\ref{sec-stable rank 1} show that it is trivial in our setting. 
\end{remark}
 
\section{Isogenies, products, and the reductive case}
\label{Isogenies}

We have assumed that the semisimple group $G$ is simply connected. This ensures that the group $K$ is connected. In this section we briefly describe how to handle the general case. We also discuss the case of a reductive group. The case of a disconnected $K$ will come up in section~\ref{endoscopy} on endoscopy. 

 Let us consider an arbitrary semisimple group $G$ equipped with a finite order automorphism $\theta$ which induces a  (stable) grading on $\fg$. Let us write $G_{\!{sc}}$ for the simply connected cover of $G$. The automorphism $\theta$ lifts to an automorphism of $G_{sc}$ which we continue to denote by $\theta$. We write $K=G^\theta$ and $K_{sc}=(G_{\!{sc}})^\theta$. Then $K_{sc}$ is connected and is a cover of $K^0$, the connected component of $K$. Thus $\on{Perv}_{K^0}(\cN_1)$ is a full subcategory of $\on{Perv}_{K_{sc}}(\cN_1)$. Therefore, the theory of character sheaves for the pair $(K^0,\fg_1)$ can be deduced from the theory from $(K_{sc},\fg_1)$ as $I^0 = Z_{K^0}(\fa)$ is naturally a quotient of  $I_{sc} = Z_{K_{sc}}(\fa)$.

Let us now compare the theories for $(K^0,\fg_1)$ and $(K,\fg_1)$. Let us write, as before, $I = Z_{K}(\fa)$. Then we have an exact sequence 
$$
1 \to I^0 \to I \to K/K^0 \to 1\,.
$$
By~\cite[Proposition 13]{Vin} we have saturation, i.e., for the little Weyl group we have
\beqn
W_\fa \ = \ N_K(\fa)/Z_K(\fa) \ = \ N_{K^0}(\fa)/Z_{K^0}(\fa)\,.
\eeqn
The braid group $B_{W_{\fa}}$ acts on $I$. The action factors through $W_{\fa}$ and is trivial on the quotient  $K/K^0$. Therefore we obtain the following diagram:
\beqn
\begin{CD}
1 @>>> I^0 @>>> \widetilde B^0_{W_\La}= \pi_1^{K^0} (\Lg_1^{rs}, \basepta) @>>> B_{W_\La} @>>> 1
\\
@. @VVV @VVV @| @.
\\
1 @>>> I @>>> \widetilde B_{W_\La}= \pi_1^K (\Lg_1^{rs}, \basepta) @>>> B_{W_\La} @>>> 1
\\
@. @VVV @VVV @. @.
\\
@.  K / K^0 @= K / K^0 \, . @. @. 
\end{CD}
\eeqn

Let us now consider a character $\chi \in \hat I$ and let us write $\chi^0$ for its restriction to $I^0$. Because the action of $W_{\fa}$ is trivial on $K/K^0$ we wee that  $W_{\fa,\chi}=W_{\fa,\chi^0}$. These facts allow one to carry out the proof of the main theorem in~\cite{GVX2} to the situation of the pair $(K,\fg_1)$ and we obtain the same formula as in the connected case:
\beq
\label{M}
\cM_\chi \ = \ \left( \bC [\widetilde B_{W_\fa}]
\otimes_{\bC [\widetilde B_{W_\fa}^{\chi, 0}]}
(\bC_\chi \otimes  \cH_{W_{\fa,\chi}^0}) \right)\otimes \bC_\tau \,. 
\eeq
As we remarked in the previous section we know by case-by-case analysis that for connected $K$ the character $\tau$ is trivial. Thus, the character $\tau$ comes from a character of $K/K^0$. Thus, the disconnectedness of $K$ only adds the character $\tau$ to the story. This is consistent with our previous work in the symmetric pair situation~\cite{VX1}. Note also that the key point of this formula is the representation 
\beqn
 \cH_{W_{\fa,\chi}^0} \otimes \bC_{\chi+\tau} 
\eeqn
of $I\rtimes B_{W_{\fa,\chi}^0}$. Recall also that the character $\bC_{\chi+\tau}$ can be regarded as a character of $I\rtimes B_{W_{\fa,\chi}^0}$ because $\chi+\tau$ is  fixed by $W_{\fa,\chi}^0$. Note also that the representation  $\cH_{W_{\fa,\chi}^0}$ only depends on $\chi^0$. 

We could also study the relationship between the pairs $(K^0,\fg_1)$ and $(K,\fg_1)$ by restriction and induction. The forgetful functor $\on{For}:\on{Perv}_{K}(\fg_{\pm}) \to \on{Perv}_{K^0}(\fg_{\pm})$ has a simultaneously left and right adjoint 
$$
\on{Ind}_{K^0}^K: \on{Perv}_{K^0}(\fg_{\pm}) \to \on{Perv}_{K}(\fg_{\pm})\,.
$$
Obviously,
\beqn
\on{For} (\cM_\chi) \ =\  \cM_{\chi^0}\,.
\eeqn
Conversely, 
Let us consider a character $\chi^0\in \widehat {I^0} = \on{Hom} (I^0, \bG_m)$. Let us write
\beqn
\on{Ind}_{I^0}^I(\chi^0) \ = \ \oplus \chi_i  \,.
\eeqn
The $\chi_i$ constitute all the characters of $I$ whose restriction to $I^0$ is $\chi^0$. The functor $\on{Ind}_{K^0}^K$ commute with Fourier transform and thus we obtain
\beqn
\fF( \oplus P_{\chi_i}) \cong  \on{IC} (\fg_1^{rs},  \on{Ind}_{K^0}^K\cM_{\chi^0}) .
\eeqn
As a representation of $I\rtimes B_{W_{\fa,\chi}^0}$ inducing from $I^0\rtimes B_{W_{\fa,\chi}^0}$ we get
\beqn
\on{Ind}_{K^0}^K( \cH_{W_{\fa,\chi}^0} \otimes \bC_{\chi^0}) = \cH_{W_{\fa,\chi}^0} \otimes \on{Ind}_{I^0}^I( \bC_{\chi^0}) =   \oplus (\cH_{W_{\fa,\chi}^0} \otimes\bC_{\chi_i}) \,.
\eeqn
Note, however, that the formula~\eqref{M} is more precise. It tells us that starting from $\chi_i$ we obtain the representation $\cH_{W_{\fa,\chi}^0} \otimes \bC_{\chi_i+\tau}$ whereas the induction process only tells us where the whole packet of characters goes. For the purposes of classification of character sheaves the shift by $\tau$ can of course be ignored. 

The above considerations show that the case of semisimple $G$ can be reduced to the case of simply connected $G$ and we will now consider that case. Let us write $G$ as a products of its simple components $G^i$, $i=1, \dots, d$. The automorphism $\theta$ can permute isomorphic components. To analyze this situation, it suffices to consider the case when all the $G^i$ are isomorphic and $\theta$ permutes them cyclically. In that case $d|m$ and we can write $m=d m'$. Now, $\theta' = \theta^d$ is a (stable) automorphism of $G^1$ of order $m'$. We have
 \beqn
 K^1 = (G^1)^{\theta'} \cong G^\theta = K \qquad \text{and} \qquad (\fg^1)_1 \cong \fg_1\,
 \eeqn
 where $\Lg^1=\on{Lie}(G^1)$ and $(\fg^1)_1$ denotes the $\zeta_m^d$-eigenspace of $\theta'$ in $\Lg^1$. 
 These isomorphisms induce an isomorphism of polar representations 
 \beq\label{reduction-prod}
 (K,\fg_1) \ \cong \ (K^1,(\fg^1)_1)\,.
 \eeq
 Thus we are reduced to the case of an almost simple simply connected $G$.
 \begin{remark}
 When $m=d=2$ this argument shows that the case of $G\times G$ where $\theta$ permutes the components reduces to the classical case of the adjoint action of $G$ on $\fg$. 
 \end{remark}

Let us now assume that $G$ is reductive. We obtain a polar representation of $K=G^\theta$ on $\fg_1$. Proceeding as above, we can reduce to the study of the pair $(K^0,\fg_1)$. Let us now consider the derived group $G_{\text{der}}$ of $G$. We claim that we can further reduce the study of the pair $(K^0,\fg_1)$ to the pair $(K_{\text{der}}^0,(\fg_{\text{der}})_1)$, where $K_{\text{der}}=(G_{\text{der}})^\theta$ and $\fg_{\text{der}}=\on{Lie}G_{\text{der}}$. Observe that stability of the grading implies that  $K^0=K_{\text{der}}^0$. Thus, if $\fa$ is a Cartan subspace of $\fg_1$ we see that 
\beqn
\fg_1\inv K^0 = \fg_1\inv K_{\text{der}}^0 = \fa/W_\fa\,.
\eeqn
We note that $\fa = \fa_{\text{der}} \oplus Z(\Lg)_1$ where $\fa_{\text{der}}\subset(\Lg_{\on{der}})_1$ is a Cartan subspace and $Z(\Lg)$ denotes the center of $\fg$. Thus, we have
\beqn
\fa/W_\fa =  \fg_1\inv K_{\text{der}}^0 =   (\fg_{\text{der}})_1\inv K_{\text{der}}^0\times Z(\Lg)_1 = \fa_{\text{der}}/W_\fa \times Z(\Lg)_1. 
\eeqn
As $K^0=K_{\text{der}}^0$ we obviously have 
\beq\label{Isder}
I_{\text{der}} = Z_{K_{\text{der}}^0}(\fa_{\text{der}}) = Z_{K^0}(\fa) = I.
\eeq
In this manner we conclude that  the calculation of the Hecke algebra $\cH_{W_{\fa,\chi}^0}$ in formula~\eqref{M} can be done in the derived group $G_{\text{der}}$. 

\begin{remark}
The stability of the grading does not play an essential role in this section. The result in this section hold as long as the pair $(K,\fg_1)$ satisfies the hypotheses of~\cite{GVX2} with the exception that the map $I_{\text{der}} \to I$ is in general only a surjection. 
\end{remark}

\section{Stable rank 1 gradings: the Coxeter and twisted Coxeter cases}\label{sec-stable rank 1}

In this section, we consider the stable gradings of the Lie algebra $\Lg$ of an almost simple simply connected algebraic group $G$ such that $m$ is a maximal regular elliptic number of $W\vartheta$. That is, $m$ equals the Coxeter number when $\vartheta=1$ (resp. the twisted Coxeter number when $\vartheta\neq 1$). We refer to these gradings as Coxeter cases (resp. twisted Coxeter cases). For classical types, these exhaust the rank 1 stable gradings, while for exceptional groups, there are a few stable rank 1 gradings that need to be dealt with separately.

In subsection~\ref{twisted nearby} we have associated to a character $\chi\in \hat I$ a nearby cycle sheaf $P_\chi$ on $\cN_{-1}$. We have $\fF P_\chi \cong \on{IC} (\Lg_1^{rs}, \cM_\chi)$ where $\cM_\chi$ is a $K$-equivariant local system on $\Lg_1^{rs}$. In this section we determine $\cM_\chi$ as a representation of the equivariant fundamental group $\pi_1^K(\Lg_1^{rs})$ in the Coxeter and twisted Coxeter cases. We use the notations from subsection~\ref{twisted nearby}.

As in subsection~\ref{ssec-stable gradings} we fix a pinning $(B,T,\{X_\alpha\}_{\alpha\in\Delta})$. We extend $\{X_{\alpha},\,\alpha\in\Delta\}$ to a basis of $\oplus_{\beta\in\Phi}\Lg_\beta$ by choosing $X_\beta\in\Lg_\beta$ for each $\beta\in\Phi-\Delta$.

\subsection{The Coxeter case}\label{sec-cox}
In this subsection, let $\theta=\on{Int}(\check\rho_{ad}(\zeta_m))$ (see~\eqref{stable gradings}), where $m$ is the Coxeter number of $W=W(G,T)$. Let us write $\nu_G=\on{rank}G=|\Delta|$ and let $\alpha_i,i=1,\ldots,\nu_G$ denote the simple roots. For $\beta=\sum_{i=1}^{\nu_G}k_i\alpha_i\in\Phi$, we write $\on{ht}\beta=\sum_ik_i$. We have
\beqn
m=\max\{\on{ht}\beta\mid\beta\in\Phi\}+1.
\eeqn

\begin{lemma}We have
\beqn
\Lg_0=\on{Lie}K=\Lt,\ \ \Lg_1=\bigoplus_{i=1}^{\nu_G}\bC X_{\alpha_i}\oplus\bC X_{-\alpha_0},\ K=T \eeqn
where $\alpha_0\in\Phi^+$ is the highest root.
\end{lemma}
\begin{proof}

Note that $\on{ht}\alpha_0=m-1$. Since for each $\beta\in\Phi$,
\beqn
\theta(X_\beta)=\zeta_m^{\on{ht}\beta}X_\beta,\text{ and }|\on{ht}\beta|\leq m-1,
\eeqn
the lemma follows.
\end{proof}

 We identify 
\beq\label{identify-affine1}
\Lg_1\cong\bC^{\nu_G+1},\ \ \sum_{i=1}^{\nu_G}{a_i}X_{\alpha_i}+a_0X_{-\alpha_0}\mapsto(a_1,\ldots,a_{\nu_G},a_0).
\eeq
Thus, in the $a_i$-coordinates the torus $T$ acts on $\Lg_1$ by the formula 
\beqn
 t.(a_1,\ldots,a_{\nu_G},a_0) =(\alpha_1(t) a_1,\ldots,\alpha_{\nu_G}(t) a_{\nu_G},\alpha_0(t^{-1}) a_0).
\eeqn
We write 
\beq\label{the nis}
\alpha_0=\sum_{i=1}^{\nu_G}n_i\alpha_i.
\eeq Under the identification~\eqref{identify-affine1}
one checks readily that
\beq
\label{invariant}
\bC[\Lg_1]^{K}\cong \bC[a_0\prod_{i=1}^{\nu_G}a_i^{n_i}].
\eeq
Moreover
\beqn
\Lg_1^{rs}=\{\sum_{i=1}^{\nu_G}{a_i}X_{\alpha_i}+a_0X_{-\alpha_0}\mid a_i\neq 0, i=0,1,\ldots,\nu_G\}\cong(\bC^*)^{\nu_G+1}
\eeqn
and so the nilpotent cone $\cN_1$ is the normal crossing divisor given by 
\beqn
\cN_1 \ = \ \bigcup_{i=0}^{\nu_G} \{(a_1,\ldots,a_{\nu_G},a_0)\in \fg_1\mid a_i=0  \}\,.
\eeqn
\begin{lemma}We have
\beqn
I=Z(G)\ \ \text{ and }\ \ W_\fa\cong\bZ/m\bZ.
\eeqn
\end{lemma}
\begin{proof}
A Cartan subspace $\La\subset\Lg_1$ can be chosen as $\fa=\bC(\sum_{i=1}^{\nu_G}X_{\alpha_i}+X_{-\alpha_0})$. We then have
\bern
&&N_{K}(\fa)=\{t\in T\,|\,\alpha_k(t)=\prod_{i=1}^{\nu_G}\alpha_i(t^{-n_i}),\,i=1,\ldots,\nu_G\}\\&&=\{t\in T\mid\alpha_1(t)^m=1,\,\alpha_k(t)=\alpha_{1}(t),k=2,\ldots,\nu_G\};\\
&&Z_{K}(\fa)=\{t\in T\,|\,\alpha_i(t)=1,\,i=1,\ldots,\nu_G\}=Z(G).
\eern
The lemma follows.
\end{proof}
To describe the Kostant slice recall the element $E= \sum_{\alpha \in \Delta} X_\alpha$ of~\eqref{E}. In our coordinates it corresponds to the point $(1, \dots, 1,0)$. Thus, we can choose the Kostant slice $\Ls$ to be
\beq
\label{rank one Kostant}
\Ls \ = \ \{(1, \dots, 1, x)\mid x\in \bC\}\,.
\eeq
Recall the top row in the diagram~\eqref{mainDiagram}:
\beqn
1 \to I=Z(G) \to \pi_1^{T}(\Lg_1^{rs}) \to \pi_1(\fa^{rs}/W_\fa) = \bZ \to 0 \,.
\eeqn
We split this exact sequence by the Kostant slice and obtain an isomorphism
\begin{equation}
\pi_1^{T}(\Lg_1^{rs})\   \cong \ \pi_1(\fa^{rs}/W_\fa) \oplus I = \bZ  \oplus Z(G). 
\end{equation}
To state our theorem in this subsection we will make the description of $\pi_1^T(\Lg_1^{rs})$ very explicit. 
Let us write $\gamma_i$ for the generators of $\pi_1(\fg^{rs})$ given by loops around the axes $a_i=0$. Consider the map $T\to \Lg_1^{rs}\cong(\bC^*)^{\nu_G+1}$
\beqn
t\mapsto t.(1,\ldots,1)=(\alpha_1(t),\ldots,\alpha_{\nu_G}(t),\alpha_0(t^{-1})).
\eeqn
This induces a map
\beq\label{psi0}
\bega
\Psi_0:\pi_1(T)\cong X_*(T)\to\pi_1(\Lg_1^{rs})=(\prod_{i=1}^{\nu_G}\bZ\gamma_i)\times\bZ\gamma_0\\
\lambda\mapsto(\langle\lambda,\alpha_1\rangle\gamma_1, \langle\lambda,\alpha_2\rangle\gamma_2,\ldots,\langle\lambda,\alpha_{\nu_G}\rangle\gamma_{\nu_G},\langle\lambda,-\alpha_0\rangle\gamma_0)\,
\eega
\eeq
where $\langle-,-\rangle:X_*(T)\times X^*(T)\to \bZ$ is the natural pairing. 
By definition $\pi_1^{T}(\Lg_1^{rs})$ is the cokernel of $\Psi_0$, i.e.,  we have an exact sequence
\beqn
0 \to X_*(T) \xrightarrow{\Psi_0} \pi_1(\Lg_1^{rs})=(\prod_{i=1}^{\nu_G}\bZ\gamma_i)\times\bZ\gamma_0 \to \pi_1^{T}(\Lg_1^{rs}) \to 0\,.
\eeqn
The Kostant slice of~\eqref{rank one Kostant} gives us an isomorphism $\Ls^{reg} \cong \fa^{rs}/W_\fa$ and hence an identification $\pi_1(\Ls^{reg} ) \cong \pi_1(\fa^{rs}/W_\fa) = \bZ$. Thus it allows us to identify the element $\gamma_0$ with the element $1\in  \bZ =  \pi_1(\fa^{rs}/W_\fa)$.

We also note that by the formula~\eqref{invariant} the composition of maps 
\begin{subequations}
\label{to braid group}
\beq
\pi_1(\Lg_1^{rs})=(\prod_{i=1}^{\nu_G}\bZ\gamma_i)\times\bZ\gamma_0 \to \pi_1^{T}(\Lg_1^{rs}) \to  \pi_1(\fa^{rs}/W_\fa) = \bZ
\eeq
is given by 
\beq
(b_1, \dots, b_{\nu_G},b_0) \mapsto b_0 + \sum_{i=1}^{\nu_G} n_ib_i \,.
\eeq
\end{subequations}

Let us now consider the center $Z(G)$ of $G$. Recall that we have a canonical identification 
\beqn
Z(G) \ = \ \operatorname{Hom}(X^*(T)/\bZ\!\cdot\!\Phi, \bG_m)\,.
\eeqn
As $X^*(T)/\bZ\!\cdot\!\Phi$ is a finite group we can replace $\bG_m$ with $\bQ/\bZ$ in this equality. This identification is not canonical but depends on choosing an appropriate primitive root of unity. This choice is not so important, but we do make it explicit in working out the formulas later in the paper.  We will identify $\prod_{i=1}^{\nu_G}\bZ\gamma_i$ with the co-weight lattice $\check P\subset \bQ \otimes_\bZ X_*(T)$, where the generators $\gamma_i$ form a basis dual to the simple roots. We now have a canonical map
\beqn
\check P \xrightarrow{\bar\Pi} \operatorname{Hom}(X^*(T)/\bZ\!\cdot\!\Phi, \bQ/\bZ) \ = \ Z(G)
\eeqn
which gives rise to the following exact sequence 
\beq
\label{center}
0  \longrightarrow X_*(T) \xrightarrow{(\alpha_1,\ldots,\alpha_{\nu_G})} \prod_{i=1}^{\nu_G}\bZ\gamma_i =\check P\xrightarrow{\ \overline\Pi\ } Z(G)=I\longrightarrow 0\,.
\eeq
Putting this all together we have
\beqn
\xymatrix{0\ar[r]&X_*(T)\ar[r]^-{(\alpha_1,\ldots,\alpha_{\nu_G})}\ar@{=}[d]&\prod_{i=1}^{\nu_G}\bZ\gamma_i\ar[r]^{\overline\Pi}&I=Z(G)\ar[r]&0
\\
0\ar[r]&X_*(T)\ar[r]^-{\Psi_0}&\pi_1(\Lg_1^{rs})=(\prod_{i=1}^{\nu_G}\bZ\gamma_i)\times\bZ\gamma_0\ar[r]^-{\Pi}\ar@{->>}[u]_{\text{pr}_1}&\pi_1^{K}(\Lg_1^{rs})=\bZ\oplus Z(G)\ar[r]\ar[d]^{\tilde q}\ar[u]_{\text{pr}_2}&0
\\
&&&\pi_1(\fa^{rs}/W_\fa)\cong\bZ\ar[ul]^{\kappa}&}
\eeqn
The Kostant slice $\kappa$ is given by  $\kappa(1)=\gamma_0$. Thus, the map $\text{pr}_2$ is induced by the canonical projection $\text{pr}_1$.  Let us write 
\beq\label{def of gammai}
\bar\gamma_i = \overline\Pi(\gamma_i)\in Z(G),\,i=1, \dots , i_{\nu_G}.
\eeq Then we have, making use of~\eqref{to braid group}, that
\beq
\label{generators}
\Pi(\gamma_0)=(1,\bar 1)\in \bZ\oplus Z(G) \qquad  \Pi(\gamma_i)= (n_i,\bar\gamma_i )\in \bZ\oplus Z(G)\ \  \text{for} \ \ i=1, \dots , i_{\nu_G} \,,
\eeq
where $\bar 1$ stands for the identity element in $Z(G)$. 
 In the following theorem we describe the structure of $\cM_\chi$ as a representation of $ \pi_1^{K}(\Lg_1^{rs})\cong\pi_1(\fa^{rs}/W_\fa) \oplus I = \bZ  \oplus Z(G)$. To that end let us denote the action of $1\in \bZ = \ \pi_1(\fa^{rs}/W_\fa)$ on $\cM_\chi$ by $x$.
\begin{theorem}
\label{coxeter rank one}
We have
\beqn
\cM_\chi \ = \bC_\chi\otimes \big(\bC[x]/R_\chi(x)\big)\,,\ R_\chi(x)=(x-1)\prod_{i=1}^{\nu_G} ( \chi(\bar\gamma_i)x^{n_i}-1)\, ,
\eeqn
where the $n_i$'s are defined in~\eqref{the nis} 
and $\bar\gamma_i$ is defined in~\eqref{def of gammai}. 
\end{theorem}
The polynomials $R_\chi$ are calculated explicitly in each type in~\S\ref{cal-cox}.

We make use of the  following 
\begin{prop}
Let us consider an $IC$-sheaf $\cF$ on $\bC^n$ associated to a  local system on $\bC^n-D$ where $D$ is the normal crossing divisor consisting of the coordinate hyperplanes. We write $\mu_i$ for the monodromies around the coordinate hyperplanes $D_i$. If the support of the Fourier transform of $\cF$ is precisely $D$ then the actions of $\mu_i$ satisfy
$$
\prod_{i=1}^{n} ( \mu_i -1)  = 0 \qquad \text{but for all }j\qquad   \prod_{i,\,i\neq j} (\mu_i -1)  \neq 0 \,.
$$
\end{prop}
\begin{proof}

Let us recall the classification of perverse sheaves in the normal crossing situation, i.e., the perverse sheaves that are constructible with respect to the stratification given by intersections of the coordinate hyperplanes. We consider the quiver whose vertices are subsets $J$ of $\{1, \dots, n\}$ and whose arrows are given as follows. If $J\subset J'$ and $|J'|=|J|+1$ then we have morphisms $q_{J,J'}: J' \to J$ and $p_{J',J}: J \to J'$ subject to the following conditions:
\begin{subequations}
\beq
\text{$\mu_{J',J} =1 + q_{J,J'}p_{J',J}$\ \  and \ \ $\mu_{J,J'}=1 + p_{J',J}q_{J,J'}$\ \ \  are invertible}
\eeq
\beq
\label{commute}
\bega
\text{For $J\subset J'$, $|J'|=|J|+2$  and $J\subset J''\subset J'$ $J\subset J'''\subset J'$  we have }
\\
p_{J',J''}p_{J'',J}=p_{J',J'''} p_{J''',J}\ \ \   q_{J,J''} q_{J'',J'}=q_{J,J'''}q_{J''',J'} \ \ \ q_{J''',J'}p_{J',J''}=p_{J''',J}q_{J,J''}
\eega
\eeq
\end{subequations}
There is a functor which associates to a perverse sheaf $\cF$, constructible with respect to the normal crossings divisor, the representation of the above quiver given by its various microlocalizations as follows. To each $J\subset\{1,\ldots,n\}$ we associate the vector space $B_J$ which is the (generic) microlocalization along the stratum $D_J = \cap_{i\in J} D_i$. The maps $q_{J,J'}$ and $p_{J',J}$ are constructed in \cite{GGM} and depend on a particular (natural) identification of various fundamental groups. 

Let us now consider the IC-sheaf $\cF$ of the statement. The local system is given by $B_\emptyset$ and the monodromies on it are given by $\mu_i :=\mu_{\{i\},\emptyset}= 1 + q_{\emptyset,\{i\}} p_{\{i\},\emptyset}$. The IC extension is characterized by the property that all the $p_{J',J}$ are surjections and all the  $q_{J,J'}$ are injections. Making use of equations~\eqref{commute} one concludes that we can choose
\beqn
\bega
B_J \ = \ \operatorname{Im}( \prod_{i \in J} (\mu_i -1) )\,
\\
\text{and for $J'=J\cup\{i\}$ \ \ $p_{J',J}=\mu_i-1$ \ \ and \ \  $q_{J,J'}: B_{J'}\to B_J$ \  is the inclusion}\,.
\eega
\eeqn
Let us temporarily write $L=\{1, \dots, n\}$. 
The fact that the Fourier transform does not have the conormal bundle at the origin in its support implies that the IC-sheaf has no vanishing cycles at the origin, i.e., that  $B_L=0$. The fact that the support of the Fourier transform is precisely $D$ means that $\cF$ does have vanishing cycles generically along the coordinate hyperplanes, i.e., that  $B_{L-\{i\}}\neq 0$ for any $i\in L$. This gives us the conclusion. 
\end{proof}

\begin{proof}[Proof of Theorem~\ref{coxeter rank one}.] We apply the previous proposition in our setting. Let us recall formulas~\eqref{generators} and the fact that we write $x$ for the action of $\bar\Pi(\gamma_0)$ on $\cM_\chi$. Thus we set $\mu_0= x$ and $\mu_i=  \chi(\bar\gamma_i)x^{n_i}$ for $i\neq 0$. 
Applying this proposition to our situation we conclude that
$$
R_\chi(x)=(x-1)\prod_{i=1}^{\nu_G} ( \chi(\bar\gamma_i)x^{n_i}-1) \,.
$$
Recall that  $\cM_\chi$ is a rank $m$ cyclic module over $\bZ = \ \pi_1(\fa^{rs}/W_\fa)$. Since the polynomial $R_\chi(x)$ is of degree $m$ (recall that
$m=1+\sum_{i=1}^{\nu_G}n_i$),   we conclude that it is the minimal polynomial (and at the same time the characteristic polynomial) of the transformation $x$. Thus, the theorem follows.
\end{proof}

\subsection{The twisted Coxeter case}\label{sec-tcox}
In this subsection, let $\theta=\on{Int}\check\rho_{ad}(\zeta_m)\circ\vartheta$ (see~\eqref{stable gradings}), where $m$ is the twisted Coxeter number and $\vartheta$ is a non-trivial pinned automorphism. 

Let $\sigma\in\on{Aut}\Phi$ be the automorphism induced by $\vartheta$, i.e., $\vartheta(X_{\alpha_i})=X_{\sigma\alpha_i}$, $\alpha_i\in\Delta$. Let $e=\on{ord}\vartheta$. Then $e=2$ in type $A_n$, $n\geq 2$, $D_n$, $n\geq 3$ or $E_6$, or $e=3$ in type $D_4$. We say that $(G,\vartheta)$ is of type $A_n^2$, $D_n^2$, $E_6^2$, $D_4^3$ respectively.

Let us write $n_{\vartheta}$ for the number of $\sigma$-orbits in $\Delta$, and $\Delta_i$, $i=1,\ldots,n_{\vartheta}$ for the $\sigma$-orbits in $\Delta$. Let
\beq\label{simple roots reps}
\{\alpha_1,\ldots,\alpha_{n_{\vartheta}}\}\text{ be a set of representatives of the $\sigma$-orbits in $\Delta$.}
\eeq
We define 
\beq\label{def-yi-1}
Y_i=\sum_{\alpha\in\Delta_i}X_\alpha.
\eeq 
Consider the set $\Sigma_0$ consisting of roots of height $m/e-1$ in $\Phi$. One checks easily that $\Sigma_0$ forms a single $\sigma$-orbit.  Let $\beta_0\in\Sigma_0$. We define
\beq\label{def of y0}
Y_0=\sum_{i=0}^{e-1}\zeta_m^{-im/e}\vartheta^i(X_{-\beta_0})
\eeq
 Note that $\Sigma_0$ consists of $e$ elements except in the case of $A_{2n}^2$. In the latter case $\Sigma_0=\{\beta_0\}$ and $\vartheta(X_{-\beta_0})=-X_{-\beta_0}$. Thus $Y_0\neq 0$ and the subspace spanned by $Y_0$ does not depend on the choice of $\beta_0$ in $\Sigma_0$.

\begin{lemma}
We have
\beqn\label{g0 and g1}
\bega
\Lg_0=\Lt^{\vartheta}=\{t\in\Lt\,|\,\vartheta(t)=t\},\quad
\Lg_1=\bigoplus_{i=0}^{n_\vartheta}\bC Y_i,\text{ and }K=T^{\vartheta}.
\eega
\eeqn
where $Y_i$, $i=1,\ldots,n_\vartheta$ are defined in~\eqref{def-yi-1}, and $Y_0$ is defined in~\eqref{def of y0}.
\end{lemma}
\begin{proof}

Note that $\on{ht}\alpha<m$ for all $\alpha\in\Phi$. 
We have
\beqn
\bega
\Lg_0=\Lt^{\vartheta}\oplus\on{span}\left\{X_\alpha-\vartheta(X_\alpha)\mid \frac{m}{2}+\on{ht}\alpha\equiv 0\mod m\right\}.
\eega
\eeqn
Let us write $\Sigma_1=\{\alpha\in\Phi\mid\frac{m}{2}+\on{ht}\alpha\equiv 0\mod m\}$. Then it is easy to check that $\Sigma_1=\emptyset$  in type $A_{2n-1}^2$, and $\sigma(\alpha)=\alpha$ for all $\alpha\in\Sigma_1$ in other types. The claim on $\Lg_0$ follows  in type $A_{2n-1}^2$, and follows  from the fact that $\vartheta(X_\alpha)=X_\alpha$ for each $\alpha$ satisfying $\sigma(\alpha)=\alpha$ in other types. One checks readily that $Y_i\in\Lg_1$, $i=0,\ldots,n_\vartheta$. Note that $\dim\Lg_0=\dim\Lt^\vartheta=n_\vartheta$. Since $\on{dim}\Lg_1=\on{dim}\Lg_0+1$, the claim on $\Lg_1$ follows. 

Since $T^\vartheta$ is connected (as $G$ is simply connected) and $\dim\Lg_0=\dim\Lt^\vartheta$, it follows that
$
K=T^\vartheta.
$
\end{proof}

We will make use of  the identification 
 \beq\label{identification-twisted}
 \Lg_1\cong\bC^{n_\vartheta+1},\ \sum_{i=1}^{n_\vartheta}a_iY_i+a_0Y_0\mapsto (a_1,\ldots,a_{n_{\vartheta}},a_0).
 \eeq
 Thus, in the $a_i$-coordinates the torus $T^\vartheta$ acts on $\Lg_1$ by the formula 
\beqn
 t.(a_1,\ldots,a_{\nu_G},a_0) =(\alpha_1(t) a_1,\ldots,\alpha_{n_\vartheta}(t) a_{n_\vartheta},\beta_0(t^{-1}) a_0),
\eeqn
where $\alpha_i$'s are defined in~\eqref{simple roots reps} (note that the above formula does not depend on the choice of $\alpha_i$'s). One then readily checks that 
\beq
\label{outer invariant}
\bC[\Lg_1]^K\cong\bC[f_0], \ \ f_0=a_0\prod_{i=1}^{n_\vartheta}a_i^{m_i},
\eeq 
where the $m_i$, $i=1,\ldots,n_\vartheta$, are defined by the equation
\beq\label{mis}
\beta_0|_{T^\vartheta}=\sum_{i=1}^{n_\vartheta}m_i\alpha_i|_{T^\vartheta}.
\eeq
Note that $m_i$'s are well-defined. We remark that the powers appearing in the invariant polynomial $f_0$ are exactly the labels in the twisted affine Dynkin diagram. 

Under the identification~\eqref{identification-twisted}, we have
\beqn
\Lg_1^{rs}=\{\sum_{i=1}^{n_\vartheta}{a_i}Y_{i}+a_0Y_0\mid a_i\neq 0, i=0,1,\ldots,n_\vartheta\}\cong(\bC^*)^{n_\vartheta+1}.
\eeqn
and so the nilpotent cone $\cN_1$ is the normal crossing divisor given by 
\beqn
\cN_1 \ = \ \bigcup_{i=0}^{n_\vartheta} \{(a_1,\ldots,a_{n_\vartheta},a_0)\in \fg_1\mid a_i=0  \}\,.
\eeqn
\begin{lemma}
We have
\beqn
I=Z(G)^\vartheta\text{ and }W_\fa\cong\bZ/\frac{m}{e}\bZ.
\eeqn
\end{lemma}
\begin{proof}
We choose a Cartan subspace $\fa=\bC(\sum_{i=1}^{n_\vartheta}Y_i+Y_0)$. Then we have
\beqn
\bega
N_{K}(\fa)=\{t\in T^\vartheta\,|\,\alpha_1(t)=\cdots=\alpha_{n_{\vartheta}}(t),\ \alpha_1(t)^{\frac{m}{e}}=1\}\\
Z_{K}(\fa)=\{t\in T^\vartheta\,|\,\alpha_1(t)=\cdots=\alpha_{n_{\vartheta}}(t)=1\}=Z(G)^\vartheta.
\eega
\eeqn
The lemma follows.
\end{proof}

The element $E$ of~\eqref{E} can be written as  $E= \sum_{\alpha \in \Delta} X_\alpha=\sum_{i=1}^{n_\vartheta} Y_i$. In our coordinates it corresponds to the point $(1, \dots, 1,0)$. Thus, we can choose the Kostant slice $\Ls$ to be
\beq
\label{rank one Kostant outer}
\Ls \ = \ \{(1, \dots, 1, y)\mid y\in \bC\}\,.
\eeq
Recall the top row in the diagram~\eqref{mainDiagram}:
\beqn
1 \to I=Z(G)^\vartheta \to \pi_1^{T^\vartheta}(\Lg_1^{rs}) \to \pi_1(\fa^{rs}/W_\fa) = \bZ \to 1 \,.
\eeqn
We split this exact sequence by the Kostant slice and obtain an isomorphism
\begin{equation*}
\pi_1^{T^\vartheta}(\Lg_1^{rs}) \   \cong \ \pi_1(\fa^{rs}/W_\fa) \oplus Z(G)^\vartheta = \bZ  \oplus Z(G)^\vartheta. 
\end{equation*}
As in the previous subsection we will make the description of $\pi_1^{T^\vartheta}(\Lg_1^{rs})$  very explicit. 
Let us write $\gamma_i$ for the generators of $\pi_1(\fg_1^{rs})$ given by loops around the axes $a_i=0$. 

Consider the map $T^\vartheta\to \Lg_1^{rs}\cong(\bC^*)^{n_\vartheta+1}$
\beqn
t\mapsto t.(1,\ldots,1)=(\alpha_1(t),\ldots,\alpha_{n_\vartheta}(t),\beta_0(t^{-1})).
\eeqn
This induces a map
\beqn\label{psi1}
\bega
\Psi_1:\pi_1(T^\vartheta)\cong X_*(T)^\vartheta\to\pi_1(\Lg_1^{rs})\cong\bZ^{n_\vartheta}=(\prod_{i=1}^{n_\vartheta}\bZ\gamma_i)\times\bZ\gamma_0\\
\lambda\mapsto(\langle\lambda,\alpha_1\rangle\gamma_1, \langle\lambda,\alpha_2\rangle\gamma_2,\ldots,\langle\lambda,\alpha_{n_\theta}\rangle\gamma_{n_\vartheta},\langle\lambda,-\beta_0\rangle\gamma_0).
\eega
\eeqn
By definition $\pi_1^{T^\vartheta}(\Lg_1^{rs}) $ is the cokernel of $\Psi_1$, i.e.,  we have an exact sequence
\beqn
0 \to X_*(T)^\vartheta \xrightarrow{\Psi_1} \pi_1(\Lg_1^{rs})=(\prod_{i=1}^{n_\vartheta}\bZ\gamma_i)\times\bZ\gamma_0 \to\pi_1^{T^\vartheta}(\Lg_1^{rs})  \to 0\,.
\eeqn
The Kostant slice of~\eqref{rank one Kostant outer} gives us an isomorphism $\Ls^{reg} \cong \fa^{rs}/W_\fa$ and hence an identification $\pi_1(\Ls^{reg} ) \cong \pi_1(\fa^{rs}/W_\fa) = \bZ$. Thus it allows us to identify the element $\gamma_0$ with the element $1\in  \bZ =  \pi_1(\fa^{rs}/W_\fa)$.

We also note  that because of the formula~\eqref{outer invariant} the map 
\begin{subequations}
\label{E2}
\beq
\pi_1(\Lg_1^{rs})=(\prod_{i=1}^{n_\vartheta}\bZ\gamma_i)\times\bZ\gamma_0 \to \pi_1^{{T^\vartheta}}(\Lg_1^{rs}) \to  \pi_1(\fa^{rs}/W_\fa) = \bZ
\eeq
is given by 
\beq
(b_1, \dots, b_{n_\vartheta},b_0) \mapsto b_0 + \sum_{i=1}^{n_\vartheta} m_ib_i \,.
\eeq
\end{subequations}
Applying $\vartheta$ to the exact sequence~\eqref{center}
we obtain an exact sequence
\beqn
1  \longrightarrow X_*(T)^\vartheta \xrightarrow{(\alpha_1,\ldots,\alpha_{n_\vartheta})} \prod_{i=1}^{n_\vartheta}\bZ\gamma_i \longrightarrow Z(G)^\vartheta=I\longrightarrow 1\,.
\eeqn
Putting this all together we have
\beqn
\xymatrix{0\ar[r]&X_*(T)^\vartheta\ar[r]^-{(\alpha_1,\ldots,\alpha_{n_\vartheta})}\ar@{=}[d]&\prod_{i=1}^{n_\vartheta}\bZ\gamma_i\ar[r]^{\overline\Pi}&I=Z(G)^\vartheta\ar[r]&0
\\
0\ar[r]&X_*(T)^\vartheta\ar[r]^-{\Psi_1}&\pi_1(\Lg_1^{rs})=(\prod_{i=1}^{n_\vartheta}\bZ\gamma_i)\times\bZ\gamma_0\ar[r]^-{\Pi}\ar@{->>}[u]_{\text{pr}_1}&\pi_1^{T^\vartheta}(\Lg_1^{rs})=\bZ\oplus Z(G)^\vartheta\ar[r]\ar[d]^{\tilde q}\ar[u]_{\text{pr}_2}&0
\\
&&&\pi_1(\fa^{rs}/W_\fa)\cong\bZ\ar[ul]^{\kappa}&}
\eeqn
The Kostant slice $\kappa$ is given by  $\kappa(1)=\gamma_0$. Thus, the map $\text{pr}_2$ is induced by the canonical projection $\text{pr}_1$.  Let us write
\beq\label{def of gammai in twisted case}
\bar\gamma_i = \overline\Pi(\gamma_i)\in Z(G)^\vartheta,\,i=1, \dots , {n_\vartheta}.
\eeq  Then we have, making use of~\eqref{E2}, that
\beqn
\label{generators-outer}
\Pi(\gamma_0)=(1,\bar 1)\in \bZ\oplus Z(G)^\vartheta \qquad  \Pi(\gamma_i)= (m_i,\bar\gamma_i )\in \bZ\oplus Z(G)^\vartheta\ \  \text{for} \ \ i=1, \dots , i_{n_\vartheta}, 
\eeqn
where $\bar 1$ denotes the identity element in $Z(G)^\vartheta$ and $m_i$'s are defined in~\eqref{mis} (see also~\eqref{outer invariant}).

 In the following theorem we describe the structure of $\cM_\chi$ as a representation of $ \pi_1^K(\Lg_1^{rs}) = \pi_1(\fa^{rs}/W_\fa)\oplus Z(G)^\vartheta=\bZ  \oplus Z(G)^\vartheta$. To that end let us denote the action of $1\in \bZ = \ \pi_1(\fa^{rs}/W_\fa)$ on $\cM_\chi$ by $x$.
 Using the same argument  as in the previous subsection we obtain
\begin{theorem}\label{thm-twisted coxeter}We have
\begin{equation*}
\cM_\chi \ = \  \bC_\chi\otimes\big(\bC[x]/R_\chi(x)\big)\,,\ R_\chi(x)=(x-1)\prod_{i=1}^{n_\vartheta} ( \chi(\bar\gamma_i)x^{m_i}-1)\,,
\end{equation*}
where  the $m_i$ are the exponents of the $a_i$ in the invariant $f_0$ (see~\eqref{mis}) and $\bar\gamma_i$ are defined in~\eqref{def of gammai in twisted case}. 

\end{theorem}
The polynomials $R_\chi$ are calculated explicitly in each type in~\S\ref{cal-tcox}.

\section{Endoscopic interpretation}
\label{endoscopy}

In the main body of this paper we make use of a particular reflection subgroup $W_{\fa,\chi}^0$ of $W_{\fa,\chi}$ which was introduced in Definition~\ref{def 0}. In this section we will define another reflection subgroup  $W_{\fa,\chi}^{en}$ of $W_{\fa,\chi}$ which arises naturally from the endoscopic point of view we now describe. The results from this section are not used in the main body of the text to prove any statements. However, the point of view of endoscopy has guided much of the project so we include its discussion here. In the constructions we make we do not necessarily stay within our standing simplifying assumption that $G$ is simply connected. In section~\ref{Isogenies} we have explained how to avoid this hypothesis.

Let us write $\check G$ for the dual group of $G$ and write $\Tds\subset \check G$ for the dual torus of $\Ts$. We have $ \Tds = X_*(\Tds)\otimes_\bZ\bC^*=X^*(\Ts)\otimes_\bZ \bC^*$ and hence the automorphism $\theta$ induces an automorphism $\check\theta:\Tds\to \Tds$. Recall that $I=(\Ts)^\theta$. We claim that
\beq\label{hat I}
\hat I=\on{Hom}(I,\bC^*)\cong(\Tds)^{\check\theta}.
\eeq
We can argue as follows. Let us consider the exact sequence
\beqn
1 \to I \to \Ts \xrightarrow{t\mapsto (\theta t)t^{-1}} \Ts \to 1. 
\eeqn
We now apply $\on{Hom}(- ,\bC^*)$ to this sequence to obtain
\beqn
1 \leftarrow  \hat I  \leftarrow X^*(\Ts )\xleftarrow{ (\theta \phi)-\phi\mapsfrom \phi} X^*(\Ts ) \leftarrow 1. 
\eeqn
Applying the functor $-\otimes_\bZ \bC^*$ to this sequence we obtain
\beqn
1 \to \oh_1(\hat I \overset{L}\otimes_\bZ \bC^*) \to \Tds \xrightarrow{\check t\mapsto (\check\theta \check t)\check t^{-1}} \Tds \to 1. 
\eeqn
Making use of the flat resolution
\beqn
0 \to \bZ \to \bC \to \bC^*\to 1
\eeqn
of $\bC^*$ we conclude that we have an isomorphism $\oh_1(\hat I \overset{L}\otimes_\bZ \bC^*)\cong \hat I$. Thus we have established~\eqref{hat I}.

We extend $\check\theta:\Tds\to  \Tds$ to an automorphism of $\check G$ as follows. Recall the pinned automorphism $\vartheta_{\mathbf{s}}$ of $G$ defined by the pinning $(\Ts,B_{\mathbf{s}},\{x_{\alpha}\})$. Consider the pinning $(\Tds,\check B_{\mathbf{s}},\{\check x_{\alpha}\})$ which is well-defined up to inner automorphisms by elements of $\Tds$.  This gives rise to a pinned automorphism $\check\vartheta_{\mathbf{s}}:\check G\to\check G$.
Now the element $n_w\in N_G(\Ts)$ given by equation~\eqref{theta-new} gives rise to an element $w\in W=N_G(\Ts)/\Ts=N_{\check G}(\Tds)/\Tds$. We lift $w^{-1}$ to an element $\check n_w\in N_{\check G}(\Tds)$ and define
\beqn
\check\theta=\on{Int}(\check n_w)\circ\check\vartheta_{\mathbf{s}}:\check G\to\check G.
\eeqn
According to \cite[\S4.1]{RLYG}, the $\check G$-conjugacy class of $\check\theta$ does not depend on the choice of the lifting $\check n_w$. 
\begin{remark}
Note that we use the inverse $w^{-1}$ instead of $w$ as we have defined $\check\theta$ without inverting $\theta$. 
\end{remark}

As we have assumed that $G$ is simply connected, the group $\check G$ is adjoint. Now, by~\cite[Corollary 14]{RLYG} we conclude that the grading induced by $\check\theta$ on $\check\fg$ is also stable. The restriction of $\check \theta$ to $\check \ft_{\mathbf s} = \ft_{\mathbf s}^*$ is given by the adjoint $\theta^*$. By~\cite[Proposition 6]{Vin} we conclude that we can regard $\fa^*$ as a Cartan subspace of $\check \fg_1$ and $\Tds$ is a split torus on the dual side.

Let $\chi\in \hat I$. Via~\eqref{hat I} we view $\chi$ as an element in $\check G$. We define
\beqn
\check G(\chi)=\check G^{\on{Int}\chi},\ \ \check G(\chi)^0=(\check G^{\on{Int}\chi})^0.
\eeqn
Note that $\Tds\subset \check G(\chi)$ so $\Tds$  is a maximal torus of $\check G(\chi)^0$. Moreover, the root datum of $\check G(\chi)^0$ is given by $(X^*(\Tds), \check\Phi_\chi, X_*(\Tds), \Phi_\chi) $ where
\beqn
\check\Phi_\chi=\{\check\alpha\in\check\Phi\mid\check\alpha(\chi)=1\}
\eeqn
and $\Phi_\chi$ is given by the corresponding subset of $\Phi$. 

Since $\check\theta(\chi)=\chi$ we see that $\on{Int}\chi$ and $\check\theta$ commute and hence the 
subgroup $\check G(\chi)\subset \check G$ is $\check\theta$-stable.  The $\check \theta$ is an automorphism of order $m$ on $\check G(\chi)^0$ and it induces a stable grading on $\check \fg(\chi)$ with Cartan subspace $\fa^*$. 

We first observe that the Weyl group 
$$
W(\check G(\chi), \Tds) = W_\chi=\{w\in W\mid w\chi=\chi\}
$$ by definition. On the other hand, by our stable grading assumption, the zero eigenspace $\check \ft_0$ of $\check \theta$ on $\check \ft$ is trivial and so an easy argument as in~\cite[Proposition 19]{Vin} shows that 
$$
W(\check G(\chi)^\theta, \fa^*) = W_\chi^\theta=W_\chi\cap W^\theta\,.
$$
Recall that the little Weyl group $W_\fa = W^\theta$ and hence
\beq
\label{stablizer}
W_{\fa,\chi} = W_\chi\cap W^\theta = W(\check G(\chi)^\theta, \fa^*)\,.
\eeq
We now define
\beq
W_{\fa,\chi}^{en}=W((\check G(\chi)^{\check\theta})^0,\fa^*),
\eeq
the little Weyl group associated to $(\check\theta,\check G(\chi)^0)$. 
Then $W_{\fa,\chi}^{en}$ is a complex reflection group. 

Let $s\in W_\fa$ be a distinguished reflection. Recall the reduction to (semisimple) rank one we performed in subsection~\ref{subsec-regular-splitting} where we set $G_s=Z_G(\fa_s)$. We write $\theta_s = \theta|_{G_s}$ and observe that $\Ts$ is also a maximal torus of $G_s$ and so the groups $G_s$  and $G$ have the same $I=\Ts^\theta$. We set $K_s=G_s^{\theta_s}$. Note that contrary to what we did in subsection~\ref{subsec-regular-splitting} we do not pass to $K_s^0$ at this stage. Recall that we have, by saturation, $W_{\fa,s}= N_{K^0_s} (\La) / Z_{K^0_s} (\La) =N_{K_s}(\La)/Z_{K_s}(\La)$ so the little Weyl group does not change if we pass from $K_s$ to $K_s^0$.

 We can also make this construction on the dual side by setting
\beqn
(\check G)_s \ = \ Z_{\check G}(\fa_s^*)\,.
\eeqn
Note that we have $\widecheck{G_s} = (\check G)_s$ as can be easily seen by looking at the root datum, where $\widecheck{G_s}$ is the dual group of $G_s$. Furthermore, we see that
\beqn
\widecheck{G_s}(\chi) \ = \ (\check G(\chi))_s\,,
\eeqn
where $\widecheck{G_s}(\chi)=\widecheck{G_s}\cap\check G(\chi)$. 
We can think of this situation diagrammatically as follows:
\beqn
\bega
\xymatrix{
(G,\theta, \chi) \ar@{~>}[r]\ar@{~>}[d]& (\check G(\chi),\check \theta)\ar@{~>}[r] \ar@{~>}[d]&(\check G(\chi)^0,\check \theta)\ar@{~>}[d]
\\
(G_s,\theta_s, \chi) \ar@{~>}[r]&  (\check G(\chi)_s,\check \theta_s)\ar@{~>}[r]&  (\check G(\chi)^0_s,\check \theta_s)\,.}
\eega
\eeqn
Applying formula~\eqref{stablizer} to the bottom row we obtain
$$
W_{\fa,s,\chi}=W_{\fa,s}\cap W_{\fa,\chi}=W(\check G(\chi)_s^{\check\theta_s}, \fa^*)\,.
$$
Thus, we see that
\beqn
W_{\fa,s,\chi}^{en}=W((\check G(\chi)_s^{\check\theta_s})^0, \fa^*) \subset W(\check G(\chi)_s^{\check\theta_s}, \fa^*) =W_{\fa,s,\chi}\,.
\eeqn
Recall that the Weyl group $W_{\fa,s}=\langle s\rangle$ and that we have 
set 
\beqn
e_s=|W_{\fa,s}|/|W_{\fa,s,\chi}|\,.
\eeqn
Furthermore, we have defined another complex reflection group $W_{\fa,\chi}^{0}\subset W_{\fa,\chi}$ (see Definition~\ref{def 0}) to be the subgroup generated by $W_{\fa,s,\chi}=W_{\fa,s}\cap W_{\fa,\chi}=\langle s^{e_s}\rangle$. On the other hand, the group $W_{\fa,\chi}^{en}=W(\check (G(\chi)^{\check\theta})^0,\fa^*)$ is generated by $W_{\fa,s,\chi}^{en}$ and so we conclude
\beqn
W_{\fa,\chi}^{en} \subset W_{\fa,\chi}^{0}\,.
\eeqn
Let
\beq
d_{s}=|W_{\fa,s}/W_{\fa,s,\chi}^{en}|.
\eeq
Then, we have $e_s | d_s$. We now make the following hypothesis:
\beqn
\text{there is a polynomial $\bar{\bar{R}}_{\chi,s}\in\bC[x]$ such that $R_{\chi,s}(x)=\bar{\bar{R}}_{\chi,s}(x^{d_{s}})$}\,.
\eeqn
When this hypothesis is satisfied then we can define a Hecke algebra $\cH_{W_{\fa,\chi}^{en}}$ using the polynomials $\bar{\bar{R}}_{\chi,s}$ and we can run through the arguments in~\cite{GVX2} to conclude that 
\beqn
\cM_\chi \ = \ \left( \bC [\widetilde B_{W_\fa}]
\otimes_{\bC[\widetilde B_{W_{\fa}}^{\chi,en}]}
(\bC_\chi \otimes  \cH_{W_{\fa,\chi}^{en}}) \right)\otimes \bC_\tau \, ,
\eeqn
where $\widetilde B_{W_{\fa}}^{\chi,en}=\tilde q^{-1}(B_{W_{\fa}}^{\chi,en})$ and $B_{W_{\fa}}^{\chi,en}=p^{-1}(W_{\fa,\chi}^{en})$. In particular, we have
\beqn
\bC[B_{W_{\fa}}^{\chi,0}]\otimes_{\bC[B_{W_{\fa}}^{\chi,en}]}\cH_{W_{\fa,\chi}^{en}}\cong\cH_{W_{\fa,\chi}^0}\,.
\eeqn
We expect the following to hold 
\begin{subequations}
\beq
\label{min-mono}
\begin{gathered}
\text{The polynomial $\bar{\bar{R}}_{\chi,s}$ is the monondromy polynomial attached to}
\\
\text{the pair $(\check G(\chi)^0_s,\check \theta_s)$ and the trivial character of $I(\check G(\chi)^0_s,\check \theta_s)$}
\end{gathered}
\eeq
and
\beq\label{mono-2}
\text{$d_{s}$ is the largest integer such that there exists $g\in\bC[x]$ with $R_{\chi,s}(x)=g(x^{d_s})$}.
\eeq
\end{subequations}

We will verify in section~\ref{sec-stable rank 1 explicit} the expectations~\eqref{min-mono} and~\eqref{mono-2} for the rank 1 stable automorphisms considered in section~\ref{sec-stable rank 1}. We state it as
\begin{theorem}\label{thm-expectations}
The statements ~\eqref{min-mono} and~\eqref{mono-2} hold except for some exceptional types. 
\end{theorem}
We note that we fully expect the theorem to be true in general but have not checked all exceptional cases. 
\begin{remark}
We do not know of a uniform proof for the above theorem. 
\end{remark}

\section{Stable rank 1 gradings: explicit calculations}\label{sec-stable rank 1 explicit}
In this section we determine the polynomials $R_{\chi}(x)$ (normalised to be monic) in Theorem~\ref{coxeter rank one} and Theorem~\ref{thm-twisted coxeter} in each type explicitly. We will make use of the results in this section to determine the $\cM_\chi$ for classical types in Section~\ref{sec-classical stable}. We also describe the endoscopic point of view discussed in section~\ref{endoscopy} and prove Theorem~\ref{thm-expectations} in each case.

Throughout the rest of the paper we label the set $\Delta$ of simple roots of each type as follows.
\bern
&\Delta=\{\alpha_i=\epsilon_i-\epsilon_{i+1},\, 1\leq i\leq n\}&\text{ type }A_n\\
&\Delta=\{\alpha_i=\epsilon_i-\epsilon_{i+1},\, 1\leq i\leq n-1,\,\alpha_n=\epsilon_n\}&\text{ type }B_n\\
&\Delta=\{\alpha_i=\epsilon_i-\epsilon_{i+1},\, 1\leq i\leq n-1,\,\alpha_n=2\epsilon_{n}\}&\text{ type }C_n\\
&\Delta=\{\alpha_i=\epsilon_i-\epsilon_{i+1},\, 1\leq i\leq n-1,\,\alpha_n=\epsilon_{n-1}+\epsilon_n\}&\text{ type }D_n.
\eern
 In type $E_6$ (resp. $E_7$), the Dynkin diagram is as follows
\beqn
\xymatrix{{\substack{\alpha_1\\\circ}}\ar@{-}[r]&{\substack{\alpha_3\\\circ}}\ar@{-}[r]&{\substack{\alpha_4\\\circ}}\ar@{-}[r]\ar@{-}[d]&{\substack{\alpha_5\\\circ}}\ar@{-}[r]&{\substack{\alpha_6\\\circ}}\\&&{\substack{\circ\\\alpha_2}}&&}\text{ (resp. }\xymatrix{{\substack{\alpha_1\\\circ}}\ar@{-}[r]&{\substack{\alpha_3\\\circ}}\ar@{-}[r]&{\substack{\alpha_4\\\circ}}\ar@{-}[r]\ar@{-}[d]&{\substack{\alpha_5\\\circ}}\ar@{-}[r]&{\substack{\alpha_6\\\circ}}\ar@{-}[r]&{\substack{\alpha_7\\\circ}}\\&&{\substack{\circ\\\alpha_2}}&&})
\eeqn

Recall $\mathbf{i}=\sqrt{-1}\in\bC$. For $k\in\bZ_+$, we define the following $k\times k$ diagonal matrix
\beq\label{the matrix Jk}
J_k:=\begin{cases}\on{diag}(1,-1,1,-1,\ldots,1,-1)&\text{if $k$ is even}\\
\on{diag}(\mathbf{i},-\mathbf{i},\mathbf{i},-\mathbf{i},\ldots,\mathbf{i})&\text{if $k$ is odd}.\end{cases}
\eeq

\subsection{The Coxeter case}\label{cal-cox} We use the notations from subsection~\ref{sec-cox}. In this case we can assume (see~\cite[section 7]{RLYG}) $$\theta=\on{Int}{n_h}$$ where $n_h\in N_G(\Ts)$ is a representative of $w_h\in W=N_G(\Ts)/\Ts$, a Coxeter element. 
\subsubsection{Type $A_{N-1}$}\label{cal-cox-A}Suppose that $G=SL(N)$. We have 
 $$m=N,\ \ I=Z(G)=\langle z=\prod_{j=1}^{N-1}\check\alpha_j(e^{2\pi\mathbf{i}\frac{N-j}{N}})\rangle\cong\mu_N,\ \ n_i=1,\,i=1,\ldots,N-1\,.$$  Identifying 
 $$\gamma_i=\check\omega_i=\frac{1}{N}\big((N-i)\sum_{j=1}^{i-1}\check\alpha_j+i\sum_{j=i}^{N-1}(N-j)\check\alpha_j\big),$$ we see that
\beqn
\bar\gamma_k=z^k\in Z(G),\ 1\leq k\leq N-1.
\eeqn
Let $\chi\in\hat I$. Then 
$
R_\chi(x)=(x-1)\prod_{k=1}^{N-1}\left(\chi(z)^kx-1\right).
$
Thus
\beq\label{mono-type A inner}
\text{ $R_{\chi}(x)=(x^{\frac{N}{d}}-1)^d$ if $\chi(z)$ is a primitive ${N}/{d}$-th root of $1$, where  $d|N$.}
 \eeq

We take $w_h=t_{\alpha_1}t_{\alpha_2}\cdots t_{\alpha_{N-1}}:=s$. Then $W_\fa=\langle s\rangle \cong\bZ/N\bZ.$ Let $\chi\in \hat I$. Suppose $\chi(z)=\zeta_{N/d}$, a primitive $\frac{N}{d}$-th root of $1$, where  $d|N$. From the endoscopic point of view, $\chi$ corresponds to the following element (see~\eqref{hat I})
 \beqn
 \chi=\on{diag}(1,\zeta_{N/d},\zeta_{N/d}^2,\ldots,\zeta_{N/d}^{N-1})\in \check G=PGL(N).
 \eeqn
Note that $\check n_h\in \check G(\chi)$ and $\check n_h^{N/d}\in\check G(\chi)^0$, where $\check n_h\in\check G$ is a representative of $w_h$. Then 
 \beqn
 \check G(\chi)/\check G(\chi)^0=\langle\overline{\check n_h}\rangle\cong\bZ/(N/d)\bZ,\,\check G(\chi)^0\cong P(GL(d)\times GL(d)\times\cdots GL(d))
 \eeqn 
 and
  \beqn
 W_{\fa,\chi}^{en}=\langle s^{N/d}\rangle\cong\bZ/d\bZ.
 \eeqn
 Moreover, $\check\theta|_{\check G(\chi)^0}$ permutes the $N/d$ factors of $GL(d)$ and it can be identified with an order $d$ stable (inner) automorphism of type $A_{d-1}$ (see~\eqref{reduction-prod}). Thus Theorem~\ref{thm-expectations} in this case follows from~\eqref{mono-type A inner}.

\subsubsection{Type $B_n$}Suppose that $G=Spin(2n+1)$. We have  
$$m=2n,\ I=Z(G)=\langle z=\check{\alpha}_n(-1)\rangle\cong\mu_2,\text{ and }n_1=1,\,n_i=2,\,2\leq i\leq n.$$ 
Identifying $$\gamma_i=\check\omega_i=\sum_{j=1}^{i-1}j\check\alpha_j+i\sum_{j=i}^{n-1}\check\alpha_j+\frac{i}{2}\check\alpha_n,$$ we obtain
\beqn
\bar\gamma_k=z^k\in Z(G),\ 1\leq k\leq n.
\eeqn
Let $\chi\in\hat I$. Then 
\begin{subequations}
\beq\label{mono-type B-trivial}
R_\chi(x)=
(x-1)^{n+1}(x+1)^{n-1},\text{ if }\chi(z)=1
\eeq
\beq\label{mono-type B-nontrivial}R_\chi(x)=(x^2-1)^{[\frac{n}{2}]+1}(x^2+1)^{[\frac{n-1}{2}]},\text{ if }\chi(z)=-1.
\eeq
\end{subequations}

We take $w_h=t_{\alpha_1}t_{\alpha_2}\cdots t_{\alpha_n}:=s$. We have
$
W_\fa=\langle s\rangle \cong\bZ/2n\bZ.
$
From the endoscopic point of view, the non-trivial character $\chi$ corresponds to the following element
\beqn
\chi=\on{diag}(J_n,-J_n)\in \check G=PSp(2n),
\eeqn
where $J_n$ is defined in~\eqref{the matrix Jk}. It follows that
\beqn
\check G(\chi)/\check G(\chi)^0\cong \bZ/2\bZ,\ \check G(\chi)^0\cong\begin{cases} P(Sp(n)\times Sp(n))&\text{ when $n$ is even},\\ P(GL(n)\times GL(n)^*)&\text{ when $n$ is odd.} \end{cases}
\eeqn
and
\beqn
W_{\fa,\chi}^{en}=\langle s^2\rangle\cong\bZ/n\bZ.
\eeqn
Suppose $n$ is even, then $\check\theta|_{\check G(\chi)^0}$ permutes the two factors and can be identified with an order $n$ stable automorphism of type $C_{n/2}$. Thus Theorem~\ref{thm-expectations} follows in this case from~\eqref{mono-type B-nontrivial} and~\eqref{type C-trivial} (see \S\ref{ssec-cc}). Suppose $n$ is odd, then $\check\theta|_{\check G(\chi)^0}$ is an order $2n$ stable (outer) automorphism of  type $A_{n-1}^2$. Thus Theorem~\ref{thm-expectations} follows in this case from~\eqref{mono-type B-nontrivial} and~\eqref{type A2-odd-1} (see~\S\ref{ssec-tcae}).

\subsubsection{Type $C_n$}\label{ssec-cc}
Suppose that $G=Sp(2n)$. We have $$m=2n,\ I=Z(G)=\langle z=\prod_{j=0}^{[\frac{n-1}{2}]}\check{\alpha}_{2j+1}(-1)\rangle\cong\mu_2\text{ and }n_i=2,\,1\leq i\leq n-1,\,n_n=1.$$ 
Identifying 
\beqn
\text{$\gamma_i=\check\omega_i=\sum_{j=1}^{i-1}j\check\alpha_j+i\sum_{j=i}^{n}\check\alpha_j$, $i=1,\ldots,n-1$, $\gamma_n=\check\omega_n=\frac{1}{2}\sum_{i=1}^ni\check\alpha_i$,}
\eeqn
 we obtain
\beqn
\bar\gamma_k=\bar 1,\ \,1\leq k\leq n-1,\,\bar\gamma_n=z\in Z(G).
\eeqn
Let $\chi\in\hat I$. Then 
\begin{subequations}
\beq\label{type C-trivial}
R_\chi(x)=
\displaystyle{(x-1)^{n+1}(x+1)^{n-1}},\ \text{ if }\chi(z)=1
\eeq
\beq\label{type C-nontrivial}
R_\chi(x)=\displaystyle{(x^2-1)^{n}},\text{ if }\chi(z)=-1.
\eeq
\end{subequations}

We take $w_h=t_{\alpha_1}t_{\alpha_2}\cdots t_{\alpha_n}:=s$. We have
$
W_\fa=\langle s\rangle \cong\bZ/2n\bZ.
$
From the endoscopic point of view, the non-trivial character $\chi$ corresponds to the element $$\chi=\on{diag}(-I_n,1,-I_n)\in SO(2n+1).$$ Thus 
$$\check G(\chi)=S(O(2n)\times O(1)),\, \check G(\chi)^0=SO(2n)\ \text{ and }\ W_{\fa,\chi}^{en}=\langle s^2\rangle\cong\bZ/n\bZ.$$ Moreover, $\check\theta|_{\check G(\chi)^0}$ is an order $2n$ stable (outer) automorphism of type $D_n^2$. Thus Theorem~\ref{thm-expectations} follows in this case from~\eqref{type C-nontrivial}
 and~\eqref{type D2-trivial} (see~\S\ref{ssec-tcd}).

\subsubsection{Type $D_n$}\label{ssec-cox-D}Suppose that $G=Spin(2n)$.  Let  
\beqn
\bega
z_1=\begin{cases}\prod_{j=1}^{n/2}\check\alpha_{2j-1}(-1),&\text{ if $n$ is even}\\\check\alpha_{n-1}(\mathbf{i})\check\alpha_n(-\mathbf{i})\prod_{j=1}^{(n-1)/2}\check\alpha_{2j-1}(-1)&\text{ if $n$ is odd,}\end{cases}
\text{ and $z_2=\check\alpha_{n-1}(-1)\check\alpha_n(-1)$.}
\eega 
\eeqn
 We have 
 \beqn
m=2n-2,\ I=\langle z_1,z_2\rangle\cong
\begin{cases}\mu_4&\text{ if $n$ is odd}\\
\mu_2\times \mu_2&\text{ if $n$ is even}\,;\end{cases}\text{ and }n_i=\begin{cases}1&i=1,n-1,n\\2&2\leq i\leq n-2\,.\end{cases}
 \eeqn
We identify 
\bern
&&\gamma_i=\check\omega_i=\sum_{j=1}^{i-1}j\check\alpha_j+i\sum_{j=i}^{n-2}\check\alpha_j+\frac{1}{2}i(\check\alpha_{n-1}+\check\alpha_n),\, i=1,\ldots,n-2,\\
&&\gamma_{n-1}=\check\omega_{n-1}=\frac{1}{2}\sum_{i=1}^{n-2}i\check\alpha_i+\frac{1}{4}(n\check\alpha_{n-1}+(n-2)\check\alpha_n),\\
&&  \gamma_{n}=\check\omega_{n}=\frac{1}{2}\sum_{i=1}^{n-2}i\check\alpha_i+\frac{1}{4}((n-2)\check\alpha_{n-1}+n\check\alpha_n).
\eern

Assume that $n$ is odd. We have
\beqn
\bar\gamma_k=z_1^{2k},\ \,1\leq k\leq n-2,\,\bar\gamma_{n-1}=z_1^n,\,\bar\gamma_n=z_1^{n-2}\in Z(G).
\eeqn
Let $\chi\in\hat I$. Then 
\begin{subequations}
\beq\label{type D-odd-trivial}
R_\chi(x)=
\displaystyle{(x-1)^{n+1}(x+1)^{n-3}}\text{ if }\chi(z_1)=1
\eeq
\beq\label{type D-odd-order 4}
R_\chi(x)=
\displaystyle{(x^4-1)^{\frac{n-1}{2}}}\text{ if }\chi(z_1)=\mathbf{i}\text{ or }\mathbf{-i}
\eeq
\beq\label{type D-odd-order 2}
R_\chi(x)=
\displaystyle{(x^2-1)^{n-1}}\text{ if }\chi(z_1)=-1.
\eeq
\end{subequations}

Assume that $n$ is even. We have
\beqn
\begin{gathered}
\bar\gamma_k=z_2^{k},\ \,1\leq k\leq n-2,\,
\{\bar\gamma_{n-1},\bar\gamma_n\}=\{z_1z_2,z_1\}.\end{gathered}
\eeqn
 Let $\chi\in\hat I$. Then 
\begin{subequations}
\beq\label{typeD-even-1}
R_\chi(x)=
\displaystyle{(x-1)^{n+1}(x+1)^{n-3}}\hspace{.2in}\text{ if }\chi(z_1)=\chi(z_2)=1\,;
\eeq
\beq\label{typeD-even-2}
\ R_\chi(x)=\displaystyle{(x^2-1)^{n-1}}\hspace{.6in}\text{ if }\chi(z_1)=-1,\ \chi(z_2)=1\,;
\eeq
\beq\label{typeD-even-3}
R_\chi(x)=\displaystyle{(x^2-1)^{\frac{n}{2}+1}(x^2+1)^{\frac{n}{2}-2}}\hspace{.8in} \text{otherwise}\,.
\eeq
\end{subequations}

We take $w_h=t_{\alpha_1}t_{\alpha_2}\cdots t_{\alpha_n}:=s$. We have
$
W_\fa=\langle s\rangle \cong\bZ/(2n-2)\bZ.
$

Suppose first that $n$ is odd. From the endoscopic point of view, the characters $\chi\in\hat I$ such that $\chi(z_1)=\mathbf{i}\text{ or }\mathbf{-i}$ correspond to the elements
\beqn
\on{diag}(J_{n-1},\pm\mathbf{i},\pm\mathbf{i}^{-1},-J_{n-1})\in\check G=PSO(2n),
\eeqn
where $J_{n-1}$ is defined in~\eqref{the matrix Jk}. 
We have $$\check G(\chi)/\check G(\chi)^0\cong\bZ/4\bZ,\,\check G(\chi)^0\cong P\big(GL(1)\times GL(1)^*\times SO(n-1)\times SO(n-1)\big)$$ 
and
\beqn
W_{\fa,\chi}^{en}=\langle s^4\rangle\cong \bZ/{ {\frac{n-1}{2}}}\bZ,\, \ d_s=4.
\eeqn
Moreover, $\check\theta|_{\check G(\chi)^0}$ is an order $n-1$ stable (outer) automorphism of type $D_{(n-1)/2}^2$. Thus Theorem~\ref{thm-expectations}  follows in this case from~\eqref{type D-odd-order 4} and~\eqref{type D2-trivial} (see~\S\ref{ssec-tcd}).

\noindent The character $\chi$ such that $\chi(z_1)=-1$ corresponds to the element 
\beqn
\chi=\on{diag}(I_{n-1},-1,-1,I_{n-1})\in\check G=PSO(2n).
\eeqn
We have 
$$\check G(\chi)/\check G(\chi)^0\cong\bZ/2\bZ,\,\check G(\chi)^0\cong P\big(SO(2n-2)\times SO(2)\big)$$
and
\beqn
W_{\fa,\chi}^{en}=\langle s^2\rangle\cong\bZ/(n-1)\bZ,\, d_s=2.
\eeqn
Moreover, $\check\theta|_{\check G(\chi)^0}$ is an order $2(n-1)$ stable (outer) automorphism  of type $D_{n-1}^2$. Thus Theorem~\ref{thm-expectations} follows in this case from~\eqref{type D-odd-order 2} and~\eqref{type D2-trivial} (see~\S\ref{ssec-tcd}).

Suppose that $n$ is even. From the endoscopic point of view, the characters $\chi\in\hat I$ such that $\chi(z_2)=-1$ correspond to the elements
\beqn
\chi=\on{diag}(J_{n-1},\pm\mathbf{i},\pm\mathbf{i}^{-1},-J_{n-1})\in\check G=PSO(2n),\eeqn
where $J_{n-1}$ is defined in~\eqref{the matrix Jk}. 
We have
$$\check G(\chi)/\check G(\chi)^0\cong\bZ/2\bZ,\,\check G(\chi)^0\cong P(GL(n)\times GL(n)^*)\text{ and }W_{\fa,\chi}^{en}=\langle s^2\rangle\cong\bZ/(n-1)\bZ,\,d_s=2.$$
Moreover $\check\theta|_{\check G(\chi)^0}$ is an order $2(n-1)$ stable (outer) automorphism of  type $A_{n-1}^2$. Thus Theorem~\ref{thm-expectations} follows in this case from~\eqref{typeD-even-3} and~\eqref{type A2-even-1} (see~\S\ref{ssec-tcao}).

\noindent The character $\chi\in\hat I$ such that $\chi(z_1)=-1$ and $\chi(z_2)=1$ corresponds to the element 
\beqn
\check\chi=\on{diag}(I_{n-1},-1,-1,I_{n-1})\in\check G=PSO(2n).
\eeqn 
We have
$$\check G(\chi)/\check G(\chi)^0\cong\bZ/2\bZ,\ \check G(\chi)^0\cong P\big(SO(2n-2)\times SO(2)\big)\text{ and }W_{\fa,\chi}^{en}=\langle s^2\rangle\cong\bZ/(n-1)\bZ.$$
Moreover $\check\theta|_{\check G(\chi)^0}$ is an order $2(n-1)$ stable (outer) automorphism of type $D_{n-1}^2$. Thus Theorem~\ref{thm-expectations} follows in this case from~\eqref{typeD-even-2} and~\eqref{type D2-trivial} (see~\S\ref{ssec-tcd}).

\subsubsection{Type $E_6$}\label{ssec-cox-e6}Let $G$ be a simply connected group of type $E_6$. We have that 
\bern
&&m=12,\,I=\langle z=\check\alpha_1(\zeta_3)\check\alpha_3(\zeta_3^2)\check\alpha_5(\zeta_3)\check\alpha_6(\zeta_3^2)\rangle\cong\mu_3\\&&n_1=n_6=1,\,n_2=n_3=n_5=2,\,n_4=3\\
&&\text{and }\bar\gamma_1=\bar\gamma_5=z,\ \bar\gamma_2=\bar\gamma_4=\bar 1,\ \bar\gamma_3=\bar\gamma_6=z^2\in Z(G). 
\eern
Let $\chi\in\hat I$. Then 
\begin{subequations}
\beq\label{poly-e6t}
R_\chi(x)=(x-1)^{3}(x^2-1)^3(x^3-1)\text{ if }\chi(z)=1\,;
\eeq
\beq\label{poly-e6nt}
R_\chi(x)=(x^3-1)^{2}(x^6-1)\text{ otherwise}\,.
\eeq
\end{subequations}
We take $
w_h=t_{\alpha_1}t_{\alpha_4}t_{\alpha_6}t_{\alpha_3}t_{\alpha_5}t_{\alpha_2}:=s.
$
Then
$
W_\fa=\langle s\rangle\cong\bZ/12\bZ.
$
Let $\omega_i\in X_*(\check T)=X^*(T)$, $1\leq i\leq 6$, be such that
\beqn
\langle\omega_i,\check\alpha_j\rangle=\delta_{i,j},\,i,j=1,\ldots,6.
\eeqn
From the endscopic point of view, the non-trivial characters $\chi\in\hat I$ correspond to the  elements
$$\check\chi=\omega_1(\zeta_3^{2i})\,\omega_3(\zeta_3^{2i})\,\omega_5(\zeta_3^i)\,\omega_6(\zeta_3^i)\in\check G,\,i=1,2. $$
Then $\check G(\chi)/\check G(\chi)^0\cong\bZ/3\bZ$ and the root system of $\check G(\chi)^0$ is 
\beqn
\langle \check\beta_1,\ldots,\check\beta_4 \rangle\cong D_4,\ \beta_1=\alpha_4,\,\beta_2=\alpha_2,\,\beta_3=\alpha_3+\alpha_4+\alpha_5,\,\beta_4=\alpha_1+\sum_{i=3}^6\alpha_i.
\eeqn
Moreover, $\check\theta|_{\check G(\chi)^0}$ is an order $12$ stable (outer) automorphism of type $D_4^3$. Thus Theorem~\ref{thm-expectations} follows in this case from~\eqref{poly-e6nt} and~\eqref{poly-d43} (see~\S\ref{ssec-tcd4}).

\subsubsection{Type $E_7$} Let $G$ be a simply connected group of type $E_7$. We have that 
\bern
&&m=18,\ I=\langle z= \check\alpha_2(-1)\check\alpha_5(-1)\check\alpha_7(-1)\rangle\cong\bZ/2\bZ\\
&&n_1=n_2=n_6=2,\,n_3=n_5=3,\,n_4=4,\,n_7=1\\
&&\text{and }\bar\gamma_1=\bar\gamma_3=\bar\gamma_4=\bar\gamma_6=\bar 1,\ \bar\gamma_2=\bar\gamma_5=\bar\gamma_7=z\in Z(G).
\eern
Let $\chi\in\hat I$. Then 
\begin{subequations}
\beq\label{poly-e7t}
R_\chi(x)=(x-1)^{2}(x^2-1)^3(x^3-1)^2(x^4-1)\text{ if }\chi(z)=1\,;
\eeq
\beq\label{poly-e7nt}
R_{\chi}(x)=(x^2-1)^{2}(x^4-1)^2(x^6-1)\text{ otherwise}\,.
\eeq
\end{subequations}
 
From the endoscopic point of view, one checks similarly as in type $E_6$ that the the endoscopy group $\check G(\chi)^0$ for the  non-trivial character $\chi$ is of type $E_6$. Moreover, $\check\theta|_{\check G(\chi)^0}$ is an order $18$ stable (outer) automorhism of type $E_6^2$. Thus Theorem~\ref{thm-expectations} follows in this case from~\eqref{poly-e7nt} and~\eqref{poly-tce6} (see~\S\ref{ssec-tce6}).

 \subsubsection{Type $G_2$, $F_4$, and $E_8$}In these cases $I=1$.
 We have 
\ber
&&R(x)=
(x-1)(x^2-1)(x^3-1)\ \ \text{ type }G_2\,;\\
&&R(x)=(x-1)(x^2-1)^2(x^3-1)(x^4-1)\ \ \text{ type }F_4\,;\\
&&R(x)=(x-1)(x^2-1)^{2}(x^3-1)^2(x^4-1)^2(x^5-1)(x^6-1)\ \ \text{ type }E_8\,.
\eer

\subsection{The twisted Coxeter case}\label{cal-tcox}We use the notations from subsection~\ref{sec-tcox}. In the twisted Coxter case, we can assume that (see~\cite[section 7]{RLYG}) 
\beq\label{theta-tcox}
\theta|_{\Ts}=\on{Int}{n_w}\circ\vartheta_{\mathbf s},
\eeq
 where $n_w\in N_G(\Ts)$ is a representative of $w=\prod_{i=1}^{n_\vartheta}s_{\alpha_i}\in W=N_G(\Ts)/\Ts$ (product in any order). Recall $\{\alpha_i,\,i=1,\ldots,n_\vartheta\}$ is a set of representations of $\vartheta$-orbits in $\Delta$.    
\subsubsection{Type $A_{2n}^2$}\label{ssec-tcae}Suppose that $G=SL(2n+1)$. We have 
\beqn
\text{$m=4n+2$, $\beta_0=\sum_{i=1}^{2n}\alpha_i$, $I=\{1\}$ and }m_k=2,\,1\leq k\leq n.
\eeqn
It follows that
\beq\label{type A2-odd-1}
R(x)=(x-1)^{n+1}(x+1)^n.
\eeq

\subsubsection{Type $A_{2n-1}^2$}\label{ssec-tcao}Suppose that $G=SL(2n)$. We have 
\bern
&&\text{$m=4n-2$, $I=\langle z=\prod_{i=1}^{n}\check\alpha_{2i-1}(-1)\rangle\cong\mu_2$,\  $\Delta_i=\{\alpha_i,\alpha_{2n-i}\}$, $i=1,\ldots,n$.}
\eern
 We can choose $\beta_0=\sum_{j=1}^{2n-2}\alpha_j$. Thus
 \beqn
 \text{$m_1=m_n=1,$ and $m_k=2,\,2\leq k\leq n-1$.}
 \eeqn
  Identifying 
 \beqn
 \text{$\gamma_i=\check\omega_i\check\omega_{2n-i}$, $i=1,\ldots,n-1$ and $\gamma_n=\check\omega_n$,}
 \eeqn
  we obtain (cf. \S\ref{cal-cox-A})
\beqn
\bar\gamma_k=\bar 1, 1\leq k\leq n-1,\,\bar\gamma_n=z\in Z(G)^{\vartheta}.
\eeqn
Let $\chi\in\hat I$. We have
\begin{subequations}
\beq\label{type A2-even-1}
R_\chi(x)=
\displaystyle{(x-1)^{n+1}(x+1)^{n-2}}\text{ if }\chi(z)=1\,;\\
\eeq
\beq\label{type A2-even-2}
R_\chi(x)=\displaystyle{(x-1)^{n}(x+1)^{n-1}}\text{ otherwise}\,.
\eeq
\end{subequations}

In~\eqref{theta-tcox}, we take $w=t_{\alpha_1}t_{\alpha_2}\cdots t_{\alpha_n}\in W(G,\Ts)$. 
Then
\beqn
W_\fa=\langle s\rangle\cong\bZ/(2n-1)\bZ,\ s=(1\ \ 2n-1\ \ 2\ \ 2n-2\cdots n-1\ \ n+1\ \ 2n)(n).
\eeqn
From the endoscopic point of view, the non-trivial character $\chi\in\hat I$ corresponds to 
$$\chi=\on{diag}(I_{n-1},-1,I_n) \in \check G=PGL(2n).$$ We have
\beqn
\check G(\chi)=\check G(\chi)^0\cong GL(2n-1),\ \text{ and }\ W_{\fa,\chi}^{en}=W_\fa.
\eeqn
Moreover $\check\theta|_{\check G(\chi)^0}$ is an order $2(2n-1)$ stable (outer) automorphism  of type $A_{2n-2}^2$. Thus Theorem~\ref{thm-expectations} follows in this case from~\eqref{type A2-even-2} and~\eqref{type A2-odd-1}.

\subsubsection{Type $D_{n}^2$}\label{ssec-tcd}Suppose that $G=Spin(2n)$.  We have 
\beqn
m=2n,\ I=\langle z_2=\check\alpha_{n-1}(-1)\check\alpha_n(-1)\rangle\cong\mu_2,\,\text{$\Delta_i=\{\alpha_i\}$, $i=1,\ldots,n-2$, }\,\Delta_{n-1}=\{\alpha_{n-1},\alpha_n\}\, .
\eeqn
  We can choose $\beta_0=\sum_{j=1}^{n-1}\alpha_j.$ Thus
  \beqn
  m_i=1,\,1\leq i\leq n-1.
  \eeqn
 Identifying 
$
 \text{$\gamma_i=\check\omega_i$, $i=1,\ldots,n-2$ and $\gamma_{n-1}=\check\omega_{n-1}\check\omega_n$,}
$
 we obtain (cf. \S\ref{ssec-cox-D})
\beqn
\bar\gamma_k=z_2^k\in  Z(G)^{\vartheta},\, 1\leq k\leq n-1.
\eeqn
Let $\chi\in\hat I$. Then 
\begin{subequations}
\beq\label{type D2-trivial}
R_\chi(x)=
\displaystyle{(x-1)^{n}}\hspace{1in}\text{ if }\chi(z_2)=1\,;
\eeq
\beq\label{type D2-nontrivial}
R_\chi(x)=\displaystyle{(x-1)^{[\frac{n+1}{2}]}(x+1)^{[\frac{n}{2}]}}\hspace{.2in}\text{ otherwise}\,.
\eeq
\end{subequations}

We take $w=t_{\alpha_1}t_{\alpha_2}\cdots t_{\alpha_{n-1}}\in W(G,\Ts).$ 
Then
\beqn
W_\fa=\langle s\rangle\cong\bZ/n\bZ,\ s=(1\ \ 2\ \ \cdots\ n-1\ \ -n\ \ -1\ \ -2\cdots -(n-1)\ \ n)^2.
\eeqn
From the endoscopic point of view, the non-trivial character $\chi\in\hat I$ corresponds to 
$$\chi=\on{diag}(J_n,-J_n)\in\check G=PSO(2n),$$
where $J_n$ is defined in~\eqref{the matrix Jk}. 

Suppose $n$ is even. We have  
$$\check G(\chi)/\check G(\chi)^0\cong\bZ/2\bZ,\,\check G(\chi)^0=P\big(SO(n)\times SO(n)\big)\text{ and }W_{\fa,\chi}^{en}=\langle s^2\rangle\cong \bZ/\frac{n}{2}\bZ.$$
Moreover, $\check\theta|_{\check G(\chi)^0}$ can be identified with an order $n$ stable (outer) automorphism of type $D_{n/2}^2$. Thus Theorem~\ref{thm-expectations} follows in this case from~\eqref{type D2-nontrivial} and~\eqref{type D2-trivial}.

Suppose $n$ is odd. We have $$\check G(\chi)=\check G(\chi)^0=P(GL(n)\times GL(n)^*)\text{ and }W_{\fa,\chi}^{en}=W_\fa.$$  
Moreover,  $\check\theta|_{\check G(\chi)^0}$ is an order $2n$ stable (outer) automorphism of type $A_{n-1}^2$. Thus Theorem~\ref{thm-expectations} follows in this case from~\eqref{type D2-nontrivial} and~\eqref{type A2-odd-1}.

\subsubsection{Type $E_{6}^2$}\label{ssec-tce6}Suppose that $G$ is a simply connected group of type $E_6$.  We have 
\beqn
\text{$\Delta_1=\{\alpha_1,\alpha_6\}$, $\Delta_2=\{\alpha_2\}$, $\Delta_3=\{\alpha_3,\alpha_5\}$, $\Delta_4=\{\alpha_4\}$, $I=\{1\}$ and $m=18$.} 
\eeqn
We can assume $$\beta_0=\alpha_1+\alpha_2+\alpha_3+2\alpha_4+2\alpha_5+\alpha_6.$$ 
Thus
\beqn
m_1=2,\,m_2=1,\,m_3=3,\,m_4=2.
\eeqn
It follows that
\beq\label{poly-tce6}
R(x)=(x-1)^2(x^2-1)^2(x^3-1)\,.\eeq

\subsubsection{Type $D_{4}^{3}$}\label{ssec-tcd4}Suppose that $G=Spin(8)$. We have \beqn
\text{$\Delta_1=\{\alpha_1,\alpha_3,\alpha_4\}$, $\Delta_2=\{\alpha_2\}$,  $I=\{1\}$ and $m=12$.}
\eeqn
 We can choose $$\beta_0=\alpha_1+\alpha_2+\alpha_3.$$ Thus
\beqn
m_1=2,\ m_2=1.
\eeqn
It follows that
\beq\label{poly-d43}
R(x)=(x-1)^2(x^2-1).
\eeq

\section{Stable gradings of classical types}\label{sec-classical stable}
In this section let $G$ be an (almost simple) simply connected group of type $A_n$, $B_n$, $C_n$ or $D_n$. Let $\theta$ be a stable automorphism of $G$ of order $m$. We describe the Fourier transforms of the nearby cycle sheaves $P_\chi$  explicitly following the strategy discussed in subsections~\ref{twisted nearby}-\ref{sec-local system Mchi}. We use the notations there. As a corollary, we obtain full support character sheaves. 

By~\cite{RLYG}, we have the following table describing all stable gradings of classical Lie algebras
\FloatBarrier
\begin{table}[H]
\center
\begin{tabular}{c|c|ccc}
\hline
Type&$m$&$W_\fa$\\
\hline
$A_n$&$n+1$&$\bZ/m\bZ$\\
\hline
$A_n^{2}$&$2d,\,d=2l+1,\,\,d=(n+1)/r\text{ or $d=n/r>1$}$&$G(d,1,r)$\\
\hline
$B_n$ or $C_n$&$2l,\,l=n/r$&$G(m,1,r)$\\
\hline
$D_n$ or $D_n^{2}$&$2l,\,l=n/r$&$G(m,2,r)$\\
\hline
$D_n$ or $D_n^{2}$&$2l,\,l=(n-1)/r>1$&$G(m,1,r)$\\
\hline
$D_4^{3}$&$12$&$\bZ/4\bZ$\\
\hline
\end{tabular}   
\caption{Stable gradings}
\label{table 1}
\end{table}
Recall that the subscript in $A_n^{2}$ $D_n^{2}$ or $D_4^{3}$ indicates the order of $\vartheta$. In the case of $n=rl$ (resp. $n=rl+1$) in type $D_n$, $\vartheta$ is non-trivial if and only if $r$ is odd (resp. even).

\subsection{The Hecke algebras \texorpdfstring{$\cH^{a,m-a}(G(m,1,r))$}{Lg} and \texorpdfstring{$\cH^{m/2}(G(m,2,r))$}{Lg}}In this subsection we recall the definitions of the complex reflection group $G(m,1,r)$ and $G(m,2,r)$ and introduce notations for the Hecke algebras appearing in our setting.
Recall the fixed primitive $m$-th root of unity $\zeta_m$.

The complex reflection group $G(m,1,r)$ can be exhibited as the subgroup of $GL(V,\bC)$, where $V=\on{span}\{e_1,\ldots,e_r\}$, generated by the following linear transformations 
\beqn
\text{$e_i\mapsto c_ie_{\sigma(i)}$, $i=1,\ldots,r$, where $\sigma\in S_r$ (the symmetric group),  $c_i\in\bC$, $c_i^m=1$.}
\eeqn
The complex reflection group $G(m,2,r)$ is an index $2$ subgroup of $G(m,1,r)$ generated by the  linear transformations as above subject to the additional condition
\beqn
 (c_1\cdots c_r)^{m/2}=1.
\eeqn
Let us write
\begin{eqnarray*}
&&\tilde\tau_i:e_i\mapsto\zeta_me_i,\,e_j\mapsto e_j,\,j\neq i,\ 1\leq i\leq r;\ \ \tau_i=\tilde\tau_i^2,\,1\leq i\leq r;\\
&&s_{ij}^{(k)}:e_i\mapsto\zeta_m^ke_j,\ e_j\mapsto\zeta_m^{-k}e_i,\ e_a\mapsto e_a,\,a\neq i,j,\ 1\leq i<j\leq r,\,0\leq k\leq m-1.
\end{eqnarray*}
 We can choose a set of generators of $G(m,1,r)$ (resp. $G(m,2,r)$) as follows
\beqn
\begin{gathered}
s_i=s_{i,i+1}^{(0)},\ 1\leq i\leq r-1,\ \tilde\tau_r\ \ \text{(resp. $s_i=s_{i,i+1}^{(0)},\ 1\leq i\leq r-1,\ s_{r-1}'=s_{r-1,r}^{(1)},\ \tau_r$)}. 
\end{gathered}
\eeqn
The distinguished reflections in $G(m,1,r)$ (resp. $G(m,2,r)$) are 
\beq\label{eqn-reflections}
\bega
s_{ij}^{(k)},\ 1\leq i<j\leq r,\,0\leq k\leq m-1;\ 
\tilde\tau_i\text{ (resp. $\tau_i$)},\ 1\leq i\leq r.
\eega
\eeq
Note that
\beq\label{eqn-reflections-2}
\bega\tilde\tau_k=s_ks_{k+1}\cdots s_{r-1}\tilde\tau_rs_{r-1}\cdots s_k,\ \ s_{ij}^{(k)}=\tilde\tau_i^{-k}\circ s_{ij}^{(0)}\circ\tilde\tau_i^k=\tilde\tau_j^{k}\circ s_{ij}^{(0)}\circ\tilde\tau_j^{-k}\\
s_{ij}^{(0)}=s_is_{i+1}\cdots s_{j-2}s_{j-1}s_{j-2}\cdots s_i.
\eega
\eeq

Let $0\leq a\leq m$. We denote by $\cH^{a,m-a}(G(m,1,r))$ (resp. $\cH^{m/2}(G(m,2,r))$) the Hecke algebra associated to $G(m,1,r)$ (resp. $G(m,2,r)$) defined as follows (see subsection~\ref{sec-hec}). It is the quotient of $\bC[B_{G(m,1,r)}]$ (resp. $\bC[B_{G(m,2,r)}]$) by the ideal generated by elements
\beqn
(\sigma_{H_{s_{ij}^{(k)}}}-1)^2,\,1\leq i<j\leq r,\,0\leq k\leq m-1,\,(\sigma_{H_{\tau_i}}-1)^a(\sigma_{H_{\tau_i}}+1)^{m-a},\,1\leq i\leq r,
\eeqn
\beqn
\text{(resp. $(\sigma_{H_{s_{ij}^{(k)}}}-1)^2,\,1\leq i<j\leq r,\,0\leq k\leq m-1,\ (\sigma_{H_{\tau_i}}-1)^{m/2},\,1\leq i\leq r$)}.
\eeqn

\subsection{Nearby cycle sheaves and full support character sheaves}

Note that the Fourier transforms of the nearby cycle sheaves $P_\chi$ for types $A_n$ and $D_4^{3}$ in Table~\ref{table 1}, which are of rank one, have already been described in sections~\ref{sec-stable rank 1} and~\ref{sec-stable rank 1 explicit}. For completeness we include these two cases below. As before, we identify the $K$-equivariant local system  $\cM_\chi$ on $\Lg_1^{rs}$ with the corresponding representation $\cM_\chi$ of $\pi_1^K(\Lg_1^{rs})=\widetilde{B}_{W_\fa}$.
\begin{theorem}{\rm (i)} Suppose that $\theta$ is an order $n$ stable inner automorphism of $SL(n)$. We have
\beqn
\{\fF(P_\chi)\mid\chi\in\hat I/W_\fa\}=\{\on{IC}(\Lg_1^{rs},\cM_{\chi})\mid\chi\in\widehat{\bZ/n\bZ}\}
\eeqn
where as a representation of $\pi_1^K(\Lg_1^{rs})\cong\bZ\oplus\bZ/n\bZ$
\beqn
\cM_{\chi}\cong\big(\bC[x]/(x^{n/d}-1)^d\big)\otimes \bC_\chi,\text{ if $\chi(z)$ is a primitive $n/d$-th root of $1$, where $d|n$.}
\eeqn
Here $z$ is a generator of $\bZ/n\bZ$.

{\rm (ii)} Suppose that $\theta$ is an order $12$ stable outer automorphism of $Spin(8)$ of type $D_4^{3}$. We have
\beqn
\fF P=\on{IC}(\Lg_1^{rs},\cM)
\eeqn
where as a representation of $\pi_1^K(\Lg_1^{rs})\cong\bZ$
\beqn
\cM\cong\bC[x]/\big((x-1)^3(x+1)\big).
\eeqn
\end{theorem}
\begin{proof}
Part (i) follows from Theorem~\ref{coxeter rank one} and~\eqref{mono-type A inner}. Part (ii) follows from Theorem~\ref{thm-twisted coxeter} and~\eqref{poly-d43}.
\end{proof}

\begin{theorem} \label{nearby cycle-outer A}

 Suppose that $\theta$ is an order $m=2d$ stable outer automorphism of $SL(N)$, $d=2l+1$.

$\mathrm{(i)}$ Suppose $N=rd$. We have
\begin{eqnarray*}
\{\fF(P_\chi)\mid\chi\in\hat I/W_\fa\}=\{\on{IC}(\Lg_1^{rs},\cM_{\chi_k})\mid 0\leq k\leq r/2\},
\end{eqnarray*}
\begin{eqnarray*}
&&\cM_{\chi_k}\cong\bC[\widetilde{B}_{W_\fa}]\otimes_{\bC[\widetilde{B}_{W_\fa}^{\chi_k}]}\left(\bC_{\chi_k}\otimes(\cH^{l+1,l}(G(d,1,k))\otimes \cH^{l+1,l}(G(d,1,r-k))\right), k\neq r/2\\
&&\cM_{\chi_{\frac{r}{2}}}\cong\bC[\widetilde{B}_{W_\fa}]\otimes_{\bC[\widetilde{B}_{W_\fa}^{\chi_{{r}/{2}},0}]}\left(\bC_{\chi_{\frac{r}{2}}}\otimes(\cH^{l+1,l}(G(d,1,{r}/{2}))\otimes \cH^{l+1,l}(G(d,1,{r}/{2}))\right).
\end{eqnarray*}
Here the $\chi_k$'s are defined in Lemma~\ref{type A-lemma-1}.

$\mathrm{(ii)}$ Suppose that $N=rd+1$ and $d\geq 3$. We have
\begin{eqnarray*}
\{\fF(P_\chi)\mid\chi\in\hat I/W_\fa\}=\{\on{IC}(\Lg_1^{rs},\cM_{\chi_k})\mid 0\leq k\leq r\},
\end{eqnarray*} 
\begin{eqnarray*}
&&\cM_{\chi_k}\cong\bC[\widetilde{B}_{W_\fa}]\otimes_{\bC[\widetilde{B}_{W_\fa}^{\chi_k}]}\left(\bC_{\chi_k}\otimes(\cH^{l+2,l-1}(G(d,1,r-k))\otimes \cH^{l+1,l}(G(d,1,k))\right).
\end{eqnarray*}
Here the $\chi_k$'s are defined in Lemma~\ref{type A-lemma-2}.
\end{theorem}

\begin{theorem}\label{nearby cycle-type B} Suppose that $\theta$ is an order $m=2l$ stable automorphism of $Spin(rm+1)$. We have
\begin{eqnarray*}
\{\fF(P_\chi)\mid\chi\in\hat I/W_\fa\}=\{\on{IC}(\Lg_1^{rs},\cM_{\chi_k})\mid 0\leq k\leq r/2\}\cup \{\on{IC}(\Lg_1^{rs},\cM_{\chi_r})\},
\end{eqnarray*}
\begin{eqnarray*}
&&\cM_{\chi_k}\cong\bC[\widetilde{B}_{W_\fa}]\otimes_{\bC[\widetilde{B}_{W_\fa}^{\chi_k}]}\left(\bC_{\chi_k}\otimes(\cH^{l+1,l-1}(G(m,1,k))\otimes \cH^{l+1,l-1}(G(m,1,r-k))\right), k\neq r/2,r\\
&&\cM_{\chi_{\frac{r}{2}}}\cong\bC[\widetilde{B}_{W_\fa}]\otimes_{\bC[\widetilde{B}_{W_\fa}^{\chi_{{r}/{2}},0}]}\left(\bC_{\chi_{\frac{r}{2}}}\otimes(\cH^{l+1,l-1}(G(m,1,{r}/{2}))\otimes \cH^{l+1,l-1}(G(m,1,{r}/{2}))\right)\\
&&\cM_{\chi_{{r}}}\cong\bC[\widetilde{B}_{W_\fa}]\otimes_{\bC[\widetilde{B}_{W_\fa}^{\chi_{r},0}]}\left(\bC_{\chi_{r}}\otimes\cH^{[\frac{l}{2}]+1,[\frac{l-1}{2}]}(G(l,1,r))\right).
\end{eqnarray*}
Here the $\chi_k$'s are defined in Lemma~\ref{lemma-type B}.
\end{theorem}

\begin{theorem}\label{nearby cycle-type C} Suppose that $\theta$ is an order $m=2l$ stable automorphism of $Sp(rm)$. We have
\begin{eqnarray*}
\{\fF(P_\chi)\mid\chi\in\hat I/W_\fa\}=\{\on{IC}(\Lg_1^{rs},\cM_{\chi_k})\mid 0\leq k\leq r\},
\end{eqnarray*}
\begin{eqnarray*}
&&\cM_{\chi_k}\cong\bC[\widetilde{B}_{W_\fa}]\otimes_{\bC[\widetilde{B}_{W_\fa}^{\chi_k}]}\left(\bC_{\chi_k}\otimes(\cH^{l,l}(G(m,1,k))\otimes \cH^{l+1,l-1}(G(m,1,r-k))\right).
\end{eqnarray*}
Here the $\chi_k$'s are defined in Lemma~\ref{lemma-type C}.
\end{theorem}

\begin{theorem}\label{nearby cycle-type D} Suppose that $\theta$ is an order $m=2l$ stable automorphism of $Spin(2n)$.

$\mathrm{(i)}$ Suppose that $n=rl$. We have
\begin{eqnarray*}
&&\{\fF(P_\chi)\mid\chi\in\hat I/W_\fa\}\\
&&=\begin{cases}\{\on{IC}(\Lg_1^{rs},\cM_{\chi_k})\mid 0\leq k\leq (r-1)/2\}\cup\{\on{IC}(\Lg_1^{rs},\cM_{\chi_r})\}&\text{ if $r$ is odd}\\\{\on{IC}(\Lg_1^{rs},\cM_{\chi_k})\mid 0\leq k\leq r/2\}\cup\{\on{IC}(\Lg_1^{rs},\cM_{\chi_k})\mid k=r,r+1\}&\text{ if $r$ is even}\end{cases}\nonumber
\end{eqnarray*} 
\begin{eqnarray*}
&&\cM_{\chi_k}\cong\bC[\widetilde{B}_{W_\fa}]\otimes_{\bC[\widetilde{B}_{W_\fa}^{\chi_k,0}]}\left(\bC_{\chi_k}\otimes(\cH^{l}(G(m,2,k))\otimes \cH^{l}(G(m,2,r-k))\right), k\neq r,r+1\\
&&\cM_{\chi_{r}}\cong\bC[\widetilde{B}_{W_\fa}]\otimes_{\bC[\widetilde{B}_{W_\fa}^{\chi_{r}}]}\left(\bC_{\chi_r}\otimes(\cH^{[\frac{l+1}{2}],[\frac{l}{2}]}(G(l,1,{r})\right)\text{ if $r$ is odd}\\
&&\cM_{\chi_{k}}\cong\bC[\widetilde{B}_{W_\fa}]\otimes_{\bC[\widetilde{B}_{W_\fa}^{\chi_{k},0}]}\left(\bC_{\chi_k}\otimes(\cH^{[\frac{l+1}{2}],[\frac{l}{2}]}(G(l,1,{r})\right),\,k=r,r+1\text{ if $r$ is even}.
\end{eqnarray*}
Here the $\chi_k$'s are defined in Lemma~\ref{lemma-type D-1}.

$\mathrm{(ii)}$ Suppose that $n=rl+1$. We have
\begin{eqnarray*}
&&\{\fF(P_\chi)\mid\chi\in\hat I/W_\fa\}=\begin{cases}\{\on{IC}(\Lg_1^{rs},\cM_{\chi_k})\mid 0\leq k\leq r+2\}&\text{ if $r$ is odd}\\\{\on{IC}(\Lg_1^{rs},\cM_{\chi_k})\mid 0\leq k\leq r+1\}&\text{ if $r$ is even}\end{cases}
\end{eqnarray*}
\begin{eqnarray*}
&&\cM_{\chi_k}\cong\bC[\widetilde{B}_{W_\fa}]\otimes_{\bC[\widetilde{B}_{W_\fa}^{\chi_k}]}\left(\bC_{\chi_k}\otimes(\cH^{l,l}(G(m,1,k))\otimes\cH^{l+2,l-2} (G(m,1,r-k))\right), 0\leq k\leq r,\\
&&\cM_{\chi_{k}}\cong\begin{cases}\bC[\widetilde{B}_{W_\fa}]\otimes_{\bC[\widetilde{B}_{W_\fa}^{\chi_{k},0}]}\left(\bC_{\chi_k}\otimes(\cH^{\frac{l}{2},\frac{l}{2}}(G(l,1,{r})\right)&\text{if $l$ is even}\\
\bC[\widetilde{B}_{W_\fa}]\otimes_{\bC[\widetilde{B}_{W_\fa}^{\chi_{k},0}]}\left(\bC_{\chi_k}\otimes(\cH^{\frac{l+3}{2},\frac{l-3}{2}}(G(l,1,{r})\right)
&\text{if $l$ is odd}\end{cases}\,\text{ if $r$ is odd, $k>r$,}\\
&&\cM_{\chi_{r+1}}\cong\begin{cases}\bC[\widetilde{B}_{W_\fa}]\otimes_{\bC[\widetilde{B}_{W_\fa}^{\chi_{r+1}}]}\left(\bC_{\chi_{r+1}}\otimes(\cH^{\frac{l}{2},\frac{l}{2}}(G(l,1,{r})\right)&\text{if $l$ is even}\\
\bC[\widetilde{B}_{W_\fa}]\otimes_{\bC[\widetilde{B}_{W_\fa}^{\chi_{r+1}}]}\left(\bC_{\chi_{r+1}}\otimes(\cH^{\frac{l+3}{2},\frac{l-3}{2}}(G(l,1,{r})\right)
&\text{if $l$ is odd}\end{cases}\,\text{ if $r$ is even}.
\end{eqnarray*}
Here the $\chi_k$'s are defined in Lemma~\ref{lemma-type D-2}.
\end{theorem}

Theorems~\ref{nearby cycle-outer A}-\ref{nearby cycle-type D} follow from Lemmas~\ref{type A-lemma-1}-\ref{lemma-type D-2} in subsections~\ref{ssec-pfA}-\ref{ssec-pfD}, the triviality of the character $\tau$ established in Lemma~\ref{lemma-tau} below (see Remark~\ref{remark tau}) and Remark~\ref{remark rho}.

Let us write 
\beqn\label{loc-full}
\begin{gathered}
\Theta_{(\Lg_1,K)}=\{\text{irreducible representations of $\pi_1^K(\Lg_1^{rs})$ that appear}\\
\quad\text{ as composition factors of $\cM_{\chi}$, $\chi\in\hat I$}\}.
\end{gathered}
\eeqn
For each $\pi\in\Theta_{(\Lg_1,K)}$, we write $\cL_\pi$ for the corresponding $K$-equivariant local system on $\Lg_1^{rs}$. As in~\cite{VX1}, we have
\begin{corollary}
The IC sheaves $\on{IC}(\Lg_1^{rs},\cL_\pi),\pi\in\Theta_{(\Lg_1,K)}$, are character sheaves on $\Lg_1$.
\end{corollary}
As in~\cite{VX1}, the irreducible representations of $\widetilde{B}_{W_\fa}$ appearing as composition factors of $\cM_\chi$ are of the form $\bC[\widetilde B_{W_\La}]\otimes_{\bC[\widetilde B_{W_\La}^\chi]}\otimes(\bC_\chi\otimes\psi)$, where $\psi$ is an irreducible representation of $B_{W_\La}^{\chi}$ which appears as a composition factor in $\bC[ B_{W_\La}^{\chi}]\otimes_{\bC[B_{W_\La}^{\chi,0}]} \cH_{W_{\La,\chi}^0}$. Since $B_{W_\La}^{\chi}/B_{W_\La}^{\chi,0}$ is a 2-group, it suffices to study the decomposition $\bC[ B_{W_\La}^{\chi}]\otimes_{\bC[B_{W_\La}^{\chi,0}]} \phi$ using Clifford theory, where $\phi$ is a simple module of the Hecke algebra $\cH_{W_{\La,\chi}^0}$. In particular,  all irreducible representations $\pi\in\Theta_{(\Lg_1,K)}$ can be obtained as quotients of $\cM_\chi$.

\begin{conjecture}
\label{nearby and full support}
For the stably graded Lie algebras considered here, the set of cuspidal character sheaves on $\Lg_1$ is precisely
\beq\label{cuspidal sheaves}
\left\{\on{IC}(\Lg_1^{rs},\cL_\pi)\,|\,\pi\in\Theta_{(\Lg_1,K)}\right\}.
\eeq
\end{conjecture}
As in~\cite{VX1}, we say that a character sheaf is {\em cuspidal} if it does not arise as a direct summand (up to shift) from the parabolic induction (see for example~\cite[\S2.4]{CVX}) of a character sheaf on $\Ll_1$ of a $\theta$-stable Levi subgroup $L$ contained in a $\theta$-stable proper parabolic subgroup of $G$.

\begin{remark}Note that all sheaves in~\eqref{cuspidal sheaves} are cuspidal. This can be seen as follows. Recall that Fourier transform commutes with parabolic induction. Thus the support of all character sheaves arising as (direct summand of) parabolic induction from a $\theta$-stable proper parabolic subgroup $P$ is contained in $\overline{K.\Lp_1}$, where $\Lp=\on{Lie}P$. Using for example~\cite[\S6-8]{Y} one can check that for any proper $\theta$-stable parabolic subgroup $P$, we have $\overline{K.\Lp_1}\subsetneq\Lg_1$. Thus the claim follows since all sheaves in~\eqref{cuspidal sheaves} have full support, i.e., the support is the whole of $\Lg_1$.
\end{remark}
\begin{remark}
Conjecture~\ref{nearby and full support} holds for stably $\bZ/2\bZ$-graded classical Lie algebras (that is, split symmetric pairs) by~\cite{CVX,VX1}.
\end{remark}

 We fix a pinning using $T=\Ts$. For each $\alpha\in\Phi=\Phi(T,G)$, we write $\check{\alpha}\in\check\Phi$ for the corresponding coroot. Since $T$ is $\theta$-stable, $\theta$ induces an automorphism on $\Phi$. We write $$\theta(\Lg_\alpha)=\Lg_{\theta\alpha},\ \alpha\in\Phi.$$
Let $W=N_G(T)/T$. For each $\alpha\in\Phi$, let $t_\alpha\in W$ denote the corresponding reflection. 

We write
\beq\label{eqn-thetaorbit}
(\alpha)=\{\theta^i\alpha,\,i=0,\ldots,m-1\},\,\alpha\in\Phi.
\eeq
Since $\theta$ is stable, we have $|(\alpha)|=m$. We say $\alpha\sim\beta$, $\alpha,\beta\in\Phi$, if $(\alpha)=(\beta)$, or equivalently, $\beta=\theta^i\alpha$ for some $i$. 
\begin{lemma}\label{lemma-tau}
For each $g\in I=Z_K(\fa)=\Ts^\theta$, $\on{det}(g|_{\Lg_1})=1$.
\end{lemma}
\begin{proof}
Recall the basis $X_\alpha\in\Lg_\alpha$, $\alpha\in\Phi$. There exists $c_\alpha\in\bC$ such that
$
\theta(X_\alpha)=c_\alpha X_{\theta\alpha}.
$
We have
\beqn
\prod_{i=0}^{m-1}c_{\theta^i\alpha}=1.
\eeqn
One readily checks that
\bern
&&\Lg_i=\{x\in\Lg\,|\,\theta(x)=\zeta_m^ix\}=\Lt_i\oplus\bigoplus_{\alpha\in\Phi/\sim}\bC Y_{\alpha,i},\ \\&& \Lt_i=\Lt\cap\Lg_i,\ Y_{\alpha,i}=\sum_{k=0}^{m-1}\left(\zeta_m^{-ki}(\prod_{j=0}^{k-1}c_{\theta^j\alpha})X_{\theta^k\alpha}\right), \,0\leq i\leq m-1,
\eern
where $\Lt=\on{Lie}\Ts$.
Let $t\in Z_K(\fa)=\Ts^\theta$. It follows that $\on{det}(t|_{\Lg_1})=\prod_{\alpha\in\Phi/\sim}\alpha(t)$.

Using the explicit generators of $I$ in Lemmas~\ref{type A-lemma-1}-\ref{lemma-type D-2}, one checks  that $\alpha(t)=\pm 1$ for each $t\in I$ and $\alpha\in\Phi$. Moreover, $\#\{\alpha\in\Phi\mid\alpha(t)=-1\}/m$ is an even integer. The lemma follows.
\end{proof}

Our proofs of the theorems in this section follow from reduction to the (semisimple) rank one calculations in section~\ref{sec-stable rank 1 explicit}. In the reductions to (semisimple) rank one situation the groups $G_s$ are reductive. In section~\ref{Isogenies} we explain how we can reduce the calculation to almost simple simply connected groups. The first step is to pass to the derived group $(G_s)_{\on{der}}$ which, if we start with an almost simple simply connected $G$ is also simply connected.  However, the group $(G_s)_{\on{der}}$ is not necessarily almost simple. It can be a product of almost simple groups which $\theta$ permutes. The reduction in this case is explained in section~\ref{Isogenies}, see formula~\eqref{reduction-prod}. In our situation it will turn out that $(G_s)_{\on{der}}$ is either almost simple of the same type as $G$ or a product of $m/2$ copies of $SL(2)$'s which are cyclically permuted by $\theta$. 
 
 We explain the two cases in a bit more detail. Suppose first that $s\in W_\fa$ is a distinguished  reflection of the form $\tau_i$ (see~\eqref{eqn-reflections}). Then there exist $\beta_1,\ldots,\beta_a\in\Phi$ such that $\cup_{j=1}^a(\beta_j)$ (see~\eqref{eqn-thetaorbit}) is the set of roots for $(G_s,\Ts)$ and $(G_s)_{\on{der}}$ is of the same type as $G$. Moreover, $\theta|_{(G_s)_{\on{der}}}$ is an order $m$ stable (semisimple) rank 1 automorphism of $(G_s)_{\on{der}}$ discussed in section~\ref{sec-stable rank 1 explicit}. Thus we determine the polynomials $R_{\chi,\tau}$ using the results in section~\ref{sec-stable rank 1 explicit}.   Suppose now that $s\in W_\fa$ is a distinguished  reflection of order 2 of the form $s_{ij}^{(k)}$ (see~\eqref{eqn-reflections}). Then there exists $\beta\in\Phi$ such that $(\beta)$ (see~\eqref{eqn-thetaorbit}) is the set of roots for $(G_s,\Ts)$, which forms a root system of type $A_1^{m/2}$. In this case $$I_s=\langle\prod_{a=0}^{m/2-1}\theta^a\check\beta(-1)\rangle\cong \mu_2$$ and $\theta^{m/2}$ restricts to a (non-trivial, stable) involution on $SL(2)$. By the reduction equation~\eqref{reduction-prod} and~\eqref{mono-type A inner}, we have
\beq\label{poly-order 2}
R_{\chi,s}=(x-1)^2\text{ if $\chi|_{I_s}=1$},\ \ R_{\chi,s}=x^2-1\text{ if $\chi|_{I_s}\neq1$}.
\eeq

In the remainder of the section we carry out the above analysis in each type and discuss the endoscopic point of view.

\subsection{Type \texorpdfstring{$A$}{Lg}}\label{ssec-pfA}

Let $G=SL(N)$. Let $\theta$ be a stable outer automorphism of $G$ of order $m=2d$, $d=2l+1$, where $d|N$ or $d|(N-1)$, see Table~\ref{table 1}.  Let us write
\beqn
r=\frac{N}{d} \text{ (or $r=\frac{N-1}{d}$).}
\eeqn
Let $t_{ij}$ denote the transposition $(i\ \ j)\in S_{N}$, the Weyl group of $G$. We can and will assume that (see~\cite[section 7]{RLYG})
\beqn
\text{$\theta|_{\Lt_{\mathbf{s}}=\on{Lie}(\Ts)}=-w_0$, $w_0=\tau_1\ldots\tau_r$ and}
\eeqn
\beqn
\tau_i=t_{{(i-1)d+1},{(i-1)d+2}}t_{{(i-1)d+2},{(i-1)d+3}}\cdots t_{{(i-1)d+{d-1}},{(i-1)d+d}},\,1\leq i\leq r.
\eeqn
The little Weyl group is
\beqn
W_\fa=\langle s_1,\ldots,s_{r-1},\tau_r\rangle\cong G(d,1,r),\ s_k=\prod_{i=1}^{d}t_{(k-1)d+i, kd+i},\ 1\leq k\leq r-1.
\eeqn

\begin{lemma} \label{type A-lemma-1}
Suppose that $N=rd=rm/2$. 

\noindent{\rm(i)} We have
\beqn
I=\langle \gamma_1,\ldots,\gamma_{r-1}\rangle\cong\mu_2^{r-1},
\ 
\gamma_k=\prod_{i=1}^{d}\check\alpha_{(k-1)d+2i-1}(-1),\ 1\leq k\leq r-1.
\eeqn
{\rm (ii)} A set of representative of $W_\fa$-orbits in $\hat I$ is $\{\chi_k,\, 0\leq k\leq r/2\}$, defined by
\beqn
\chi_k(\gamma_i)=1,\ i\neq k,\ \chi_k(\gamma_k)=-1.
\eeqn
{\rm(iii)} For each $0\leq k\leq \frac{r}{2}$, we have
\bern
&&W_{\fa,\chi_k}^0\cong G(d,1,k)\times G(d,1,r-k),\\
&&\cH_{W_{\fa,\chi_k}^0}=\cH^{l+1,l}(G(d,1,k))\otimes \cH^{l+1,l}(G(d,1,r-k)).
\eern
Moreover,
\begin{eqnarray*}
W_{\fa,\chi_k}^{en}=W_{\fa,\chi_k}^0,\,0\leq k\leq r/2;\ W_{\fa,\chi_k}=W_{\fa,\chi_k}^0,\, k\neq r/2,\text{ and }\,W_{\fa,\chi_{r/2}}/ W_{\fa,\chi_{r/2}}^0\cong\bZ/2\bZ\,.
\end{eqnarray*}

\end{lemma}
\begin{proof}
Part (i) follows from the fact that $I=(\Ts)^\theta$. Let  (see~\eqref{eqn-reflections} and~\eqref{eqn-reflections-2})
$$s_{ij}^{(k)}=\tau_j^{k}\circ s_{ij}^{(0)}\circ\tau_j^{-k},\ \,
s_{ij}^{(0)}=s_is_{i+1}\cdots s_{j-2}s_{j-1}s_{j-2}\cdots s_i.$$ One checks readily that
\bern
&&\tau_k\gamma_i=\gamma_i;\ s_{ij}^{(a)}(\gamma_k)=\gamma_k\text{ if $i,j\neq k,k+1$},\ s_{ik}^{(a)}(\gamma_k)=\gamma_i\cdots\gamma_k,\ s_{i,k+1}^{(a)}(\gamma_k)=\gamma_i\cdots\gamma_{k-1},\\
&&s_{k,k+1}^{(a)}(\gamma_k)=\gamma_{k},\ s_{kj}^{(a)}(\gamma_k)=\gamma_{k+1}\cdots\gamma_{j-1},\text{ if $j\geq k+2$},\ s_{k+1,j}^{(a)}(\gamma_k)=\gamma_k\cdots\gamma_{j-1}.
\eern
It then follows that
\beqn
s_{ij}^{(a)}.\chi_k=\chi_k\Leftrightarrow j\leq k\text{ or }i\geq k+1;\ \ \tau_i.\chi_k=\chi_k.
\eeqn
Hence
\begin{equation}
\bega\label{Wachi-A2}
W_{\fa,\chi_k}^0=\langle s_{ij}^{(a)},j\leq k\text{ or }i\geq k+1,\tau_b, 1\leq b\leq r  \rangle\\
=\langle s_1,\ldots,s_{k-1},\tau_k\rangle\times\langle s_{k+1},\ldots,s_{r-1},\tau_{r}\rangle\cong G(d,1,k)\times G(d,1,r-k).
\eega
\end{equation}
Let $\sigma_0=\prod_{j=1}^{r/2}s_{j,r+1-j}^{(0)}$. It is easy to check that $\sigma_0\in W_{\fa,\chi_{r/2}}$ and $\sigma_0\not\in W_{\fa,\chi_{r/2}}^0$. Note that
\beq
\sigma_0s_i\sigma_0=s_{r-i},\,\sigma_0\tau_i\sigma_0=\tau_{r+1-i}.
\eeq
Part (ii) and the assertions on $W_{\fa,\chi_k}^0$ and $W_{\fa,\chi_k}$ in part (iii) follow.

From the endoscopic point of view, we have (assuming that $\check \Ts$ is the torus consisting of diagonal matrices)
\bern
&&(\check \Ts)^{\check\theta}=\{\on{\diag}(\lambda_1I_d,\lambda_2I_d,\ldots,\lambda_rI_d)\in PGL(N)\mid \lambda_1^2=\lambda_2^2=\cdots=\lambda_r^2\}\cong\mu_2^{r-1},\\
&&\chi_k=\on{diag}(-I_{dk},I_{(r-k)d})\in PGL(N).
\eern
Moreover, 
\bern
&&\check G(\chi_k)^0\cong P\big(GL(kd)\times GL((r-k)d)\big)\,,\\
&&\check G(\chi_k)/\check G(\chi_k)^0\cong 1\text{ (resp. $\bZ/2\bZ$)}\text{ if $k\neq r/2$ (resp. $k=r/2$)}.
\eern 
It follows that $W_{\fa,\chi_k}^{en}=W_{\fa,\chi_k}^0$. Note that $\check\theta$ restricts to an order $2d$ stable outer automorphism of each factor of $\check G(\chi_k)^0$.

It remains to prove the claim on $\cH_{W_{\fa,\chi_k}^0}$. We have
\beqn
s_{ij}^{(k)}=\prod_{a=1}^{d}t_{(i-1)d+a, (j-1)d+\overline{k+a}}\in W=S_N
\eeqn
where $\overline{k+a}\in[1,d]$ is congruent to $k+a$ mod $d$. The roots of $(G_{s_{ij}^{(k)}},\Ts)$ are $$(\beta_{ij}^{(k)})=\{\theta^{a}\beta_{ij}^{(k)}\mid\,a=0,\ldots,m-1\}=\{\pm(\epsilon_{(i-1)d+a}-\epsilon_{ (j-1)d+\overline{k+a}})\mid a=1,\ldots,d\},$$
where $\beta^{(k)}_{ij}=\epsilon_{(i-1)d+1}-\epsilon_{ (j-1)d+\overline{k+1}}$. It follows that 
\beqn
I_{s_{ij}^{(k)}}=\langle\prod_{a=1}^{d}\check\theta^{a}\beta_{ij}^{(k)}(-1)\rangle=\langle \gamma_i\cdots\gamma_{j-1}\rangle\cong\mu_2.
\eeqn
A set of simple roots of $(G_{\tau_i},\Ts)$ can be chosen as $\alpha_{(i-1)d+a},\,a=1,\ldots,d-1$. Moreover $\theta$ restricts to an order $2d$ stable (outer) automorphism on $(G_{\tau_i})_{\on{der}}\cong SL(d)$. 
Thus (see~\eqref{Isder} and \S\ref{ssec-tcae})
\begin{eqnarray*}
\ I_{\tau_i}=1.
\end{eqnarray*}
It follows that
\beq\label{charsA}
\chi_k|_{I_{s_{ij}^{(a)}}}=1\Leftrightarrow\text{ $j\leq k$ or $i\geq k+1$}\Leftrightarrow s_{ij}^{(a)}\in W_{\fa,\chi_k}^0,\text{ and } \chi_k|_{I_{\tau_i}}=1.
\eeq
 In view of~\eqref{charsA}, the claim on $\cH_{W_{\fa,\chi_k}^0}$ follows from equations~\eqref{poly-order 2},~\eqref{Wachi-A2} and~\eqref{type A2-odd-1}.
\end{proof}

\begin{lemma}\label{type A-lemma-2}
 Suppose that $N=rd+1=rm/2+1$. 
 
\noindent {\rm(i)} We have 
\beqn
I=\langle\gamma_1,\ldots,\gamma_r\rangle\cong\mu_2^r,\,\gamma_i=\prod_{j=1}^l\check\alpha_{(i-1)d+2j-1}(-1)\prod_{j=id}^{rd}\check\alpha_j(-1),\,1\leq i\leq r.
\eeqn

\noindent{\rm(ii)} A set of representatives of $W_\fa$-orbits in $\hat I$ is $\chi_k$, $0\leq k\leq r$, defined by
$$\chi_k(\gamma_i)=-1\text{ if }1\leq i\leq k,\ \chi_k(\gamma_i)=1\text{ if }k+1\leq i\leq r.$$
{\rm (iii)} We have\beqn
\begin{gathered}
W_{\fa,\chi_k}=W_{\fa,\chi_k}^0=W_{\fa,\chi_k}^{en}\cong G(d,1,k)\times G(d,1,r-k),\\
\cH_{W_{\fa,\chi_k}^0}\cong\cH^{l+2,l-1}(G(d,1,r-k))\otimes \cH^{l+1,l}(G(d,1,k)),\ 0\leq k\leq r.
\end{gathered}
\eeqn

\end{lemma}

\begin{proof}
The proof is entirely similar to that of Lemma~\ref{type A-lemma-2}. We will only discuss the key points. 
We have
\begin{eqnarray*}
&&s_{ij}^{(a)}:\ \gamma_{i}\mapsto\gamma_{j},\ \gamma_{j}\mapsto\gamma_{i},\ \gamma_k\mapsto\gamma_k,\ k\neq i, j;\ \tau_i(\gamma_k)=\gamma_k.
\end{eqnarray*}
It follows that
\beq\label{Wachi-A1}
W_{\fa,\chi_k}^0=\langle s_1,\ldots,s_{k-1},\tau_k\rangle\times\langle s_{k+1},\ldots,s_{r-1},\tau_r\rangle\cong G(d,1,k)\times G(d,1,r-k).
\eeq
From the endoscopic point of view, we have
\bern
&&(\check \Ts)^{\check\theta}=\{\on{\diag}(\lambda_1I_d,\lambda_2I_d,\ldots,\lambda_rI_d,1)\in PGL(N)\mid \lambda_i^2=1,\,i=1,\ldots,r\}\cong\mu_2^{r},\\
&&\chi_k=\on{diag}(-I_{dk},I_{(r-k)d},1)\in PGL(N).
\eern
Moreover, 
\bern
&&\check G(\chi_k)=\check G(\chi_k)^0\cong P\big(GL(kd)\times GL((r-k)d+1)\big).
\eern 
It follows that $W_{\fa,\chi_k}^{en}=W_{\fa,\chi_k}^0=W_{\fa,\chi_k}$. Note that $\check\theta$ restricts to an order $2d$ stable outer automorphism of each factor of $\check G(\chi_k)^0$. 

We have 
\beqn
I_{s_{ij}^{(k)}}=\langle\gamma_i\gamma_{j}\rangle\cong\mu_2.
\eeqn
 Moreover $\theta$ restricts to an order $2d$ stable (outer) automorphism on $(G_{\tau_i})_{\on{der}}\cong SL(d+1)$. A set of simple roots of $(G_{\tau_i},\Ts)$ can be chosen as $\{\alpha_{(i-1)d+a},\,a=1,\ldots,d-1,\,-(\epsilon_{id}-\epsilon_{rd+1})\}$.
Thus (see~\eqref{Isder} and \S\ref{ssec-tcao})
\beqn
 I_{\tau_k}=\langle\gamma_k\rangle\cong\mu_2.
\eeqn
It follows that 
\beq\label{charsA2}
\chi_k|_{I_{s_{ij}^{(a)}}}=1\Leftrightarrow\text{$j\leq k$ or $i\geq k+1$}\Leftrightarrow s_{ij}^{(a)}\in W_{\fa,\chi_k}^0;\ \ \chi_k|_{\tau_i}=1\Leftrightarrow i\geq k+1.
\eeq
 In view of~\eqref{charsA2}, the claim on $\cH_{W_{\fa,\chi_k}^0}$ follows from equations~\eqref{poly-order 2},~\eqref{Wachi-A1},~\eqref{type A2-even-1} and~\eqref{type A2-even-2}.
\end{proof}

\subsection{Type B}\label{sec-typeB}
Let $G=Spin(2N+1)$. Let $\theta$ be a stable automorphism of order $m=2l$, where $l|N$, see Table~\ref{table 1}. Let $r=N/l$. We identify the Weyl group $W$ of $G$ with the group $W_N$ of permutations $\sigma$ on $\{1,\ldots, N, -1,\ldots, -N\}$ such that $\sigma(i)=j\Leftrightarrow \sigma(-i)=-j$.  Let us write
\beqn
t_{ij}=(i\ \ j)(-i\ \ -j)\,\text{ and }\ t_a=(a\ \ -a)\in W_N.
\eeqn

We can assume that (see~\cite[section 7]{RLYG})
\beq\label{theta-type B}
\bega
\theta|_{\Lt=\on{Lie}(\Ts)}=\prod_{i=1}^r\tau_i,\\
\,\tau_i=t_{k_i+1,k_i+2}t_{k_i+2,k_i+3}\cdots t_{k_i+l-1,k_i+l}t_{k_i+l}\in W,\,k_i=(i-1)l,\ 1\leq i\leq r.
\eega
\eeq
The little Weyl group $$W_\fa=\langle s_1,\ldots,s_{r-1},\tau_r\rangle\cong G(m,1,r)$$ where
\bern
&&s_k=\prod_{a=1}^lt_{(k-1)l+a,kl+a}\,,\ 1\leq k\leq r-1.
\eern

\begin{lemma}\label{lemma-type B}
{\rm (i)}We have $I=\langle\gamma_1,\ldots,\gamma_r\rangle\cong\mu_2^r$, where
\bern
&&\gamma_k=\prod_{j=1}^l\check\alpha_{(k-1)l+2j-1}(-1),\ 1\leq k\leq r-1,\ \gamma_r=\check\alpha_{rl}(-1).
\eern
{\rm (ii)} A set of representative of $W_\fa$ orbits in $\hat I$ is 
$
\text{\{$\chi_k$, $0\leq k\leq r/2$, $\chi_r$\}}
$
defined by
\beqn
\chi_k(\gamma_i)=1,\ i\neq k,\ \chi_k(\gamma_k)=-1.
\eeqn
{\rm (iii)} We have 
\bern
&&W_{\fa,\chi_k}^0=W_{\fa,\chi_k}^{en}\cong G(m,1,k)\times G(m,1,r-k),\\
&&\cH_{W_{\fa,\chi_k}^0}\cong \cH^{l+1,l-1}(G(m,1,k))\otimes\cH^{l+1,l-1} (G(m,1,r-k)),\ 0\leq k\leq \frac{r}{2};\\
&&W_{\fa,\chi_{r}}^0=W_{\fa,\chi_r}^{en}\cong G(l,1,r),\,\cH_{W_{\fa,\chi_{r}}^0}\cong\cH^{[\frac{l}{2}]+1,[\frac{l-1}{2}]}(G(l,1,r)).
\eern
Moreover, 
\begin{eqnarray*}
W_{\fa,\chi_k}=W_{\fa,\chi_k}^0,\ 0\leq k\leq \frac{r-1}{2},\ \ W_{\fa,\chi_{k}}/ W_{\fa,\chi_{k}}^0\cong\bZ/2\bZ,\,k=r/2,r.
\end{eqnarray*}
\end{lemma}

\begin{proof}
One checks that 
\bern
&& s_{ij}^{(a)}(\gamma_k)=\gamma_k\text{ if $i,j\neq k,k+1$ or if $k=r$},\ s_{ik}^{(a)}(\gamma_k)=\gamma_i\cdots\gamma_k\text{ (resp. $\gamma_i\cdots\gamma_k\gamma_r$)} \text{ if $k\neq r$},
\\&&s_{k,k+1}^{(a)}(\gamma_k)=\gamma_{k},\ s_{i,k+1}^{(a)}(\gamma_k)=\gamma_i\cdots\gamma_{k-1}\text{ (resp. $\gamma_i\cdots\gamma_{k-1}\gamma_r$)},\text{ if $i\leq k-1$},\\
&& s_{kj}^{(a)}(\gamma_k)=\gamma_{k+1}\cdots\gamma_{j-1}\text{ (resp. $\gamma_{k+1}\cdots\gamma_{j-1}\gamma_r$)},\text{ if $j\geq k+2$},\\&& s_{k+1,j}^{(a)}(\gamma_k)=\gamma_k\cdots\gamma_{j-1}\text{ (resp. $\gamma_k\cdots\gamma_{j-1}\gamma_r$)},\text{ when $a$ is even (resp. odd);}
\\
&&\tau_i({\gamma_k})=\gamma_k,\,\text{ if $k\neq i-1,i$ or $k=r$},\ \tau_i({\gamma_k})=\gamma_k\gamma_r\text{ if }k\in\{i-1,i\}\text{ and }k\neq r.
\eern
It follows that
\bern
&&s_{ij}^{(a)}.\chi_k=\chi_k\Leftrightarrow j\leq k\text{ or }i\geq k+1,\ \ \tau_i.\chi_k=\chi_k,\text{ when $0\leq k\leq r/2$};
\\
&&s_{ij}^{(a)}.\chi_r=\chi_r\Leftrightarrow \text{$a$ is even};\ \tau_i^a.\chi_r=\chi_r\Leftrightarrow \text{$a$ is even}.\eern
Hence
\ber
&&W_{\fa,\chi_{k}}^0=\langle s_{ij}^{(a)},\,j\leq k,\,i\geq k+1,\ \tau_i\rangle=\langle s_1,\ldots,s_{k-1},\tau_k\rangle
\times\langle s_{k+1},\ldots,s_{r-1},\tau_r\rangle\label{wachi-b1}\\
&&\cong G(m,1,k)\times G(m,1,r-k), \, 0\leq k\leq r/2;\nonumber
\\
&&W_{\fa,\chi_{r}}^0=\langle s_{ij}^{(a)},\ \tau_i^2\mid a\text{ even}\rangle=\langle s_1,\ldots,s_{r-1},\tau_r^2\rangle\cong G(l,1,r).\label{wachi-b2}
\eer
From the endoscopic point of view, $\chi_k$ corresponds to \bern
&&\chi_k=\on{diag}(-I_{kl},I_{(r-k)l},I_{(r-k)l},-I_{kl})\in\check G=PSp(2N),\,0\leq k\leq r/2;
\\
&&\chi_r=\on{diag}(\underbrace{J_{l},\ldots,J_l}_{r\text{ terms}},-J_{l},\ldots,-J_l)\in\check G=PSp(2N),
\eern
where $J_l$ is defined in~\eqref{the matrix Jk}. 
 We have
\bern
&& \check G(\chi_k)^0\cong P\big(Sp(2kl)\times Sp(2rl-2kl)\big),\ 0\leq k\leq r/2; 
 \\
&&  G(\chi_r)^0\cong P\big(Sp(rl)\times Sp(rl)\big)\text{ (resp. $P(GL(rl)\times GL(rl)^*)$)}\text{ if $l$ is even (resp. odd)} \\
&&\check G(\chi_k)=\check G(\chi_k)^0,\,0\leq k<r/2,\, \check G(\chi_k)/\check G(\chi_k)^0\cong\bZ/2\bZ,\,k=r/2,r.
 \eern
For $0\leq k\leq r/2$, $\check\theta$ restricts to each factor of $\check G(\chi_k)^0$ as an order $2l$ stable automorphism. When $l$ is even,  $\check\theta|_{\check G(\chi_r)^0}$ can be identified with an order $l$ stable automorphism of type $C_{rl/2}$; when $l$ is odd, $\check\theta|_{\check G(\chi_r)^0}$ is an order $2l$ stable (outer) automorphism type $A_{rl-1}^{2}$. It follows that 
$$W_{\fa,\chi_k}^{en}\cong W_{\fa,\chi_k}^0,\ 0\leq k\leq r/2,\text{ or }k=r.$$
One checks that
\bern
&&W_{\fa,\chi_{\frac{r}{2}}}=\langle\sigma_{0}\rangle\ltimes W_{\fa,\chi_{\frac{r}{2}}}^0,\,W_{\fa,\chi_{r}}=\langle\sigma_{0}'\rangle\ltimes W_{\fa,\chi_{r}}^0,\ \sigma_0=\prod_{j=1}^{r/2}s_{j,r+1-j}^{(0)},\\
&&\sigma_0': \epsilon_{jl+i}\mapsto \epsilon_{(r-1-j)l+i+1},\ \epsilon_{jl+l}\mapsto -\epsilon_{(r-1-j)l+1},\ 0\leq j\leq r-1, 1\leq i\leq l-1.
\eern
Moroever $\sigma_0^2=1$, $(\sigma_0')^2\in W_{\fa,\chi_{r}}^0$ and 
\beq
\sigma_0s_i\sigma_0=s_{r-i},\,\sigma_0\tau_i\sigma_0=\tau_{r+1-i},\ \,\sigma_0's_i(\sigma_0')^{-1}=s_{r-i},\,\sigma_0'\tau_i(\sigma_0')^{-1}=\tau_{r+1-i}.
\eeq
We have that $\theta$ restricts to an order $2l$ stable automorphism on $(G_{\tau_i})_{\on{der}}\cong Spin(2l+1)$ and 
\beqn
I_{s_{ij}^{(k)}}=\langle\prod_{a=i}^{j-1}\gamma_a\rangle\text{ (resp. $\langle\gamma_r\prod_{a=i}^{j-1}\gamma_a\rangle$)}\cong\mu_2,\text{ if $k$ is even (resp. odd)};\ I_{\tau_k}=\langle\gamma_r\rangle\cong\mu_2.
\eeqn
Thus
\begin{subequations}
\beq\label{charsB1}
\chi_k|_{I_{s_{ij}^{(a)}}}=1\Leftrightarrow{i\geq k+1\text{ or }j\leq k}\Leftrightarrow s_{ij}^{(a)}\in W_{\fa,\chi_k}^0,\ \ \chi_k|_{I_{\tau_a}}=1,\ \text{ for }0\leq k\leq r/2;
\eeq
\beq\label{charsB2}
\chi_r|_{I_{s_{ij}^{(a)}}}=1\Leftrightarrow\text{$a$ is even}\Leftrightarrow s_{ij}^{(a)}\in W_{\fa,\chi_k}^0,\ \ \chi_r|_{I_{\tau_a}}\neq1.
\eeq
\end{subequations}
   In view of~\eqref{charsB1}, the claim on $\cH_{W_{\fa,\chi_k}^0}$, $0\leq k\leq r/2$, follows from equations~\eqref{poly-order 2},~\eqref{wachi-b1} and~\eqref{mono-type B-trivial}.
In view of~\eqref{charsB2}, the claim on $\cH_{W_{\fa,\chi_r}^0}$  follows from equations~\eqref{poly-order 2},~\eqref{wachi-b2} and~\eqref{mono-type B-nontrivial}.
\end{proof}

\subsection{Type C}

Let $G=Sp(2N)$. Let $\theta$ be a stable automorphism of order $m=2l$, where $l|N$, see Table~\ref{table 1}. Let $r=N/l$. We identify the Weyl group of $G$ with the Weyl group of type $B_N$. We can assume that $\theta|_{\Lt=\on{Lie}(\Ts)}$ is defined in the same way as in~\eqref{theta-type B} (see~\cite[section 7]{RLYG}). The little Weyl group $W_\fa\cong G(m,1,r)$ has the same generators as in type $B_N$.

\begin{lemma}\label{lemma-type C}
{\rm (i)} We have \beqn
I=\langle\gamma_1,\ldots,\gamma_r\rangle\cong\mu_2^r
\eeqn
where \begin{eqnarray*}
&&\gamma_k=\prod_{j=1}^{l/2}\check\alpha_{(k-1)l+2j-1}(-1),\ 1\leq k\leq r,\ \text{ if $l$ is even};
\\
&&\gamma_k=\prod_{j=1}^l\check\alpha_{(k-1)l+2j-1}(-1),\ 1\leq k\leq r-1,\ \gamma_r=\prod_{j=1}^{\frac{l+1}{2}}\check\alpha_{(r-1)l+2j-1}(-1),\ \text{ if $l$ is odd}.
\end{eqnarray*}
{\rm (ii)} A set of representative of $W_\fa$-orbits in $\hat I$ is $\{\chi_k,\ 0\leq k\leq r\}$, 
defined by
\begin{eqnarray*}
&&\chi_k(\gamma_i)=-1,\ 1\leq i\leq k,\ \chi_k(\gamma_i)=1,\ k+1\leq i\leq r,\ \text{when $l$ is even};
\\
&&\chi_k(\gamma_k)=-1,\ \chi_k(\gamma_i)=1,\ i\neq k,\ \text{when $l$ is odd}.
\end{eqnarray*}
{\rm (iii)} We have that
\bern
&&W_{\fa,\chi_k}=W_{\fa,\chi_k}^0\cong G(m,1,k)\times G(m,1,r-k),\ 0\leq k\leq r,
\\
&&\cH_{W_{\fa,\chi_k}^0}\cong \cH^{l,l}(G(m,1,k))\otimes\cH^{l+1,l-1} (G(m,1,r-k)).
\eern
Moreover,
\beqn
W_{\fa,\chi_k}^{en}\cong  G(m,2,k)\times G(m,1,r-k),\,W_{\fa,\chi_k}/W_{\fa,\chi_k}^{en}\cong\bZ/2\bZ, \,k\geq 1.
\eeqn

\end{lemma}

\begin{proof}

One checks that 
\begin{eqnarray*}
&&s_i:\gamma_{i-1}\mapsto\gamma_{i-1}\gamma_{i},\ \gamma_{i+1}\mapsto\gamma_{i}\gamma_{i+1},\ \gamma_k\mapsto\gamma_k,\ k\neq i-1,i+1,\ \text{ when $l$ is odd,}\\
&&s_i:\gamma_{i}\mapsto\gamma_{i+1},\ \gamma_{i+1}\mapsto\gamma_{i},\ \gamma_k\mapsto\gamma_k,\ k\neq i, i+1,\ \text{ when $l$ is even},\ 1\leq i\leq r-1;\\
&&\tau_r:\gamma_k\mapsto\gamma_k.
\end{eqnarray*}
It follows that
\beq\label{wachi-c}
W_{\fa,\chi_k}^0=\langle s_1,\ldots,s_{k-1},\tau_k\rangle\times\langle s_{k+1},\ldots,s_{r-1},\tau_r\rangle\cong G(m,1,k)\times G(m,1,r-k).
\eeq
From the endoscopic point of view,  $\chi_k$ corresponds to  
\beqn
\chi_k=\on{diag}(-I_{kl},I_{(r-k)l},1,I_{(r-k)l},-I_{kl})\in \check G=SO(2N+1),\ 0\leq k\leq r.
\eeqn
When $k\neq 0$
\beqn
\check G(\chi_k)^0\cong SO(2kl)\times SO(2(r-k)l+1)\text{ and }\check G(\chi_k)/\check G(\chi_k)^0\cong\bZ/2\bZ.
\eeqn
Moreover, $\check\theta$ restricts to an order $2l$ stable automorphism on each factor of $\check G(\chi_k)^0$. 
It follows that for $1\leq k\leq r$
\beqn
W_{\fa,\chi_k}^{en}=\langle s_1,\ldots,s_{k-1},s_{k-1,k}^{(1)},\tau_k^2\rangle\times\langle s_{k+1},\ldots,s_{r-1},\tau_r\rangle\cong G(m,2,k)\times G(m,1,r-k).
\eeqn
We have that $\theta$ restricts to an order $2l$ stable automorphism on $(G_{\tau_i})_{\on{der}}\cong Sp(2l)$ and 
\beqn
I_{s_{ij}^{(a)}}=\langle\gamma_i\gamma_{j}\rangle\text{ (resp. $\langle\prod_{k=i}^{j-1}\gamma_k\rangle$)}\cong\mu_2,\,\ I_{\tau_i}=\langle\gamma_i\rangle\text{ (resp. $\langle\prod_{k=i}^r\gamma_i\rangle$)}\cong\mu_2,\text{ if $l$ is even (resp. odd)}.
\eeqn
Thus
\beq\label{charsC}
\chi_k|_{I_{s_{ij}^{(a)}}}=1\Leftrightarrow j\leq k\text{ or }i\geq k+1\Leftrightarrow s_{ij}^{(a)}\in W_{\fa,\chi_k}^0,\ \ \chi_k|_{I_{\tau_i}}=1\Leftrightarrow i\geq k+1.
\eeq
 In view of~\eqref{charsC} the claim on $\cH_{W_{\fa,\chi_k}^0}$ follows from equations~\eqref{poly-order 2},~\eqref{wachi-c},~\eqref{type C-trivial} and~\eqref{type C-nontrivial}.
\end{proof}

\subsection{Type D}\label{ssec-pfD}

Assume that $G=Spin(2n)$. Let $\theta$ be a stable automorphism of order $m=2l$, where $l|n$ or $l|n-1$, see Table~\ref{table 1}. Let $r=n/l$ (resp. $r=(n-1)/l$).

Recall $W_n$ denote the Weyl group of type $B_n$  and the permutations $t_{ij}$, $t_a$ (see subsection~\ref{sec-typeB}). Let $W_n'\subset W_n$ be the subgroup of type $D_n$ consisting of permutations $\sigma$ such that $|\{i>0\mid\sigma(i)<0\}|$ is even.

\subsubsection{The case of $n=rl$}

Suppose that $n=rl$.  Let
\beqn
k_i=(i-1)l,\text{ and }\tilde\tau_i=t_{k_i+1,k_i+2}t_{k_i+2,k_i+3}\cdots t_{k_i+l-1,k_i+l}t_{k_i+l}\in W_n,\ 1\leq i\leq r.
\eeqn
 We can assume that (see~\cite[section 7]{RLYG}) $$\theta|_{\Lt=\on{Lie}(\Ts)}=w_0=\tilde\tau_1\cdots \tilde\tau_r.$$
Note that $w_0\in W_n'$ if and only if $r$ is even.  The little Weyl group $$W_\fa=\langle s_1,\ldots,s_{r-1},s_{r-1}',\tau_r\rangle\cong G(m,2,r)$$
 where 
\beq\label{type D sj}
s_{j}=\prod_{a=1}^lt_{k_j+a,k_{j+1}+a}\,,\ 1\leq j\leq r-1,\ s_{r-1}'=\tilde{\tau}_{r-1}^{-1}s_{r-1}\tilde{\tau}_{r-1}
,\ \ \tau_r=\tilde\tau_r^2.
\eeq

\begin{lemma}\label{lemma-type D-1}
{\rm (i)} We have
$$I=\langle\gamma_i, i=1,\ldots,r\rangle\cong \mu_2^{r}$$
where
\beqn
\bega
\gamma_k=\prod_{j=1}^{l}\check\alpha_{(k-1)l+2j-1}(-1),\ 1\leq k\leq r-1,\ 
\gamma_r=\check\alpha_{rl-1}(-1)\check\alpha_{rl}(-1).
\eega
\eeqn
{\rm (ii)} A set of representatives of $W_\fa$-orbits in $\hat I$ is
\bern
&&\{\chi_k,\ 0\leq k\leq \frac{r-1}{2}, \chi_r\},\text{ if $r$ is odd};\ \ \{\chi_k,\ 0\leq k\leq \frac{r}{2},\ \chi_{r},\ \chi_{r+1}\},\text{ if $r$ is even},
\eern
 defined by
\begin{eqnarray*}
&&\chi_k:\gamma_k\mapsto-1,\,\gamma_i\mapsto 1,\,i\neq k,\,0\leq k\leq[\frac{r-1}{2}]\text{ or }k=r;\\
&&\chi_{r+1}:\gamma_{r-1}\mapsto-1,\,\gamma_{r}\mapsto-1,\,\gamma_i\mapsto 1,\,i\neq r-1,r,\text{ if $r$ is even}.
\end{eqnarray*}
{\rm (iii)} We have
\bern
&&W_{\fa,\chi_k}^0\cong G(m,2,k)\times G(m,2,r-k),\\&&\cH_{W_{\fa,\chi_k}^0}\cong \cH^{l}(G(m,2,k))\times \cH^{l}(G(m,2,r-k)),\ 0\leq k\leq\frac{r}{2};\\
&&W_{\fa,\chi_r}^0\cong G(l,1,r),\ W_{\fa,\chi_{r+1}}^0\cong G(l,1,r)\text{ (when $r$ is even)},\\
&&\cH_{W_{\fa,\chi_r}^0}\cong\cH_{W_{\fa,\chi_{r+1}}^0}\cong
\cH^{[\frac{l+1}{2}],[\frac{l}{2}]} G(l,1,r).
\eern
Moreover
\bern
&&W_{\fa,\chi_k}^{en}=W_{\fa,\chi_k}^0\text{ if $k\neq r,r+1$ or $l$ is odd},\\
&& W_{\fa,\chi_k}^{0}/W_{\fa,\chi_k}^{en}\cong\bZ/2\bZ\text{ if $k= r,r+1$ and $l$ is even};\\
&&W_{\fa,\chi_k}/W_{\fa,\chi_k}^0\cong\bZ/2\bZ,\ 1\leq k<\frac{r}{2},\ \ W_{\fa,\chi_{r/2}}/W_{\fa,\chi_{r/2}}^0\cong\bZ/2\bZ\times\bZ/2\bZ,\\&&W_{\fa,\chi_r}\cong W_{\fa,\chi_r}^0\text{ if $r$ is odd};\ W_{\fa,\chi_k}/W_{\fa,\chi_k}^0\cong\bZ/2\bZ, \,k=r,r+1,\text{ if $r$ is even}.
\eern

\end{lemma}

\begin{proof}

We have
\begin{eqnarray*}
&&  s_{ij}^{(k)}:\gamma_{i-1}\mapsto\prod_{a=i-1}^{j-1}\gamma_{a}\text{ (resp. $\gamma_r\prod_{a=i-1}^{j-1}\gamma_{a}$)},\ \gamma_{i}\mapsto\prod_{a=i+1}^{j-1}\gamma_{a} \text{ (resp. $\gamma_r\prod_{a=i+1}^{j-1}\gamma_{a})$ (if $j\geq i+2$)},\\&& \gamma_{j-1}\mapsto\prod_{a=i}^{j-2}\gamma_{a} \text{ (resp. $\gamma_r\prod_{a=i}^{j-2}\gamma_{a}$) (if $j\geq i+2$)},\  \gamma_{j}\mapsto\begin{cases}\prod_{a=i}^{j}\gamma_{a}\text{ (resp. $\gamma_r\prod_{a=i}^{j}\gamma_{a}$)}&\text{ if $j\neq r$}\\\gamma_j&\text{ if $j=r$}\end{cases},\\&&\gamma_{a}\mapsto\gamma_a,\,a\neq i-1,i,j-1,j,\,\text{ if $k$ is even (resp. odd)};\\
&&  \tau_i:\gamma_a\mapsto\gamma_a\text{ for all }a.
\end{eqnarray*}
 It follows that
\begin{eqnarray*}
W_{\fa,\chi_k}^0&=&\langle s_{ij}^{(a)},\,i\geq k+1\text{ or }j\leq k,\ \tau_i\rangle=\langle s_1,\ldots,s_{k-1},s_{k-1,k}^{(1)},\tau_k\rangle\times\langle s_{k+1},\ldots, s_{r-1,r}^{(1)},\tau_r\rangle\\
&\cong& G(m,2,k)\times G(m,2,r-k),\ 0\leq k\leq \frac{r}{2};
\\
W_{\fa,\chi_r}^0&=&\langle s_{ij}^{(a)},\text{ $a$ even},\ \tau_i\rangle= \langle s_1,\ldots, s_{r-1},\tau_r\rangle\cong G(l,1,r);\\
W_{\fa,\chi_{r+1}}^0&=&\langle s_{ij}^{(a)},\text{ $j\leq r-1$ and $a$ even},\, s_{ir}^{(b)},\text{  $b$ odd},\,\tau_i\rangle= \langle s_1,\ldots,s_{r-2}, s_{r-1,r}^{(1)},\tau_r\rangle\cong G(l,1,r).
\end{eqnarray*}
Let $\sigma_0=\prod_{j=1}^{r/2}s_{j,r+1-j}^{(0)}$ and  $\sigma_k=\tilde\tau_k\tilde\tau_{k+1}\in W_n'$, $1\leq k\leq r/2$. Then one checks that 
\beqn
W_{\fa,\chi_k}=\langle\sigma_k\rangle\ltimes W_{\fa,\chi_k}^0,\,1\leq k<r/2,\,W_{\fa,\chi_{r/2}}=\langle\sigma_{r/2},\sigma_0\rangle\ltimes W_{\fa,\chi_{r/2}}^0.
\eeqn
Suppose $r$ is even. Let $\sigma_0': \epsilon_{jl+i}\mapsto \epsilon_{(r-1-j)l+i+1},\ \epsilon_{jl+l}\mapsto -\epsilon_{(r-1-j)l+1},\ 0\leq j\leq r-1, 1\leq i\leq l-1$, and $\sigma_0''=\tilde\tau_r^{-1}\sigma_0'\tilde\tau_r$. Then $W_{\fa,\chi_r}=\langle\sigma_0'\rangle\ltimes W_{\fa,\chi_r}^0$ and $W_{\fa,\chi_{r+1}}=\langle\sigma_0''\rangle\ltimes W_{\fa,\chi_{r+1}}^0$.

From the endoscopic point of view, $\chi_k$ corresponds to  
\bern
&&\chi_k=\on{diag}(-I_{kl},I_{(r-k)l},-I_{(r-k)l},I_{kl})\in PSO(2n),\ 0\leq k\leq r/2\\
&&\chi_{r}=\on{diag}(\underbrace{J_l,\ldots,J_l}_{r\text{ terms}},-J_l,\ldots,- J_l)\in PSO(2n),\\
&&\chi_{r+1}=\on{diag}(\underbrace{J_l,\ldots,J_l}_{r-1\text{ terms}},-J_l,J_l,-J_l\ldots,- J_l)\in PSO(2n),\ \text{if $r$ is even}
\eern
where  $J_l$ is defined in~\eqref{the matrix Jk}. Thus 
\bern
&&\check G(\chi_k)^0\cong\begin{cases}P\big(SO(2kl)\times SO(2(r-k)l)\big)&\ 0\leq k\leq r/2
\\
P(SO\big(rl)\times SO(rl)\big)&\ k=r, r+1,\, l\text{ even}\\
P(GL(rl)\times GL(rl)^*)&\ k=r, r+1,\,l\text{ odd},\end{cases}\\
&&\check G(\chi_k)/\check G(\chi_k)^0\cong\begin{cases}1&k=0, \text{ or }k=r\text{ and $l, r$ both odd}\\\bZ/2\bZ&1\leq k<r/2,\text{ or }k=r,r+1,\text{$l$ odd, $r$ even}\\\bZ/2\bZ\times\bZ/2\bZ&k=r/2,\text{ or }k=r,r+1\text{ and $l$ is even}.\end{cases}
\eern
When $0\leq k\leq r/2$, $\check\theta$ restricts to an order $m$ stable automorphism of each factor of ${\check G(\chi_k)^0}$. Thus $W_{\fa,\chi_k}^{en}\cong G(m,2,r/2)\times G(m,2,r/2)$.

\noindent When $l$ is even and $k=r$ or $r+1$,  $\check\theta|_{\check G(\chi_k)^0}$ can be identified with an order $l$  stable automorphism of type $D_{rl/2}$ (if $r$ is even) or $D_{rl/2}^2$ (if $r$ is odd). Thus $W_{\fa,\chi_k}^{en}\cong G(l,2,r)$.

\noindent When $l$ is odd,  and $k=r,r+1$,  $\check\theta|_{\check G(\chi_k)^0}$ is an order $2l$ stable (outer) automorphism of  type $A_{rl-1}^2$. Thus $W_{\fa,\chi}^{en}\cong G(l,1,r)$.

We have $\theta$ restricts to an order $2l$ stable (outer) automorphism on $(G_{\tau_i})_{\on{der}}\cong Spin(2l)$ and that 
 \beqn
 I_{s_{ij}^{(k)}}=\langle\prod_{a=i}^{j-1}\gamma_a\rangle\text{ (resp. $\langle\gamma_r\prod_{a=i}^{j-1}\gamma_a\rangle$)}\cong\mu_2\text{ if $k$ is even (resp. odd)},\ \ I_{\tau_i}=\langle\gamma_r\rangle\cong\mu_2.
 \eeqn
Moreover
\begin{subequations}
\beq\label{charsD1}
\chi_k|_{I_{s_{ij}^{(a)}}}=1\Leftrightarrow i\geq k+1\text{ or }j\leq k\Leftrightarrow s_{ij}^{(a)}\in W_{\fa,\chi_k}^0,\ \ \chi_k|_{I_{\tau_i}}=1,\text{ $0\leq k\leq[\frac{r-1}{2}]$};
\eeq
\beq\label{charsD2}
\bega
\chi_r|_{I_{s_{ij}^{(a)}}}=1\Leftrightarrow a\text{ even}\Leftrightarrow s_{ij}^{(a)}\in W_{\fa,\chi_r}^0,\ \ \chi_r|_{I_{\tau_i}}\neq 1;\\
\chi_{r+1}|_{I_{s_{ij}^{(a)}}}=1\Leftrightarrow j\leq r-1 \text{ and $a$ even, or $j=r$ and $a$ odd}\Leftrightarrow s_{ij}^{(a)}\in W_{\fa,\chi_{r+1}}^0,\\
\chi_{r+1}|_{I_{\tau_i}}\neq 1,\text{ when $r$ is even}.
\eega
\eeq
\end{subequations}
  In view of~\eqref{charsD1}, the claim on $\cH_{W_{\fa,\chi_k}^0}$, $0\leq k\leq[\frac{r-1}{2}]$, follows from equations~\eqref{poly-order 2} and~\eqref{type D2-trivial}.
In view of~\eqref{charsD2}, the claim on $\cH_{W_{\fa,\chi_r}^0}$ and $\cH_{W_{\fa,\chi_{r+1}}^0}$ follows from equations~\eqref{poly-order 2} and~\eqref{type D2-nontrivial}.
\end{proof}

\subsubsection{The case of $n=rl+1$}Suppose that $n=rl+1$.  For $1\leq i\leq r$, let
\beqn\label{taui}
\tau_i=t_{k_i+1,k_i+2}t_{k_i+2,k_i+3}\cdots t_{k_i+l-1,k_i+l}t_{k_i+l}t_{rl+1}\in W_n',
\eeqn
where $k_i=(i-1)l$. We can assume that (see~\cite[section 7]{RLYG}) $$\theta|_{\Lt=\on{Lie}(\Ts)}=w_0=\tau_1\cdots \tau_r\text{ (resp. $\tau_1\cdots \tau_rt_{rl+1}$)}\text{ if $r$ is odd (resp. even)}.$$
Note that $w_0\in W_n'$ if and only if $r$ is odd. The little Weyl group $$W_\fa=\langle s_1,\ldots,s_{r-1},\tau_r\rangle\cong G(m,1,r)$$ where $s_k,\,1\leq k\leq r-1$, is the same as in~\eqref{type D sj}.

\begin{lemma}\label{lemma-type D-2}
{\rm (i)} We have
\begin{eqnarray*}
&&I=\langle\gamma_i,\, i=1,\ldots,r\rangle\cong\mu_4\times\mu_2^{r-1}\text{ if $l$ is even}\\
&&I=\langle\gamma_i, i=1,\ldots,r+1\rangle\cong \mu_2^{r+1}\text{ if $l$ is odd},
\end{eqnarray*}
where
\begin{eqnarray*}
&&\gamma_k=\prod_{j=1}^{l}\check\alpha_{(k-1)l+2j-1}(-1),\ 1\leq k\leq r-1,\\
&&\gamma_r=\check\alpha_{rl}(\mathbf{i})\check\alpha_{rl+1}(-\mathbf{i})\prod_{j=1}^{\frac{l}{2}}\check\alpha_{(r-1)l+2j-1}(-1),\text{ if $l$ is even},\\
&&\gamma_r=\prod_{j=1}^{\frac{l+1}{2}}\check\alpha_{(r-1)l+2j-1}(-1),\ \ \gamma_{r+1}=\check\alpha_{rl}(-1)\check\alpha_{rl+1}(-1),\text{ if $l$ is odd}.
\end{eqnarray*}
{\rm (ii)} A set of representatives of $W_\fa$-orbits in $\hat I$ is 
$$\{\chi_k,\ 0\leq k\leq r+1\}\text{ if $r$ is even},\ \ \{\chi_k,0\leq k\leq r+2\}\text{ if $r$ is odd}$$
defined by
\begin{eqnarray*}
&&\chi_k:\gamma_k\mapsto-1,\ \gamma_i\mapsto1,\ i\neq k,\ 0\leq k\leq r,\\
&& \chi_{r+1}:\gamma_r\mapsto\mathbf{i},\,\gamma_i\mapsto1,\,i\neq r,\,\chi_{r+2}:\gamma_r\mapsto-\mathbf{i},\,\gamma_i\mapsto 1,\,i\neq r, \text{ if $l$ is even};\\
&&\chi_{r+1}:\gamma_{r+1}\mapsto-1,\ \gamma_i\mapsto1,\ i\neq r+1, \\&&  \chi_{r+2}:\gamma_r\mapsto-1,\ \gamma_{r+1}\mapsto-1,\ \gamma_i\mapsto1,\ i\neq r,r+1,\,\text{  if $l$ is odd}.
\end{eqnarray*}
{\rm (iii)} We have
\bern
&&W_{\fa,\chi_k}^0\cong G(m,1,k)\times G(m,1,r-k),\\
&&\cH_{W_{\fa,\chi_k}^0}\cong \cH^{l,l}(G(m,1,k))\otimes\cH^{l+2,l-2} (G(m,1,r-k)),\ 0\leq k\leq r;
\\
&&W_{\fa,\chi_{r+1}}^0\cong W_{\fa,\chi_{r+2}}^0\cong G(l,1,r)\\
&&\cH_{W_{\fa,\chi_{r+1}}^0}\cong\cH_{W_{\fa,\chi_{r+2}}^0}\cong\cH^{\frac{l}{2},\frac{l}{2}}{(G(l,1,r))}\text{ (resp. $\cH^{\frac{l+3}{2},\frac{l-3}{2}} G(l,1,r)$)}\\&&\hspace{2in}\text{ if $l$ is even (resp. odd)}.
\eern
Moreover
\bern
&&W_{\fa,\chi_k}^{en}=W_{\fa,\chi_k}^0\text{ if $k=0$, or $k=r+1,r+2$, $l$ odd},\\
&&W_{\fa,\chi_k}^0/W_{\fa,\chi_k}^{en}\cong\bZ/2\bZ\text{ if $1\leq k\leq r$, or if $k=r+1,r+2$, $l$ even};\\
&&W_{\fa,\chi_k}=W_{\fa,\chi_k}^0,\ 0\leq k\leq r\text{ or $k=r+1$ and $r$ even};\\&&
 W_{\fa,\chi_{k}}/ W_{\fa,\chi_{k}}^0\cong\bZ/2\bZ,\,k=r+1,r+2,\text{ and $r$ is odd}.
\eern

\end{lemma}

\begin{proof}
Let us write
\beqn
\gamma=\gamma_r^2\text{ (resp. $\gamma_{r+1}$)}\text{ if $l$ is even (resp. odd)}=\check\alpha_{rl}(-1)\check\alpha_{rl+1}(-1).
\eeqn
We have 
\bern
&& s_{ij}^{(a)}(\gamma_k)=\gamma_k\text{ if $i,j\neq k,k+1$},\ s_{ik}^{(a)}(\gamma_k)=\gamma_i\cdots\gamma_k\text{ (resp. $\gamma_i\cdots\gamma_k\gamma$)},\\&&\,s_{k,k+1}^{(a)}(\gamma_k)=\gamma_{k},\ s_{i,k+1}^{(a)}(\gamma_k)=\gamma_i\cdots\gamma_{k-1}\text{ (resp. $\gamma_i\cdots\gamma_{k-1}$)},\text{ if $i\leq k-1$},\\&& s_{kj}^{(a)}(\gamma_k)=\gamma_{k+1}\cdots\gamma_{j-1}\text{ (resp. $\gamma_{k+1}\cdots\gamma_{j-1}\gamma$)},\text{ if $j\geq k+2$},\\&& s_{k+1,j}^{(a)}(\gamma_k)=\gamma_k\cdots\gamma_{j-1}\text{ (resp. $\gamma_k\cdots\gamma_{j-1}\gamma$)},\text{ when $a$ is even (resp. odd)};
\\
&&\tau_i({\gamma_k})=\begin{cases}\gamma_k&\text{if }\,k\notin\{ i-1,i,r\}, \text{ $l$ even},  \text{ or }\,k\notin\{ i-1,i\}, \text{ $l$ odd}\\ \gamma_k\gamma&\text{if }\,k\in\{ i-1,i,r\}, \text{ $l$ even},  \text{ or }\,k\in\{ i-1,i\}, \text{ $l$ odd}\end{cases},\,\,i\leq r-1,
\\&&\tau_r({\gamma_k})=\gamma_k\text{ (resp. $\gamma_k\gamma$)},\text{ if $k\neq r-1$ (resp. $k=r-1$)}.
\eern
Thus 
\bern
&&s_{ij}^{(a)}.\chi_k=\chi_k\Leftrightarrow j\leq k\text{ or }i\geq k+1,\ \ \tau_i.\chi_k=\chi_k,\text{ for $0\leq k\leq r$};
\\
&&s_{ij}^{(a)}.\chi_k=\chi_k\Leftrightarrow \text{$a$ is even},\ \tau_i^a.\chi_k=\chi_k\Leftrightarrow \text{$a$ is even},\text{ for $k=r+1,r+2$}.
\eern
It follows that
\ber
&&W_{\fa,\chi_k}^0=\langle s_1,\ldots,s_{k-1},\tau_k\rangle\times\langle s_{k+1},\ldots,s_{r-1},\tau_r\rangle\label{wachi-d12}\\
&&\qquad\quad\cong G(m,1,k)\times G(m,1,r-k),\ 0\leq k\leq r\nonumber;
\\
&&W_{\fa,\chi_{r+1}}^0=\langle s_1,\ldots,s_{r-1},\tau_{r}^2\rangle\cong G(l,1,r)\label{wachi-D11}
\\
&&W_{\fa,\chi_{r+2}}^0=\langle s_1,\ldots,s_{r-1},\tau_{r}^2\rangle\cong G(l,1,r)\text{ if $r$ is odd}\label{wachi-D13}.
\eer
We have $W_{\fa,\chi_k}=\langle\sigma_0'\rangle\ltimes W_{\fa,\chi_{k}}^0$, $k=r+1,r+2$, where
$$\sigma_0': \epsilon_{jl+i}\mapsto \epsilon_{(r-1-j)l+i+1},\ \epsilon_{jl+l}\mapsto -\epsilon_{(r-1-j)l+1},\ 0\leq j\leq r-1, 1\leq i\leq l-1,\epsilon_{rl+1}\mapsto-\epsilon_{rl+1}.$$

From the endoscopic point of view, the $\chi_k$ corresponds to 
\bern
&&\chi_k=\on{diag}(-I_{kl},I_{(r-k)l},1,1,I_{(r-k)l},-I_{kl})\in\check G=PSO(2n),\ 0\leq k\leq r\\
&&\chi_{r+1}=\on{diag}(\underbrace{J_l,\ldots,J_l}_{r\text{ terms}},\mathbf{i},-\mathbf{i},\underbrace{-J_l,\ldots,-J_l}_{r\text{ terms}})\in\check G=PSO(2n),\\
&&\chi_{r+2}=\on{diag}(\underbrace{J_l,\ldots,J_l}_{r\text{ terms}},-\mathbf{i},\mathbf{i},\underbrace{-J_l,\ldots,-J_l}_{r\text{ terms}})\in\check G=PSO(2n),\ \text{ when $r$ is odd}
\eern
where $J_l$ is defined in~\eqref{the matrix Jk}. 
Thus 
\bern
&&\check G(\chi_k)^0\cong\begin{cases}P\big(SO(2kl)\times SO(2(r-k)l+2)\big)&\ 0\leq k\leq r
\\
P(GL(1)\times GL(1)^*\times SO(rl)\times SO(rl))\text{ if $l$ is even}& k=r+1, r+2\\
P\big(GL(rl+1)\times GL(rl+1)^*\big)\text{ if $l$ is odd}& k=r+1, r+2\,.\end{cases}\\
&&\check G(\chi_k)/\check G(\chi_k)^0\cong\begin{cases}1&k=0,\text{ or $k=r+1$, $l$ odd, $r$ even}\\
\bZ/2\bZ&1\leq k\leq r,\text{ or $k=r+1,r+2$, $l$ odd, $r$ odd}\\
\bZ/4\bZ&\text{$k=r+1,r+2$, $l$ even}
\end{cases}
\eern
\noindent When $0\leq k\leq r$, $\check\theta$ restricts to an order $m$ stable automorphism of each factor of ${\check G(\chi_k)^0}$. Thus $W_{\fa,\chi_k}^{en}=\langle s_1,\ldots,s_{k-1},s_{k-1,k}^{(1)},\tau_k^2\rangle\times\langle s_{k+1},\ldots,s_{r-1},\tau_r\rangle\cong G(m,2,k)\times G(m,1,r-k).$

\noindent When $l$ is even and $k=r+1,r+2$,  $\check\theta|_{\check G(\chi_k)^0}$ can be identified with an order $l$ stable automorphism of type $D_{rl/2}$ or $D_{rl/2}^2$. Thus $W_{\fa,\chi_k}^{en}=\langle s_{ij}^{(2a)},\tau_{i}^4\rangle\cong G(l,2,r)$.

\noindent When $l$ is odd,  and $k=r+1,r+2$,  $\check\theta|_{\check G(\chi_k)^0}$ is an order $2l$ stable (outer) automorphism of type $A_{rl}^2$. Thus $W_{\fa,\chi_k}^{en}=\langle s_1,\ldots,s_{r-1},\tau_{r}^2\rangle\cong G(l,1,r)$.

We have that $\theta$ restricts to an order $2l$ stable (inner) automorphism on $(G_{\tau_i})_{\on{der}}\cong Spin(2l+2)$ and that 
\begin{eqnarray*}
&&I_{s_{ij}^{(k)}}\cong\langle\prod_{a=i}^{j-1}\gamma_a\rangle\text{ (resp. $\cong\langle\gamma\prod_{a=i}^{j-1}\gamma_a\rangle$) }\cong\mu_2,\text{ if $k$ is even (resp. odd)}; \\ 
&&I_{\tau_k}=\langle\gamma_k\cdots\gamma_r\rangle\cong\mu_4,\text{ if $l$ is even};\ I_{\tau_k}=\langle\gamma_k\cdots\gamma_r,\,\gamma_{r+1}\rangle\cong\mu_2\times\mu_2,\text{ if $l$ is odd}.
\end{eqnarray*}
It follows that
$\chi_k|_{I_{s_{ij}^{(a)}}}=1\Leftrightarrow s_{ij}^{(a)}\in W_{\fa,\chi_k}^0$.

 For $0\leq k\leq r$, since $\chi_k|_{I_{\tau_i}}= 1\Leftrightarrow i\geq k+1$, $\chi_k|_{I_{\tau_i}}$, $i\leq k$, is of order 2 when $l$ is even, and $\chi_k(\gamma_{r+1})=1$ when $l$ is odd, the claim on $\cH_{W_{\fa,\chi_k}^0}$ follows from equations~\eqref{poly-order 2},~\eqref{wachi-d12},~\eqref{type D-odd-trivial}, \eqref{type D-odd-order 2}, \eqref{typeD-even-1} and~\eqref{typeD-even-2}.

Suppose $l$ is even. Since $\chi_{r+1}|_{I_{\tau_i}}$ and $\chi_{r+2}|_{I_{\tau_i}}$ are of order $4$, the claim on $\cH_{W_{\fa,\chi_{r+1}}^0}$ and $\cH_{W_{\fa,\chi_{r+2}}^0}$ follows from equations~\eqref{poly-order 2},~\eqref{wachi-D11},~\eqref{wachi-D13} and~\eqref{type D-odd-order 4}.

Suppose $l$ is odd. Since $\chi_{r+1}(\gamma_{r+1})=\chi_{r+2}(\gamma_{r+1})=-1$, the claim on $\cH_{W_{\fa,\chi_{r+1}}^0}$ and $\cH_{W_{\fa,\chi_{r+2}}^0}$ follows from equations~\eqref{poly-order 2},~\eqref{wachi-D11},~\eqref{wachi-D13} and~\eqref{typeD-even-3}.
\end{proof}

\end{document}